\documentclass[reqno]{amsart}
\usepackage{amsmath}
\usepackage{amsthm}
\usepackage{graphicx}
\usepackage{amsfonts}
\usepackage{amssymb}
\usepackage{enumerate}
\numberwithin{equation}{section}

\newtheorem{theorem}{Theorem}[section]
\newtheorem{lemma}[theorem]{Lemma}
\newtheorem{proposition}[theorem]{Proposition}
\newtheorem{corollary}[theorem]{Corollary}

\theoremstyle{definition}

\newtheorem{definition}[theorem]{Definition}
\newtheorem{remark}[theorem]{Remark}

\newtheorem{innercustomthm}{Assumption}
\newenvironment{assumption}[1]
  {\renewcommand\theinnercustomthm{#1}\innercustomthm}
  {\endinnercustomthm}

\def\E{{\mathbb E}}
\def\R{{\mathbb R}}
\def\N{{\mathbb N}}
\def\PP{{\mathbb P}}
\def\P{{\mathcal P}}
\def\RC{{\mathcal R}}
\def\B{{\mathcal B}}
\def\V{{\mathcal V}}
\def\M{{\mathcal M}}

\def\X{{\mathcal X}}

\def\Q{{\mathcal Q}}

\def\G{{\mathcal G}}

\def\W{{\mathcal W}}

\def\A{{\mathcal A}}
\def\F{{\mathcal F}}

\def\C{{\mathcal C}}

\title[The mean field limit for stochastic differential games]{A general characterization of the mean field limit for stochastic differential games}
\author{Daniel Lacker}
\address{ORFE,  Princeton University,
Princeton, NJ  08544, USA.}
\email{dlacker@princeton.edu}
\thanks{Partially supported by NSF: DMS-0806591}

\begin{document}

\begin{abstract}
The mean field limit of large-population symmetric stochastic differential games is derived in a general setting, with and without common noise, on a finite time horizon. Minimal assumptions are imposed on equilibrium strategies, which may be asymmetric and based on full information. It is shown that approximate Nash equilibria in the $n$-player games admit certain weak limits as $n$ tends to infinity, and every limit is a weak solution of the mean field game (MFG). Conversely, every weak MFG solution can be obtained as the limit of a sequence of approximate Nash equilibria in the $n$-player games. Thus, the MFG precisely characterizes the possible limiting equilibrium behavior of the $n$-player games. 
Even in the setting without common noise, the empirical state distributions may admit stochastic limits which cannot be described by the usual notion of MFG solution.
\end{abstract}

\maketitle


\section{Introduction}
A decade of active research on mean field games (MFGs) has been driven by a primarily intuitive connection with large-population stochastic differential games of a certain symmetric type.
The idea, which began with the pioneering work of Lasry and Lions \cite{lasrylionsmfg} and Huang, Malham\'e and Caines \cite{huangmfg1}, is that
a large-population game of this type should behave similarly to its MFG counterpart, which may be thought of as an infinite-player version of the game. 
Rigorous analysis of this connection, however, remains restricted in scope. Following \cite{huangmfg1}, the vast majority of the literature works backward from the mean field limit, in the sense that a solution of the MFG is used to construct approximate Nash equilibria for the corresponding $n$-player games for large $n$. Fewer papers \cite{lasrylionsmfg,fischer-mfgconnection,bardi-explicitsolutions,feleqi-MFGderivation} have approached from the other direction: given for each $n$ a Nash equilibrium for the $n$-player game, in what sense (if any) do these equilibria converge as $n$ tends to infinity?
The goal of this paper is to address both of these problems in a general framework.

More precisely, we study an $n$-player stochastic differential game, in which the private state processes $X^1,\ldots,X^n$ of the agents (or players) are given by the following dynamics:
\begin{align*}
dX^{i}_t &= b(t,X^{i}_t,\widehat{\mu}^n_t,\alpha^i_t)dt + \sigma(t,X^{i}_t,\widehat{\mu}^n_t)dW^i_t + \sigma_0(t,X^{i}_t,\widehat{\mu}^n_t)dB_t, \\
\widehat{\mu}^n_t &= \frac{1}{n}\sum_{k=1}^n\delta_{X^{k}_t}.
\end{align*}
Here $B,W^1,\ldots,W^n$ are independent Wiener processes, $\alpha^i$ is the control of agent $i$, and $\widehat{\mu}^n$ is the empirical distribution of the state processes. We call $W^1,\ldots,W^n$ the \emph{independent} or \emph{idiosyncratic noises}, since agent $i$ feels only $W^i$ directly, and we call $B$ the \emph{common noise}, since each agent feels $B$ equally. The reward to agent $i$ of the strategy profile $(\alpha^1,\ldots,\alpha^n)$ is
\[
J_i(\alpha^1,\ldots,\alpha^n) = \E\left[\int_0^Tf(t,X^{i}_t,\widehat{\mu}^n_t,\alpha^i_t)dt + g(X^{i}_T,\widehat{\mu}^n_T)\right].
\]
Agent $i$ seeks to maximize this reward, and so we say that $(\alpha^1,\ldots,\alpha^n)$ form an $\epsilon$-Nash equilibrium (or an approximate Nash equilibrium) if 
\[
J_i(\alpha^1,\ldots,\alpha^n) + \epsilon \ge J_i(\alpha^1,\ldots,\alpha^{i-1},\beta,\alpha^{i+1},\ldots,\alpha^n)
\]
for each admissible alternative strategy $\beta$. Intuitively, if the number of agents $n$ is very large, a single representative agent has little influence on the empirical measure flow $(\widehat{\mu}^n_t)_{t \in [0,T]}$, and so this agent expects to lose little in the way of optimality by ignoring her own effect on the empirical measure. Crucially, the system is symmetric in the sense that the same functions $(b,\sigma,\sigma_0)$ and $(f,g)$ determine the dynamics and objectives of each agent, and thus we may hope to learn something of the entire system from the behavior of a single representative agent.

The mean field game is specified precisely in Section \ref{se:statement}, and it follows this intuition by treating $n$ as infinite. Loosely speaking, a \emph{strong MFG solution} is a $(\F^B_t = \sigma(B_s : s\le t))_{t \in [0,T]}$-adapted measure-valued process $(\mu_t)_{t \in [0,T]}$ satisfying $\mu_t = \text{Law}(X^{\alpha^*}_t \ | \ \F^B_t)$ for each $t$, where $X^{\alpha^*}$ is an optimally controlled state process coming from the following stochastic optimal control problem:
\[
\begin{cases}
\alpha^* &\in \arg\max_\alpha\E\left[\int_0^Tf(t,X^\alpha_t,\mu_t,\alpha_t)dt + g(X^\alpha_T,\mu_T)\right], \text{ s.t.} \\
dX^\alpha_t &= b(t,X^\alpha_t,\mu_t,\alpha_t)dt + \sigma(t,X^\alpha_t,\mu_t)dW_t + \sigma_0(t,X^\alpha_t,\mu_t)dB_t.
\end{cases}
\]
In other words, with the process $(\mu_t)_{t \in [0,T]}$ treated as fixed, the representative agent solves an optimal control problem. The requirement $\mu_t = \text{Law}(X^{\alpha^*}_t \ | \ \F^B_t)$, often known as a \emph{consistency condition}, assures us that this decoupled optimal control problem is truly \emph{representative} of the entire population, and we may think of the measure flow $(\mu_t)_{t \in [0,T]}$ as an \emph{equilibrium}.

The analysis of this paper focuses on mean field games with common noise, but both of the volatility coefficients $\sigma$ and $\sigma_0$ are allowed to be degenerate. Hence, our results cover the usual mean field games without common noise (where $\sigma_0 \equiv 0$) as well as deterministic mean field games (where $\sigma \equiv \sigma_0 \equiv 0$).
The literature on mean field games with common noise is quite scarce so far, but some general analysis is provided in the recent papers \cite{carmonadelaruelacker-mfgcommonnoise,ahuja-mfgwellposedness,carmona:delarue:4lyons,bensoussan-masterequation}, and some specific models were studied in \cite{carmonafouque-systemicrisk,gueantlasrylionsmfg}. This paper can be seen as a sequel to \cite{carmonadelaruelacker-mfgcommonnoise}, from which we borrow many definitions and a handful of lemmas.
It is emphasized in \cite{carmonadelaruelacker-mfgcommonnoise} that strong solutions are quite difficult to obtain when common noise is present, and this leads to a notion of \emph{weak MFG solution}. Weak solutions, defined carefully in Section \ref{se:mfgdefinition}, differ most significantly from strong solutions in that the measure flow $(\mu_t)_{t \in [0,T]}$ need not be $(\F^B_t)_{t \in [0,T]}$-adapted, and the consistency condition is weakened to something like $\mu_t = \text{Law}(X^{\alpha^*}_t \ | \ \F^{B,\mu}_t)$, where $\F^{B,\mu}_t = \sigma(B_s,\mu_s : s \le t)$. Additionally, weak MFG solutions allow for relaxed (i.e. measure-valued) controls which need not be adapted to the filtration generated by the inputs $(X_0,B,W,\mu)$ of the control problem.

Although this weaker notion of MFG solution was introduced in \cite{carmonadelaruelacker-mfgcommonnoise} to develop an existence and uniqueness theory for MFGs with common noise, the main result of this paper is to assert that this notion is the right one from the point of view of the finite-player game, in the sense that weak MFG solutions characterize the limits of approximate Nash equilibria. The main results are stated in full generality in Sections \ref{se:mainresults} and \ref{se:converse}, but let us state them loosely for now in a simplified form: First, we show that if for each $n$ we are given an $\epsilon_n$-Nash equilibrium $(\alpha^{n,1},\ldots,\alpha^{n,n})$ for the $n$-player game, where $\epsilon_n \rightarrow 0$, then the family $(\text{Law}(B,\widehat{\mu}^n))_{n=1}^\infty$ is tight, and every weak limit agrees with the law of $(B,\mu)$ coming from some weak MFG solution. Second, we show conversely that every weak MFG solution can be obtained as a limit in this way.

Specializing our results to the case without common noise uncovers something unexpected. In the literature thus far, a MFG solution is defined in terms of a \emph{deterministic} equilibrium $(\mu_t)_{t \in [0,T]}$, corresponding to our notion of strong MFG solution. Even when there is no common noise, a weak MFG solution still involves a stochastic equilibrium, and because of our main theorems we must therefore expect the limits of the finite-player empirical measures to remain stochastic. 
Moreover, we demonstrate by a simple example that \emph{a stochastic equilibrium is not necessarily just a randomization among the family of deterministic equilibria}. 
Hence, \emph{the solution concept considered thusfar in literature on mean field games (without common noise) does not fully capture the limiting dynamics of finite-player approximate Nash equilibria}.
This is unlike the case of McKean-Vlasov limits (see \cite{oelschlagermkv,gartnermkv,sznitman}), which can be seen as mean field games with no control.
We prove some admittedly difficult-to-apply results which nevertheless shed some light on this phenomenon: The fundamental obstruction is the adaptedness required of controls, which renders the class of admissible controls quite sensitive to whether or not $(\mu_t)_{t \in [0,T]}$ is stochastic.

Our first theorem, regarding the convergence of arbitrary approximate equilibria (open-loop, full-information, and possibly asymmetric), is arguably the more novel of our two main theorems. It appears to be the first result of its kind for mean field games \emph{with common noise}, with the exception of the linear quadratic model of \cite{carmonafouque-systemicrisk} for which explicit computations are available. 
However, even in the setting without common noise we subtantially generalize the few existing results.

Several papers, such as the recent \cite{carmona:delarue:4lyons} dealing with common noise, contain purely heuristic derivations of the MFG as the limit of $n$-player games. The intuition guiding such derivations is as follows (and let us assume there is no common noise for the sake of simplicity): If $n$ is large, a single agent in a large population should lose little in the way of optimality if she ignores the small feedbacks arising through the empirical measure flow $(\widehat{\mu}^n_t)_{t \in [0,T]}$. If each of the $n$ identical agents does this, then we expect to see symmetric strategies which are nearly independent and ideally of the form $\hat{\alpha}(t,X^i_t)$, for some feedback control $\hat{\alpha}$ common to all of the agents. From the theory of McKean-Vlasov limits, we then expect that $(\widehat{\mu}^n_t)_{t \in [0,T]}$ converges to a deterministic limit. This intuition, however, is largely unsubstantiated and, we will argue, inaccurate in general. 

Lasry and Lions \cite{lasrylionsmfg,lasrylions-jeux1} first attacked this problem rigorously using PDE methods, working with an infinite time horizon and strong simplifying assumptions on the data, and their results were later generalized by Feleqi \cite{feleqi-MFGderivation}. Bardi and Priuli \cite{bardi-explicitsolutions,bardipriuli-ergodicLQMFG} justified the MFG limit for certain linear-quadratic problems, and Gomes et al. \cite{gomes-finitestatemfg} studied models with finite state space. Substantial progress was made in a very recent paper of Fischer \cite{fischer-mfgconnection}, which deserves special mention also because both the level of generality and the method of proof are quite similar to ours; we will return to this shortly.

With the exception of \cite{fischer-mfgconnection}, the aforementioned results share the important limitation that the agents have only \emph{partial information}: the control of agent $i$ may depend only on her own state process $X^{n,i}$ or Wiener process $W^i$. 
Our results allow for arbitrary full-information strategies, settling a conjecture of Lasry and Lions (stated in Remark x after \cite[Theorem 2.3]{lasrylionsmfg} for the case of infinite time horizon).
Combined in \cite{lasrylionsmfg,lasrylions-jeux1,feleqi-MFGderivation} with the assumption that the state process coefficients $(b,\sigma)$ do not depend on the empirical measure, the assumption of partial information leads to the immensely useful simplification that the state processes of the $n$-player games are independent. By showing then that they are also asymptotically identically distributed, 
the aforementioned heuristic argument can be made precise.

Fischer \cite{fischer-mfgconnection}, on the other hand, allows for full-information controls but characterizes  only the \emph{deterministic} limits of $(\widehat{\mu}^n_t)_{t \in [0,T]}$ as MFG equilibria. Assuming that the limit is deterministic implicitly restricts the class of $n$-player equilibria in question. By characterizing even the stochastic limits of $(\widehat{\mu}^n_t)_{t \in [0,T]}$, which we show are in fact quite typical, we impose no such restriction on the equilibrium strategies of the $n$-player games. This not to say, however, that our results completely subsume those of \cite{fischer-mfgconnection}, which work with a more general notion of \emph{local} approximate equilibria and which notably include conditions under which the assumption of a deterministic limit can be verified.

Our second main theorem, which asserts that every weak MFG solution is attainable as a limit of finite-player approximate Nash equilibria, is something of an abstraction of the kind of limiting result most commonly discussed in the MFG literature. In a tradition beginning with \cite{huangmfg1} and continued by the majority of the probabilistic papers on the subject \cite{carmonadelarue-mfg,carmonalacker-probabilisticweakformulation,bensoussan-lqmfg,bensoussan-mfgbook,kolokoltsov-mfgnonlinearmarkov}, an optimal control from an MFG solution is used to construct approximate equilibria for the finite-player games. 
Although our result applies in more general settings, our conclusions are duly weaker, in the sense that the approximate equilibria we construct do not necessarily consist of particularly tangible (i.e. distributed or even symmetric) strategies. We emphasize that the goal of this work is not to construct \emph{nice} approximate equilibria but rather to characterize all possible limits of approximate equilibria.

It is worth emphasizing that this paper makes no claims whatsoever regarding the existence or uniqueness of equilibria for either the $n$-player game or the MFG. Rather, we show that if a sequence of $n$-player approximate equilibria exists, then its limits are described by weak MFG solutions. Conversely, if a weak MFG solution exists, then it is achieved as the limit of some sequence of $n$-player approximate equilibria. Hence, existence of a weak MFG solution is equivalent to existence of a sequence of $n$-player approximate equilibria. Note, however, that the main assumption \ref{assumption:A} of this paper actually guarantees the existence of a weak MFG solution, because of the recent results of \cite{carmonadelaruelacker-mfgcommonnoise}. Far more results are available for MFGs without common noise; refer to the surveys \cite{cardaliaguet-mfgnotes,gomessaude-mfgsurvey} and the recent book \cite{bensoussan-mfgbook} for a wealth of wellposedness results and for further discussion of MFG theory in general.

The paper is organized as follows. Section \ref{se:statement} defines the MFG and the corresponding $n$-player games, before stating the main limit Theorem \ref{th:mainconvergence} and its converse, Theorem \ref{th:converseconvergence}, along with several useful corollaries. Section \ref{se:nocommonnoise} specializes the results to the more familiar setting without common noise and explains the gap between weak and strong solutions. Section \ref{se:canonicalspace} provides some background on weak solutions of MFGs with common noise, borrowed from \cite{carmonadelaruelacker-mfgcommonnoise}, before we turn to the proofs of the main results in Sections \ref{se:mainconvergenceproof}, \ref{se:converseconvergenceproof}, and \ref{se:withoutcommonnoiseproof}. Section \ref{se:mainconvergenceproof} is devoted to the proof of Theorem \ref{th:mainconvergence}, while Section \ref{se:converseconvergenceproof} contains the proof of the converse Theorem \ref{th:converseconvergence}. Finally, Section \ref{se:withoutcommonnoiseproof} explains how to carefully specialize these two theorems to the setting without common noise.

\section{The mean field limit with common noise} \label{se:statement}
After establishing some notation, this section first defines quickly and concisely the mean field game. We work with the same definitions and nearly the same assumptions as \cite{carmonadelaruelacker-mfgcommonnoise}, to which the reader is referred for a more thorough discussion. Then, the $n$-player game is formulated precisely, allowing for somewhat more general information structures than one usually finds in the literature on stochastic differential games. This generality is not just for its own sake; it will play a crucial role in the proofs later.

\subsection{Notation and standing assumptions}

For a topological space $E$, let $\B(E)$ denote the Borel $\sigma$-field, and let $\P(E)$ denote the set of Borel probability measures on $E$. 
For $p \ge 1$ and a separable metric space $(E,d)$, let $\P^p(E)$ denote the set of $\mu \in \P(E)$ satisfying $\int_Ed^p(x,x_0)\mu(dx) < \infty$ for some (and thus for any) $x_0 \in E$. Let $\ell_{E,p}$ denote the $p$-Wasserstein distance on $\P^p(E)$, given by
\begin{align}
\ell_{E,p}(\mu,\nu) := \inf\left\{\left(\int_{E \times E}\gamma(dx,dy)d^p(x,y)\right)^{1/p} : \gamma \in \P(E \times E) \text{ has marginals } \mu,\nu\right\} \label{def:wasserstein}
\end{align}
Unless otherwise stated, the space $\P^p(E)$ is equipped with the metric $\ell_{E,p}$, and all continuity and measurability statements involving $\P^p(E)$ are with respect to $\ell_{E,p}$ and the corresponding Borel $\sigma$-field. The analysis of the paper will make routine use of several topological properties of the spaces $\P^p(E)$ and $\P^p(\P^p(E))$, especially when $E$ is a product space. All of the results we need, well known or not, are summarized in the Appendices A and B of \cite{lacker-mfgcontrolledmartingaleproblems}.

We are given a time horizon $T > 0$, three exponents $(p',p,p_\sigma)$ with $p \ge 1$, a control space $A$, an initial state distribution $\lambda \in \P(\R^d)$, and the following functions:
\begin{align*}
(b,f) &: [0,T] \times \R^d \times \P^p(\R^d) \times A \rightarrow \R^d \times \R, \\ 
(\sigma,\sigma_0) &: [0,T] \times \R^d \times \P^p(\R^d) \rightarrow \R^{d \times m} \times \R^{d \times m_0}, \\ 
g &: \R^d \times \P^p(\R^d) \rightarrow \R.
\end{align*}
Assume \emph{throughout the paper} that the following assumption \ref{assumption:A} holds. This is exactly Assumption \textbf{A} of \cite{carmonadelaruelacker-mfgcommonnoise}, except that here we require that $p' \ge 2$ and that $(b,\sigma,\sigma_0)$ are Lipschitz not only in the state argument but also in the measure argument.

\begin{assumption}{\textbf{A}} \label{assumption:A}
{\ }
\begin{enumerate}
\item[(A.1)] $A$ is a closed subset of a Euclidean space. (More generally, as in \cite{haussmannlepeltier-existence}, a closed $\sigma$-compact subset of a Banach space would suffice.)
\item[(A.2)] The exponents satisfy $p' > p \ge 1 \vee p_\sigma$ and $p' \ge 2 \ge p_\sigma \ge 0$, and also $\lambda \in \P^{p'}(\R^d)$.
\item[(A.3)] The functions $b$, $\sigma$, $\sigma_0$, $f$, and $g$ of $(t,x,\mu,a)$ are jointly measurable and are continuous in $(x,\mu,a)$ for each $t$.
\item[(A.4)] There exists $c_1 > 0$ such that, for all $(t,x,y,\mu,\nu,a) \in [0,T] \times \R^d \times \R^d \times \P^p(\R^d) \times \P^p(\R^d) \times A$,
\begin{align*}
|b(t,x,\mu,a) - b(t,y,\nu,a)| &+ |(\sigma,\sigma_0)(t,x,\mu) - (\sigma,\sigma_0)(t,y,\nu)| \le c_1\left(|x-y| + \ell_{\R^d,p}(\mu,\nu)\right),
\end{align*}
and
\begin{align*}
|b(t,0,\delta_0,a)| &\le c_1(1 + |a|), \\
|(\sigma\sigma + \sigma_0\sigma_0^\top)(t,x,\mu)| &\le c_1\left[1 + |x|^{p_\sigma} + \left(\int_{\R^d}|z|^p\mu(dz)\right)^{p_\sigma/p}\right].
\end{align*}
\item[(A.5)] There exist $c_2, c_3 > 0$ such that, for each $(t,x,\mu,a) \in [0,T] \times \R^d \times \P^p(\R^d) \times A$, 
\begin{align*}
|g(x,\mu)| &\le c_2\left(1 + |x|^p + \int_{\R^d}|z|^p\mu(dz)\right), \\
-c_2\left(1 + |x|^p + \int_{\R^d}|z|^p\mu(dz) + |a|^{p'}\right) \le f(t,x,\mu,a) &\le c_2\left(1 + |x|^p + \int_{\R^d}|z|^p\mu(dz)\right) - c_3|a|^{p'}.
\end{align*}
\end{enumerate}
\end{assumption}

While these assumptions are fairly general, they do not cover all linear-quadratic models. Because of the requirement $p' > p$, the running objectve $f$ may grow quadratically in $a$ only if its growth in $(x,\mu)$ is strictly subquadratic. This requirement is important for compactness purposes, both for the results of this paper and for the existence results of \cite{lacker-mfgcontrolledmartingaleproblems,carmonadelaruelacker-mfgcommonnoise}. In fact, \cite{lacker-mfgcontrolledmartingaleproblems,carmonadelaruelacker-mfgcommonnoise} provide examples of MFGs with $p'=p$ which do not admit solutions even though they verify the rest of assumption \ref{assumption:A}. Existence results for this somewhat delicate boundary case have been obtained in \cite{carmonadelarue-mfg,carmonadelaruelachapelle-mkvvsmfg,bensoussan-lqmfg,carmonafouque-systemicrisk} by assuming some additional inequalities between coefficients. It seems feasible to expect our main results to adapt to such settings, but we do not pursue this here.

\subsection{Relaxed controls and mean field games} \label{se:mfgdefinition}
Define $\V$ to be the set of measures $q$ on $[0,T] \times A$ with first marginal equal to Lebesgue measure, i.e. $q([s,t] \times A) = t-s$ for $0 \le s \le t \le T$, satisfying also
\[
\int_{[0,T] \times A}|a|^pq(dt,da) < \infty.
\]
Since these measures have mass $T$, we may endow $\V$ with a suitable scaling of the $p$-Wasserstein metric.
Each $q \in \V$ may be identified with a measurable function $[0,T] \ni t \mapsto q_t \in \P^p(A)$, determined uniquely (up to a.e. equality) by $dtq_t(da) = q(dt,da)$. 
It is known that $\V$ is a Polish space, and in fact if $A$ is compact then so is $\V$; see \cite[Appendix A]{lacker-mfgcontrolledmartingaleproblems} for more details.
The elements of $\V$ are called \emph{relaxed controls}, and $q \in \V$ is called a \emph{strict control} if it satisfies $q(dt,da) = dt\delta_{\alpha_t}(da)$ for some measurable function $[0,T] \ni t \mapsto \alpha_t \in A$.
Finally, if we are given a measurable process $(\Lambda_t)_{t \in [0,T]}$ with values in $\P(A)$ defined on some measurable space and with $\int_0^T\int_A|a|^{p}\Lambda_t(da)dt < \infty$, we write $\Lambda = dt\Lambda_t(da)$ for the corresponding random element of $\V$.

Let us define some additional canonical spaces. For a positive integer $k$ let $\C^k = C([0,T];\R^k)$ denote the set of continuous functions from $[0,T]$ to $\R^k$, and define the truncated supremum norms $\|\cdot\|_t$ on $\C^k$ by
\begin{align}
\|x\|_t := \sup_{s \in [0,t]}|x_s|, \ t \in [0,T]. \label{def:truncatedsupnorm}
\end{align}
Unless otherwise stated, $\C^k$ is endowed with the norm $\|\cdot\|_T$ and its Borel $\sigma$-field.
For $\mu \in \P(\C^k)$, let $\mu_t \in \P(\R^k)$ denote the image of $\mu$ under the map $x \mapsto x_t$. 
Let
\begin{align}
\X &:= \C^m \times \V \times \C^d. \label{def:Xspace}
\end{align}
This space will house the idiosyncratic noise, the relaxed control, and the state process.
Let $(\F^\X_t)_{t \in [0,T]}$ denote the canonical filtration on $\X$, where $\F^\X_t$ is the $\sigma$-field generated by the maps
\begin{align*}
\X \ni (w,q,x) &\mapsto \left(w_s,x_s,q([0,s] \times C)\right) \in \R^m \times \R^d \times \R, \text{ for } s \le t, \ C \in \B(A).
\end{align*}
For $\mu \in \P(\X)$, let $\mu^x := \mu(\C^m \times \V \times \cdot)$ denote the $\C^d$-marginal. Finally, for ease of notation let us define the objective functional $\Gamma : \P^p(\C^d) \times \V \times \C^d \rightarrow \R$ by
\begin{align}
\Gamma(\mu,q,x) := \int_0^T\int_Af(t,x_t,\mu_t,a)q_t(da)dt + g(x_T,\mu_T). \label{def:gamma}
\end{align}
The following definition of weak mean field game (MFG) solution is borrowed from \cite{carmonadelaruelacker-mfgcommonnoise}.

\begin{definition} \label{def:weakMFGsolution}
A weak MFG solution with weak control (with initial state distribution $\lambda$), or simply a \emph{weak MFG solution}, is a tuple $(\widetilde{\Omega},(\F_t)_{t \in [0,T]},P,B,W,\mu,\Lambda,X)$, where $(\widetilde{\Omega},(\F_t)_{t \in [0,T]},P)$ is a complete filtered probability space supporting $(B,W,\mu,\Lambda,X)$ satisfying
\begin{enumerate}
\item $(B_t)_{t \in [0,T]}$ and $(W_t)_{t \in [0,T]}$ are independent $(\F_t)_{t \in [0,T]}$-Wiener processes of respective dimension $m_0$ and $m$, the process $(X_t)_{t \in [0,T]}$ is $(\F_t)_{t \in [0,T]}$-adapted with values in $\R^d$, and $P \circ X_0^{-1} = \lambda$. Moreover, $\mu$ is a random element of $\P^p(\X)$ such that $\mu(C)$ is $\F_t$-measurable for each $C \in \F^\X_t$ and $t \in [0,T]$.
\item $X_0$, $W$, and $(B,\mu)$ are independent.
\item $(\Lambda_t)_{t \in [0,T]}$ is $(\F_t)_{t \in [0,T]}$-progressively measurable with values in $\P(A)$ and
\[
\E^P\int_0^T\int_A|a|^p\Lambda_t(da)dt < \infty.
\]
Moreover, $\sigma(\Lambda_s : s \le t)$ is conditionally independent of $\F^{X_0,B,W,\mu}_T$ given $\F^{X_0,B,W,\mu}_t$, for each $t \in [0,T]$, where
\begin{align*}
\F^{X_0,B,W,\mu}_t &= \sigma\left(X_0,B_s,W_s,\mu(C) : s \le t, \ C \in \F^\X_t\right).
\end{align*}
\item The state equation holds:
\begin{align}
dX_t = \int_Ab(t,X_t,\mu^x_t,a)\Lambda_t(da)dt + \sigma(t,X_t,\mu^x_t)dW_t + \sigma_0(t,X_t,\mu^x_t)dB_t. \label{def:relaxedSDE-weak}
\end{align}
\item If $(\widetilde{\Omega}',(\F'_t)_{t \in [0,T]},P')$ is another filtered probability space supporting $(B',W',\mu',\Lambda',X')$ satisfying (1-4) and $P \circ (B,\mu)^{-1} = P' \circ (B',\mu')^{-1}$, then
\begin{align*}
\E^P\left[\Gamma(\mu^x,\Lambda,X)\right] \ge \E^{P'}\left[\Gamma(\mu'^x,\Lambda',X')\right].
\end{align*}
\item $\mu$ is a version of the conditional law of $(W,\Lambda,X)$ given $(B,\mu)$.
\end{enumerate}
If also there exists an $A$-valued process $(\alpha_t)_{t \in [0,T]}$ such that $P(\Lambda_t = \delta_{\alpha_t} \ a.e. \ t)=1$, then we say the weak MFG solution has \emph{strict control}. If this $(\alpha_t)_{t \in [0,T]}$ is progressively measurable with respect to the completion of $(\F^{X_0,B,W,\mu}_t)_{t \in [0,T]}$, we say the weak MFG solution has \emph{strong control}.
If $\mu$ is a.s. $B$-measurable, then we have a \emph{strong MFG solution} (with either weak control, strict control, or strong control).
\end{definition}

Given a weak MFG solution $(\widetilde{\Omega},(\F_t)_{t \in [0,T]},P,B,W,\mu,\Lambda,X)$, we may view $(X_0,B,W,\mu,\Lambda,X)$ as a random element of the canonical space
\begin{align}
\Omega := \R^d \times \C^{m_0} \times \C^m \times \P^p(\X) \times \V \times \C^d. \label{def:Omega}
\end{align}
A weak MFG solution thus induces a probability measure on $\Omega$, which itself we would like to call a MFG solution, as it is really the object of interest more than the particular probability space. The following definition will be reformulated in Section \ref{se:canonicalspace} in a more intrinsic manner.

\begin{definition} \label{def:weakMFGsolutionlaw}
If $P  \in \P(\Omega)$ satisfies $P = P' \circ (X_0,B,W,\mu,\Lambda,X)^{-1}$ for some weak MFG solution $(\Omega',(\F'_t)_{t \in [0,T]},P',B,W,\mu,\Lambda,X)$, then we refer to $P$ itself as a \emph{weak MFG solution}. Naturally, we may also refer to $P$ as a weak MFG solution with strict control or strong control, or as a strong MFG solution, under the analogous additional assumptions.
\end{definition}

\subsection{Finite-player games} \label{se:finiteplayergames}
This section describes a general form of the finite-player games, allowing controls to be relaxed and adapted to general filtrations.

An \emph{$n$-player environment} is defined to be any tuple $\mathcal{E}_n = (\Omega_n,(\F^n_t)_{t \in [0,T]},\PP_n,\xi,B,W)$, where $(\Omega_n,(\F^n_t)_{t \in [0,T]},\PP_n)$ is a complete filtered probability space supporting an $\F^n_0$-measurable $(\R^d)^n$-valued random variable $\xi = (\xi^1,\ldots,\xi^n)$ with law $\lambda^{\times n}$, an $m_0$-dimensional $(\F^n_t)_{t \in [0,T]}$-Wiener process $B$, and a $nm$-dimensional $(\F^n_t)_{t \in [0,T]}$-Wiener process $W = (W^1,\ldots,W^n)$, independent of $B$. For simplicity, we consider i.i.d. initial states $\xi^1,\ldots,\xi^n$ with common law $\lambda$, although it is presumably possible to generalize this. Perhaps all of the notation here should be parametrized by $\mathcal{E}_n$ or an additional index for $n$, but, since we will typically focus on a fixed sequence of environments $(\mathcal{E}_n)_{n=1}^\infty$, we avoid complicating the notation. Indeed, the subscript $n$ on the measure $\PP_n$ will be enough to remind us on which environment we are working at any moment.

Until further notice, we work with a fixed $n$-player environment $\mathcal{E}_n$. An \emph{admissible control} is any $(\F^n_t)_{t \in [0,T]}$-progressively measurable $\P(A)$-valued process $(\Lambda_t)_{t \in [0,T]}$ satisfying
\[
\E^{\PP_n}\int_0^T\int_A|a|^p\Lambda_t(da)dt < \infty.
\]
An \emph{admissible strategy} is a vector of $n$ admissible controls. The set of admissible controls is denoted $\A_n(\mathcal{E}_n)$, and accordingly the set of admissible strategies is the Cartesian product $\A_n^n(\mathcal{E}_n)$. A \emph{strict control} is any control $\Lambda \in \A_n(\mathcal{E}_n)$ such that $\PP_n(\Lambda_t = \delta_{\alpha_t}, \ a.e. \ t) = 1$ for some $(\F^n_t)_{t \in [0,T]}$-progressively measurable $A$-valued process $(\alpha_t)_{t \in [0,T]}$, and a \emph{strict strategy} is any vector of $n$ strict controls. 
Given an admissible control $\Lambda=(\Lambda^1,\ldots,\Lambda^n) \in \A_n^n(\mathcal{E}_n)$ define the state processes $X[\Lambda] := (X^1[\Lambda],\ldots,X^n[\Lambda])$ by
\begin{align*}
dX^i_t[\Lambda] &= \int_Ab(t,X^i_t[\Lambda],\widehat{\mu}^x_t[\Lambda],a)\Lambda^i_t(da)dt + \sigma(t,X^i_t[\Lambda],\widehat{\mu}^x_t[\Lambda])dW^i_t \\
	&\quad + \sigma_0(t,X^i_t[\Lambda],\widehat{\mu}^x_t[\Lambda])dB_t, \quad\quad X^i_0 = \xi^i, \\
\widehat{\mu}^x[\Lambda] &:= \frac{1}{n}\sum_{k=1}^n\delta_{X^k[\Lambda]}.
\end{align*}
Note that assumption \ref{assumption:A} ensures that a unique strong solution of this SDE system exists\footnote{As in \cite{carmonadelaruelacker-mfgcommonnoise}, we avoid augmenting the filtrations to be right-continuous, taking advantage of the careful treatment of stochastic integration of \cite[Section 4.3]{stroockvaradhanbook}.}. Indeed, the Lipschitz assumption of (A.4) and the obvious inequality
\[
\ell^p_{\R^d,p}\left(\frac{1}{n}\sum_{i=1}^n\delta_{x_i},\frac{1}{n}\sum_{i=1}^n\delta_{y_i}\right) \le \frac{1}{n}\sum_{i=1}^n|x_i-y_i|^p
\]
together imply, for example, that the function
\[
(\R^d)^n \ni (x_1,\ldots,x_n) \mapsto b\left(t,x_1,\frac{1}{n}\sum_{i=1}^n\delta_{x_i},a\right) \in \R^d
\]
is Lipschitz, uniformly in $(t,a)$. A standard estimate using assumption (A.4), which is worked out in Lemma \ref{le:finitestateestimate}, shows that $\E^{\PP_n}[\|X^i[\Lambda]\|_T^p] < \infty$ for each $\Lambda \in \A_n^n(\mathcal{E}_n)$, $n \ge i \ge 1$.

The value for player $i$ corresponding to a strategy $\Lambda = (\Lambda^1,\ldots,\Lambda^n) \in \A_n^n(\mathcal{E}_n)$ is defined by
\[
J_i(\Lambda) := \E^{\PP_n}\left[\Gamma(\widehat{\mu}^x[\Lambda],\Lambda^i,X^i_t[\Lambda])\right].
\]
Note that $J_i(\Lambda) < \infty$ is well-defined because of the upper bounds of assumption (A.5), but it is possible that $J_i(\Lambda) = -\infty$, since we do not require that an admissible control possess finite moment of order $p'$.
Given a strategy $\Lambda = (\Lambda^1,\ldots,\Lambda^n) \in \A_n^n(\mathcal{E}_n)$ and a control $\beta \in \A_n(\mathcal{E}_n)$, define a new strategy $(\Lambda^{-i},\beta) \in \A_n^n(\mathcal{E}_n)$ by
\[
(\Lambda^{-i},\beta) = (\Lambda^1,\ldots,\Lambda^{i-1},\beta,\Lambda^{i+1},\ldots,\Lambda^n).
\]
Given $\epsilon = (\epsilon_1,\ldots,\epsilon_n) \in [0,\infty)^n$, a \emph{relaxed $\epsilon$-Nash equilibrium in $\mathcal{E}_n$} is any strategy $\Lambda \in \A_n^n(\mathcal{E}_n)$ satisfying
\[
J_i(\Lambda) \ge \sup_{\beta \in \A_n(\mathcal{E}_n)}J_i((\Lambda^{-i},\beta)) - \epsilon_i, \quad i=1,\ldots,n.
\]
Naturally, if $\epsilon_i=0$ for each $i=1,\ldots,n$, we use the simpler term \emph{Nash equilibrium}, as opposed to \emph{$0$-Nash equilibrium}.
A \emph{strict $\epsilon$-Nash equilibrium in $\mathcal{E}_n$} is any \emph{strict} strategy $\Lambda \in \A_n^n(\mathcal{E}_n)$ satisfying
\[
J_i(\Lambda) \ge \sup_{\beta \in \A_n(\mathcal{E}_n) \text{ strict}}J_i((\Lambda^{-i},\beta)) - \epsilon_i, \quad i=1,\ldots,n.
\]
Note that the optimality is required only among \emph{strict} controls.

Note that the role of the filtration $(\F^n_t)_{t \in [0,T]}$ in the environment $\mathcal{E}_n$ is mainly to specify the class of admissible controls. We are particularly interested in the sub-filtration generated by the Wiener processes and initial states; define $(\F^{s,n}_t)_{t \in [0,T]}$ to be the $\PP_n$-completion of
\[
\left(\sigma(\xi,B_s,W_s : s \le t)\right)_{t \in [0,T]}.
\]
Of course, $\F^{s,n}_t \subset \F^n_t$ for each $t$. Let us say that $\Lambda \in \A_n(\mathcal{E}_n)$ is a \emph{strong control} if $\PP_n(\Lambda_t = \delta_{\alpha_t} \ a.e. \ t)=1$ for some $(\F^{s,n}_t)_{t \in [0,T]}$-progressively measurable $A$-valued process $(\alpha_t)_{t \in [0,T]}$. Naturally, a \emph{strong strategy} is a vector of strong controls. 
A \emph{strong $\epsilon$-Nash equilibrium in $\mathcal{E}_n$} is any \emph{strong} strategy $\Lambda \in \A_n^n(\mathcal{E}_n)$ such that
\[
J_i(\Lambda) \ge \sup_{\beta \in \A_n(\mathcal{E}_n) \text{ strong}}J_i((\Lambda^{-i},\beta)) - \epsilon_i, \quad i=1,\ldots,n.
\]

\begin{remark} \label{re:equivalentstrongequilibrium}
Equivalently, a strong $\epsilon$-Nash equilbrium in $\mathcal{E}_n = (\Omega_n,(\F^n_t)_{t \in [0,T]},\PP_n,\xi,B,W)$ is a strict $\epsilon$-Nash equilibrium in $\widetilde{\mathcal{E}}_n := (\Omega_n,(\F^{s,n}_t)_{t \in [0,T]},\PP_n,\xi,B,W)$.
\end{remark}

The most common type of Nash equilibrium considered in the literature is, in our terminology, a strong Nash equilibrium. The next proposition assures us that our equilibrium concept using relaxed controls (and general filtrations) truly generalizes this more standard situation, thus permitting a unified analysis of all of the equilibria described thusfar. The proof is deferred to Appendix \ref{se:proof-equilibriuminclusions}.

\begin{proposition} \label{pr:equilibriuminclusions}
On any $n$-player environment $\mathcal{E}_n$, every strong $\epsilon$-Nash equilibrium is also a strict $\epsilon$-Nash equilibrium, and every strict $\epsilon$-Nash equilibrium is also a relaxed $\epsilon$-Nash equilibrium.
\end{proposition}

\begin{remark}
Another common type of strategy in dynamic game theory is called \emph{closed-loop}. Whereas our strategies (also called \emph{open-loop}) are specified by \emph{processes}, a closed-loop (strict) strategy is specified by feedback functions $\phi_i : [0,T] \times (\R^d)^n \rightarrow A$, for $i=1,\ldots,n$, to be evaluated along the path of the state process. In the model of Carmona et al. \cite{carmonafouque-systemicrisk}, both the open-loop and closed-loop equilibria are computed explicitly for the $n$-player games, and they are shown to converge to the same MFG limit. There is no distinction between open-loop and closed-loop in the MFG, and this begs the question of whether or not closed-loop equilibria converge to the same MFG limit that we obtain in Theorem \ref{th:mainconvergence}. This paper does not attempt to answer this question.
\end{remark}

\subsection{The main limit theorem} \label{se:mainresults}
We are ready now to state the first main Theorem \ref{th:mainconvergence} and its corollaries. 
The proof is deferred to Section \ref{se:mainconvergenceproof}.
Given an admissible strategy $\Lambda = (\Lambda^1,\ldots,\Lambda^n) \in \A_n^n(\mathcal{E}_n)$ defined on some $n$-player environment $\mathcal{E}_n = (\Omega_n,(\F^n_t)_{t \in [0,T]},\PP_n,\xi,B,W)$, define (on $\Omega_n$) the random element $\widehat{\mu}[\Lambda]$ of $\P^p(\X)$ (recalling the definition of $\X$ from \eqref{def:Xspace}) by
\[
\widehat{\mu}[\Lambda] := \frac{1}{n}\sum_{i=1}^n\delta_{(W^i,\Lambda^i,X^i[\Lambda])}.
\]
As usual, we identify a $\P(A)$-valued process $(\Lambda^i_t)_{t \in [0,T]}$ with the random element $\Lambda^i = dt\Lambda^i_t(da)$ of $\V$. Recall the definition of the canonical space $\Omega$ from \eqref{def:Omega}.

\begin{theorem} \label{th:mainconvergence}
Suppose assumption \ref{assumption:A} holds. For each $n$, let $\epsilon^n = (\epsilon^n_1,\ldots,\epsilon^n_n) \in [0,\infty)^n$, and let $\mathcal{E}_n = (\Omega_n,(\F^n_t)_{t \in [0,T]},\PP_n,\xi,B,W)$ be any $n$-player environment. Assume
\begin{align}
\lim_{n \rightarrow\infty} \frac{1}{n}\sum_{i=1}^n\epsilon^n_i = 0. \label{def:epsilonconverge}
\end{align}
Suppose for each $n$ that $\Lambda^n = (\Lambda^{n,1},\ldots,\Lambda^{n,n}) \in \A_n^n(\mathcal{E}_n)$ is a relaxed $\epsilon^n$-Nash equilibrium, and let
\begin{align}
P_n := \frac{1}{n}\sum_{i=1}^n\PP_n \circ \left(\xi^i,B,W^i,\widehat{\mu}[\Lambda^n],\Lambda^{n,i},X^i[\Lambda^n]\right)^{-1}. \label{def:pn}
\end{align}
Then $(P_n)_{n=1}^\infty$ is relatively compact in $\P^p(\Omega)$, and each limit point is a weak MFG solution.
\end{theorem}

\begin{remark} \label{re:exchangeability}
Averaging over $i=1,\ldots,n$ in \eqref{def:pn} circumvents the problem that the strategies $(\Lambda^{n,1},\ldots,\Lambda^{n,n})$ need not be exchangeable, and we note that the limiting behavior of $\PP_n \circ (B,\widehat{\mu}[\Lambda^n])^{-1}$ can always be recovered from that of $P_n$. To interpret the definition of $P_n$, note that we may write
\[
P_n = \PP_n \circ \left(\xi^{U_n},B,W^{U_n},\widehat{\mu}[\Lambda^n],\Lambda^{n,U_n},X^{U_n}[\Lambda^n]\right)^{-1},
\]
where $U_n$ is a random variable independent of $\F^n_T$, uniformly distributed among $\{1,\ldots,n\}$, constructed by extending the probability space $\Omega_n$. In words, $P_n$ is the joint law of the processes relevant to a \emph{randomly selected representative agent}. Of course, Theorem \ref{th:mainconvergence} specializes when there is exchangeability, in the following sense. For any set $E$, any element $e = (e^1,\ldots,e^n) \in E^n$, and any permutation $\pi$ of $\{1,\ldots,n\}$, let $e_\pi := (e^{\pi(1)},\ldots,e^{\pi(n)})$. If
\[
\PP_n \circ \left(\xi_\pi,B,W_\pi,\Lambda^n_\pi\right)^{-1}
\]
is independent of the choice of permutation $\pi$, then so is
\[
\PP_n \circ \left(\xi_\pi,B,W_\pi,\widehat{\mu}[\Lambda^n_\pi],\Lambda^n_\pi,X[\Lambda^n_\pi]_\pi\right)^{-1}.
\]
It then follows that
\[
P_n = \PP_n \circ \left(\xi^k,B,W^k,\widehat{\mu}[\Lambda^n],\Lambda^{n,k},X^k[\Lambda^n]\right)^{-1}, \text{ for } n \ge k.
\]
\end{remark}

Theorem \ref{th:mainconvergence} is stated in quite a bit of generality, devoid even of standard convexity assumptions on the objective functions $f$ and $g$. Theorem \ref{th:mainconvergence} includes quite degenerate cases, such as the case of \emph{no objectives}, where $f \equiv g \equiv 0$ and $A$ is compact. In this case, \emph{any strategy profile} whatsoever in the $n$-player game is a Nash equilibrium, and any weak control can arise in the limit. Exploiting results of \cite{carmonadelaruelacker-mfgcommonnoise}, the following corollaries demonstrate how, under various additional convexity assumptions, we may refine the conclusion of Theorem \ref{th:mainconvergence} by ruling out certain types of limits, such as those involving relaxed controls.

\begin{corollary} \label{co:weakstrictlimit}
Suppose the assumptions of Theorem \ref{th:mainconvergence} hold, and assume also that for each $(t,x,\mu) \in [0,T] \times \R^d \times \P^p(\R^d)$ the following subset of $\R^d \times \R$ is convex:
\[
\left\{(b(t,x,\mu,a),z) : a \in A, \ z \le f(t,x,\mu,a) \right\}.
\]
Then 
\[
\left\{\frac{1}{n}\sum_{i=1}^n\PP_n \circ \left(B,W^i,\widehat{\mu}^x[\Lambda^n],X^i[\Lambda^n]\right)^{-1} : n \ge 1 \right\}
\]
is relatively compact in $\P^p(\C^{m_0} \times \C^m \times \P^p(\C^d) \times \C^d)$, 
and every limit is of the form $P \circ (B,W,\mu^x,X)^{-1}$, for some weak MFG solution with strict control $(\widetilde{\Omega},(\F_t)_{t \in [0,T]},P,B,W,\mu,\Lambda,X)$.
\end{corollary}
\begin{proof}
This follows from Theorem \ref{th:mainconvergence} and the argument of \cite[Theorem 4.1]{carmonadelaruelacker-mfgcommonnoise}. Indeed, the latter shows that for every weak MFG solution with \emph{weak} control $(\widetilde{\Omega},(\F_t)_{t \in [0,T]},P,B,W,\mu,\Lambda,X)$, there exists a weak MFG solution with \emph{strict} control $(\widetilde{\Omega}',(\F'_t)_{t \in [0,T]},P',B',W',\mu',\Lambda',X')$ such that $P \circ (B,W,\mu^x,X)^{-1} = P' \circ (B',W',\mu'^x,X')^{-1}$.
\end{proof}

\begin{corollary} \label{co:stronglimit}
Suppose the assumptions of Theorem \ref{th:mainconvergence} hold, and define $P_n$ as in \eqref{def:pn}. Assume also that for each fixed $(t,\mu) \in [0,T] \times \P^p(\R^d)$, $(b,\sigma,\sigma_0)(t,x,\mu,a)$ is affine in $(x,a)$, $g(x,\mu)$ is concave in $x$, and $f(t,x,\mu,a)$ is strictly concave in $(x,a)$. Then $(P_n)_{n=1}^\infty$ is relatively compact in $\P^p(\Omega)$, and every limit point is a weak MFG solution with \emph{strong} control.
\end{corollary}
\begin{proof}
By \cite[Proposition 4.4]{carmonadelaruelacker-mfgcommonnoise}, the present assumptions guarantee that every weak MFG solution is a weak MFG solution with strong control. The claim then follows from Theorem \ref{th:mainconvergence}.
\end{proof}

Finally, we provide an example of the satisfying situation, in which there is a unique MFG solution. Say that \emph{uniqueness in law} holds for the MFG if any two weak MFG solutions induce the same law on $\Omega$. The following corollary is an immediate consequence of Theorem \ref{th:mainconvergence} and the uniqueness result of \cite[Theorem 6.2]{carmonadelaruelacker-mfgcommonnoise}, which makes use of the monotonicity assumption of Lasry and Lions \cite{lasrylionsmfg}.

\begin{corollary} \label{co:uniquelimit}
Suppose the assumptions of Corollary \ref{co:stronglimit} hold, and define $P_n$ as in \eqref{def:pn}. Assume also that
\begin{enumerate}
\item $b$, $\sigma$, and $\sigma_0$ have no mean field term, i.e. no $\mu$ dependence,
\item $f$ is of the form $f(t,x,\mu,a) = f_1(t,x,a) + f_2(t,x,\mu)$,
\item For each $\mu,\nu \in \P^p(\C^d)$ we have
\[
\int_{\C^d}(\mu-\nu)(dx)\left[g(x_T,\mu_T) - g(x_T,\nu_T) + \int_0^T\left(f_2(t,x,\mu) - f_2(t,x,\nu)\right)dt\right] \le 0.
\]
\end{enumerate}
Then there exists a unique in law weak MFG solution, and it is a strong MFG solution with strong control. In particular, $P_n$ converges in $\P^p(\Omega)$ to this unique MFG solution.
\end{corollary}

\subsection{The converse limit theorem} \label{se:converse}
This section states and discusses a converse to Theorem \ref{th:mainconvergence}. 
For this, we need an additional technical assumption, which we note holds automatically under assumption \ref{assumption:A} in the case that the control space $A$ is compact.

\begin{assumption}{\textbf{B}} \label{assumption:B}
The function $f$ of $(t,x,\mu,a)$ is continuous in $(x,\mu)$, \emph{uniformly in $a$}, for each $t \in [0,T]$. That is,
\[
\lim_{(x',\mu') \rightarrow (x,\mu)}\sup_{a \in A}\left|f(t,x',\mu',a) - f(t,x,\mu,a)\right| = 0, \ \forall t \in [0,T].
\]
Moreover, there exists $c_4 > 0$ such that, for all $(t,x,x',\mu,\mu',a)$,
\[
\left|f(t,x',\mu',a) - f(t,x,\mu,a)\right| \le c_4\left(1 + |x'|^p + |x|^p + \int_{\R^d}|z|^p(\mu' + \mu)(dz)\right).
\]
\end{assumption}

\begin{theorem} \label{th:converseconvergence}
Suppose assumptions \ref{assumption:A} and \ref{assumption:B} hold. Let $P \in \P(\Omega)$ be a weak MFG solution, and for each $n$ let $\mathcal{E}_n = (\Omega_n,(\F^n_t)_{t \in [0,T]},\PP_n,\xi,B,W)$ be any $n$-player environment. Then there exist, for each $n$, $\epsilon_n \ge 0$ and a strong $(\epsilon_n,\ldots,\epsilon_n)$-Nash equilibrium $\Lambda^n = (\Lambda^{n,1},\ldots,\Lambda^{n,n})$ on $\mathcal{E}_n$,
such that $\lim_{n\rightarrow\infty}\epsilon_n = 0$ and
\begin{align}
P = \lim_{n\rightarrow\infty}\frac{1}{n}\sum_{i=1}^n\PP_n \circ \left(\xi^i,B,W^i,\widehat{\mu}[\Lambda^n],\Lambda^{n,i},X^i[\Lambda^n]\right)^{-1}, \text{ in } \P^p(\Omega). \label{def:conversesequence}
\end{align}
\end{theorem}

Combining Theorems \ref{th:mainconvergence} and \ref{th:converseconvergence} shows that the set of weak MFG solutions is exactly the set of limits of strong approximate Nash equilibria. More precisely, the set of weak MFG solutions is exactly the set of limits
\[
\lim_{k\rightarrow\infty}\frac{1}{n_k}\sum_{i=1}^{n_k}\PP_{n_k} \circ \left(\xi^i,B,W^i,\widehat{\mu}[\Lambda^{n_k}],\Lambda^{n_k,i},X^i[\Lambda^{n_k}]\right)^{-1},
\]
where $\Lambda^n \in \A_n^n(\mathcal{E}_n)$ are strong $\epsilon^n$-Nash equilibria and $\epsilon^n=(\epsilon^n_1,\ldots,\epsilon^n_n) \in [0,\infty)^n$ satisfies \eqref{def:epsilonconverge}.
The same statement is true when the word ``strong'' is replaced by ``strict'' or ``relaxed'', because of Proposition \ref{pr:equilibriuminclusions}. Similarly, combining Theorem \ref{th:converseconvergence} with Corollaries \ref{co:weakstrictlimit} and \ref{co:stronglimit} yields characterizations of the mean field limit without recourse to relaxed controls.

\begin{remark} \label{re:insensitivetoenvironment}
In light of Remark \ref{re:equivalentstrongequilibrium}, the statement of Theorem \ref{th:converseconvergence} is insensitive to the choice of environments $\mathcal{E}_n$. Without loss of generality, they may all be assumed to satisy $\F^n_t = \F^{s,n}_t$ for each $t$; that is, the filtration may be taken to be the one generated by the process $(\xi,B_t,W_t)_{t \in [0,T]}$.
\end{remark}

\begin{remark}
It follows from the proofs of Theorems \ref{th:mainconvergence} and \ref{th:converseconvergence} that the \emph{values} converge as well, in the sense that $\frac{1}{n}\sum_{i=1}^nJ_i(\Lambda^n)$ converges (along a subsequence in the case of Theorem \ref{th:mainconvergence}) to the corresponding optimal value corresponding to the MFG solution.
\end{remark}

\begin{remark} \label{re:distributedstrategies}
Theorem \ref{th:converseconvergence} is admittedly abstract, and not as strong in its conclusion as the typical results of this nature in the literature. Namely, in the setting without common noise, it is usually argued as in \cite{huangmfg1} that a MFG solution may be used to construct not just any sequence of approximate equilibria, but rather one consisting of \emph{symmetric distributed strategies}, in which the control of agent $i$ is of the form $\hat{\alpha}(t,X^i_t)$ for some function $\hat{\alpha}$ which depends neither on the agent $i$ nor the number of agents $n$. The techniques of this paper seem too abstract to yield a result of this nature, but in any case this would stray from the objective of the paper.
On a somewhat related note, at the level of generality of Theorem \ref{th:converseconvergence} we do not expect to obtain a rate of convergence of $\epsilon_n$, as in \cite{kolokoltsov-mfgnonlinearmarkov,carmonadelarue-mfg}.
\end{remark}

\section{The case of no common noise} \label{se:nocommonnoise}

The goal of this section is to specialize the main results to MFGs without common noise. Indeed \emph{we assume that $\sigma_0 \equiv 0$ throughout this section}. Assumption \ref{assumption:A} permits degenerate volatility, but when $\sigma_0 \equiv 0$ our general definition of weak MFG solution still involves the common noise $B$, which in a sense should no longer play any role. To be absolutely clear, we will rewrite the definitions and the two main theorems so that they do not involve a common noise; most notably, the notion of \emph{strong controls} for the finite-player games is refined to \emph{very strong controls}.

The proofs of the main results of Section \ref{se:withoutcommonnoise-results}, Proposition \ref{pr:verystrongeq} and Theorem \ref{th:withoutcommonnoise}, are deferred to Section \ref{se:withoutcommonnoiseproof}, where we will see how to deduce
almost all of the results \emph{without} common noise from those \emph{with} common noise. 
Crucially, even without common noise, a weak MFG solution still involves a \emph{random} measure $\mu$, and the consistency condition becomes $\mu = P((W,\Lambda,X) \in \cdot \ | \ \mu)$. We illustrate by example just how different weak solutions can be from the strong solutions typically considered in the MFG literature, in which $\mu$ is deterministic. Finally we close the section by discussing some situations in which weak solutions are concentrated on the family of strong solutions.

\subsection{Definitions and results} \label{se:withoutcommonnoise-results}
First, let us state a simplified definition of MFG solution for the case $\sigma_0 \equiv 0$, which is really just Definition \ref{def:weakMFGsolution} rewritten without $B$. Again, the following definition is relative to the initial state distribution $\lambda$.

\begin{definition} \label{def:mfgsolutionwithoutcommonnoise}
A \emph{weak MFG solution without common noise} is a tuple $(\widetilde{\Omega},(\F_t)_{t \in [0,T]},P,W,\mu,\Lambda,X)$, where $(\widetilde{\Omega},(\F_t)_{t \in [0,T]},P)$ is a complete filtered probability space supporting $(W,\mu,\Lambda,X)$ satisfying
\begin{enumerate}
\item $(W_t)_{t \in [0,T]}$ is an $(\F_t)_{t \in [0,T]}$-Wiener processes of dimension $m$, the process $(X_t)_{t \in [0,T]}$ is $(\F_t)_{t \in [0,T]}$-adapted with values in $\R^d$, and $P \circ X_0^{-1} = \lambda$. Moreover, $\mu$ is a random element of $\P^p(\X)$ such that $\mu(C)$ is $\F_t$-measurable for each $C \in \F^\X_t$ and $t \in [0,T]$.
\item $X_0$, $W$, and $\mu$ are independent.
\item $(\Lambda_t)_{t \in [0,T]}$ is $(\F_t)_{t \in [0,T]}$-progressively measurable with values in $\P(A)$ and
\[
\E^P\int_0^T\int_A|a|^p\Lambda_t(da)dt < \infty.
\]
Moreover, $\sigma(\Lambda_s : s \le t)$ is conditionally independent of $\F^{X_0,W,\mu}_T$ given $\F^{X_0,W,\mu}_t$, for each $t \in [0,T]$, where
\begin{align*}
\F^{X_0,W,\mu}_t &= \sigma\left(X_0,W_s,\mu(C) : s \le t, \ C \in \F^\X_t\right).
\end{align*}
\item The state equation holds:
\begin{align}
dX_t = \int_Ab(t,X_t,\mu^x_t,a)\Lambda_t(da)dt + \sigma(t,X_t,\mu^x_t)dW_t. \label{def:relaxedSDE-weak-ncn}
\end{align}
\item If $(\widetilde{\Omega}',(\F'_t)_{t \in [0,T]},P')$ is another filtered probability space supporting $(W',\mu',\Lambda',X')$ satisfying (1-4) and $P \circ \mu^{-1} = P' \circ (\mu')^{-1}$, then
\begin{align*}
\E^P\left[\Gamma(\mu^x,\Lambda,X)\right] \ge \E^{P'}\left[\Gamma(\mu'^x,\Lambda',X')\right].
\end{align*}
\item $\mu$ is a version of the conditional law of $(W,\Lambda,X)$ given $\mu$.
\end{enumerate}
As in Definition \ref{def:weakMFGsolutionlaw}, we may refer to the law $P \circ (W,\mu,\Lambda,X)^{-1}$ itself as a weak MFG solution.
Again, if also there exists an $A$-valued process $(\alpha_t)_{t \in [0,T]}$ such that $P(\Lambda_t = \delta_{\alpha_t} \ a.e. \ t)=1$, then we say the MFG solution has \emph{strict control}. If this $(\alpha_t)_{t \in [0,T]}$ is progressively measurable with respect to the completion of $(\F^{X_0,W,\mu}_t)_{t \in [0,T]}$, we say the MFG solution has \emph{strong control}.
If $\mu$ is a.s.-constant, then we have a \emph{strong MFG solution without common noise}. In this case, we may abuse the terminology somewhat by saying that a measure $\widetilde{\mu} \in \P^p(\X)$ is itself a \emph{strong MFG solution (without common noise)}, if there exists a weak MFG solution $(\widetilde{\Omega},(\F_t)_{t \in [0,T]},P,W,\mu,\Lambda,X)$ without common noise such that $P(\mu = \widetilde{\mu}) = 1$.
\end{definition}

\begin{remark} \label{re:deterministicmfgsolution}
Our notion of \emph{strong MFG solution without common noise with strong control} corresponds to the usual definition of MFG solution in the literature.
It is exactly the definition used in the recent papers \cite{fischer-mfgconnection,lacker-mfgcontrolledmartingaleproblems}, and it is a generalization of the more standard definition of MFG solution without common noise found in \cite{huangmfg1,carmonadelarue-mfg,bensoussan-lqmfg}, for example. The latter papers require optimality \emph{only relative to other strong controls}, not among all weak controls as we do in condition (5) of Definition \ref{def:mfgsolutionwithoutcommonnoise}. Under assumption \ref{assumption:A}, however, optimality among strong controls implies optimality among weak controls, and thus our definition does include this more standard one. This is the same phenomenon driving Propositions \ref{pr:equilibriuminclusions} and \ref{pr:verystrongeq}, and it is well known in control theory. Remark \ref{re:adapteddensity} will elaborate on this point, and see also \cite{elkaroui-compactification} or the more recent \cite{karouitan-capacities} for further dicussion. 
\end{remark}

We continue to work with the definition of the $n$-player games of the previous section. Suppose we are given an $n$-player environment $\mathcal{E}_n = (\widetilde{\Omega}_n,(\F^n_t)_{t \in [0,T]},\PP_n,\xi,B,W)$, as was defined in Section \ref{se:finiteplayergames}. Let $(\F^{vs,n}_t)_{t \in [0,T]}$ denote the $\PP_n$-completion of $(\sigma(\xi,W_s : s \le t))_{t \in [0,T]}$, that is the filtration generated by the initial state and the idiosyncratic noises (but not the common noise). Let us say that a control $\Lambda \in \A_n(\mathcal{E}_n)$ is a \emph{very strong control} if $\PP_n(\Lambda_t = \delta_{\alpha_t} \ a.e. \ t) = 1$, for some $(\F^{vs,n}_t)_{t \in [0,T]}$-progressively measurable $A$-valued process $(\alpha_t)_{t \in [0,T]}$. A \emph{very strong strategy} is a vector of strong controls. 
For $\epsilon=(\epsilon_1,\ldots,\epsilon_n) \in [0,\infty)^n$, a \emph{very strong $\epsilon$-Nash equilibrium in $\mathcal{E}_n$} is any \emph{very strong} strategy $\Lambda \in \A_n^n(\mathcal{E}_n)$ such that
\[
J_i(\Lambda) \ge \sup_{\beta \in \A_n(\mathcal{E}_n) \text{ very strong}}J_i((\Lambda^{-i},\beta)) - \epsilon_i, \quad i=1,\ldots,n.
\]
The \emph{very strong} equilibrium is arguably the most natural notion of equilibrium in the case of no common noise, and it is certainly one of the most common in the literature.
The proof of the following Proposition is deferred to Appendix \ref{se:proof-verystrongeq}. 

\begin{proposition} \label{pr:verystrongeq}
When $\sigma_0\equiv 0$, every very strong $\epsilon$-Nash equilibrium is also a relaxed $\epsilon$-Nash equilibrium.
\end{proposition}

The following Theorem \ref{th:withoutcommonnoise} rewrites Theorems \ref{th:mainconvergence} and \ref{th:converseconvergence} in the setting without common noise. Although this is mostly derived from Theorems \ref{th:mainconvergence} and \ref{th:converseconvergence}, the proof is spelled out in Section \ref{se:withoutcommonnoiseproof}, as it is not entirely straightforward.

\begin{theorem} \label{th:withoutcommonnoise}
Suppose $\sigma_0 \equiv 0$. Theorem \ref{th:mainconvergence} remains true if the term ``weak MFG solution'' is replaced by ``weak MFG solution without common noise,'' and if $P_n$ is defined instead by
\begin{align}
P_n := \frac{1}{n}\sum_{i=1}^n\PP_n \circ \left(\xi^i,W^i,\widehat{\mu}[\Lambda^n],\Lambda^{n,i},X^i[\Lambda^n]\right)^{-1}. \label{def:pnwithoutcommonnoise}
\end{align}
Theorem \ref{th:converseconvergence} remains true if ``weak MFG solution'' is replaced by ``weak MFG solution without common noise,'' if $P_n$ is defined by \eqref{def:pnwithoutcommonnoise}, and if ``strong'' is replaced by ``very strong.''
\end{theorem}

Since strong MFG solutions are more familiar in the literature on mean field games and presumably more accessible computationally, it would be nice to have a description of weak solutions in terms of strong solutions. We will see that this is not possible in general, and the investigation of this issue highlights the fundamental difference between stochastic and deterministic equilibria (i.e. weak and strong MFG solutions). First, a discussion of a special case will help to clarify the ideas.

\subsection{A digression on McKean-Vlasov equations} \label{se:mkvlimits}
When there is no control (when $A$ is a singleton), the mean field game reduces to a McKean-Vlasov equation. In this case, an interesting simplification occurs: every weak solution is simply a randomization over strong solutions. To be more clear, suppose we have a system of weakly interacting diffusions, given by
\begin{align*}
dX^i_t &= \tilde{b}(t,X^i_t,\mu^n_t)dt + \tilde{\sigma}(t,X^i_t,\mu^n_t)dW^i_t, \\
\mu^n &:= \frac{1}{n}\sum_{k=1}^n\delta_{X^k}.
\end{align*}
A common argument in the theory of McKean-Vlasov limits \cite{oelschlagermkv,gartnermkv,sznitman} is to show, under suitable assumptions on $(\tilde{b},\tilde{\sigma})$, that $(\mu^n)_{n=1}^\infty$ is tight, and that every weak limit point (an element of $\P(\P(\C^d))$) is concentrated on the set of solutions $\mu \in \P(\C^d)$ of the following \emph{strong McKean-Vlasov equation}:
\[
\begin{cases}
dX_t &= \tilde{b}(t,X_t,\mu_t)dt + \tilde{\sigma}(t,X_t,\mu_t)dW_t, \\
\mu &= \text{Law}(X).
\end{cases}
\]
Consider also searching for a $\P(\C^d)$-valued random variable $\mu$ satisfying the \emph{weak McKean-Vlasov equation}:
\[
\begin{cases}
dX_t &= \tilde{b}(t,X_t,\mu_t)dt + \tilde{\sigma}(t,X_t,\mu_t)dW_t, \\
\mu &= \text{Law}(X \ | \ \mu), \text{ with } X_0, \mu, W \text{ independent},
\end{cases}
\]
It is not too difficult to convince yourself that a $\P(\C^d)$-valued random variable satisfies the weak McKean-Vlasov equation if and only if it almost surely satisfies the strong McKean-Vlasov equation. That is, every weak solution is supported on the set of strong solutions. In particular, we find that the set of strong McKean-Vlasov solutions is rich enough to characterize all of the possible limiting behaviors of the finite-particle systems.

In general, no such simplification is available for mean field games. This is essentially because the adaptedness requirement makes the class of admissible controls quite dependent on how random $\mu$ is. To highlight this point, Section \ref{se:example} below describes a model possessing weak MFG solutions which are not randomizations of strong MFG solutions. Subsection \ref{se:supportsofweaksolutions} discusses some partial results on when this simplification can occur in the MFG setting.

\subsection{An illuminating example} \label{se:example}
This section describes a deceptively simple example which illustrates the difference between weak and strong solutions. Consider the time horizon $T = 2$, the initial state distribution $\lambda = \delta_0$, and the following data (still with $\sigma_0 \equiv 0$):
\begin{align*}
b(t,x,\nu,a) &= a, \quad \sigma \text{ constant}, \quad A = [-1,1] \\
g(x,\nu) &= x\bar{\nu}, \quad f \equiv 0,
\end{align*}
where for $\nu \in \P^1(\R)$ we define $\bar{\nu} := \int x\nu(dx)$. Similarly, for $\mu \in \P^1(\X)$ write $\bar{\mu}^x_t := \int_{\R}x\mu^x_t(dx)$. Assumption \ref{assumption:A} is verified by choosing $p=2$, $p_\sigma = 0$, and any $p' > 2$. Let us first study the optimization problems arising in the MFG problem. Let $(\widetilde{\Omega},(\F_t)_{t \in [0,2]},P,W,\mu,\Lambda,X)$ satisfy (1-5) of Definition \ref{def:mfgsolutionwithoutcommonnoise}. For $(\F_t)_{t\in [0,2]}$-progressively measurable $\P([-1,1])$-valued processes $\beta = (\beta_t)_{t\in [0,2]}$, define
\[
\widetilde{J}(\beta) := \E\left[X^\beta_2\bar{\mu}^x_2\right],
\]
where
\[
X^\beta_t = \int_0^t\int_{[-1,1]}a\beta_t(da)dt + \sigma W_t, \ t \in [0,2].
\]
Independence of $W$ and $\mu$ implies
\[
\widetilde{J}(\beta) = \E\left[\int_0^2\int_{[-1,1]}a\bar{\mu}^x_2\beta_t(da)dt\right] =  \E\left[\int_0^2\int_{[-1,1]}a\E^P[\bar{\mu}^x_2 \ | \ \F^\beta_t]\beta_t(da)dt\right],
\]
where $\F^\beta_t := \sigma(\beta_s : s \le t)$.
If it is also required that $\F^\beta_t$ is conditionally independent of $\F^{X_0,W,\mu}_2$ given $\F^{X_0,W,\mu}_t$, then
\[
\E^P[\bar{\mu}^x_2 \ | \ \F^\beta_t] = \E^P[\bar{\mu}^x_2 \ | \ \F^{X_0,W,\mu}_t] = \E^P[\bar{\mu}^x_2 \ | \ \F^\mu_t],
\]
where the last equality follows from independence of $(X_0,W)$ and $\mu$, and $\F^\mu_t := \sigma(\mu(C) : C \in \F^\X_t)$. Hence
\begin{align}
\widetilde{J}(\beta) = \E\left[\int_0^2\int_{[-1,1]}a\E^P[\bar{\mu}^x_2 \ | \ \F^{\mu}_t]\beta_t(da)dt\right]. \label{ex:jdef}
\end{align}
Condition (5) of Definition \ref{def:mfgsolutionwithoutcommonnoise} implies that $\Lambda$ maximizes $J$ over all such processes $\beta$, which implies that $\Lambda_t(\omega)$ must equal $\delta_{\alpha^*_t(\omega)}$ on the $(t,\omega)$-set $\{\alpha^* \neq 0\}$, where
\begin{align*}
\alpha^*_t &:= \text{sign}\left(\E\left[\left.\bar{\mu}^x_2 \right| \F^\mu_t\right]\right), 
\end{align*}
and we use the convention $\text{sign}(0) := 0$.

\begin{remark} \label{re:optimalcontrolselection}
This already highlights the key point: When $\mu$ is deterministic, an optimal control is the constant $\text{sign}(\bar{\mu}^x_2)$, but when $\mu$ is random, this control is inadmissible since it is not adapted.
\end{remark}

\begin{proposition}
Every strong MFG solution (without common noise) satisfies $\bar{\mu}^x_2 \in \{-2,0,2\}$ and $\bar{\mu}^x_t = t\,\mathrm{sign}(\bar{\mu}^x_2)$.
\end{proposition}
\begin{proof}
Let $(\widetilde{\Omega},(\F_t)_{t \in [0,2]},P,W,\mu,\Lambda,X)$ satisfy Definition \ref{def:mfgsolutionwithoutcommonnoise}, with $\mu$ deterministic. In this case, $\alpha^*_t = \text{sign}(\bar{\mu}^x_2)$ for all $t$. Suppose that $\bar{\mu}^x_2 \neq 0$. Then $\Lambda_t = \delta_{\alpha^*_t}$ must hold $dt \otimes dP$-a.e., and thus
\[
X_t = t\,\text{sign}(\bar{\mu}^x_2) + \sigma W_t, \ t \in [0,2].
\]
The consistency condition (6) of Definition \ref{def:mfgsolutionwithoutcommonnoise} implies $\bar{\mu}^x_t = \E[X_t] = t\,\text{sign}(\bar{\mu}^x_2)$. In particular, $\bar{\mu}^x_2 = 2\,\text{sign}(\bar{\mu}^x_2)$, which implies $\bar{\mu}^x_2 = \pm 2$ since we assumed $\bar{\mu}^x_2 \neq 0$.
\end{proof}

\begin{proposition} \label{pr:exampleweaksolution}
There exists a weak MFG solution (without common noise) satisfying $P(\bar{\mu}^x_2 = 1) = P(\bar{\mu}^x_2 = -1) = 1/2$.
\end{proposition}
\begin{proof}
Construct on some probability space $(\widetilde{\Omega},\F,P)$ a random variable $\gamma$ with $P(\gamma = 1) = P(\gamma = -1) = 1/2$ and an independent Wiener process $W$. Let $\alpha^*_t = \gamma1_{(1,2]}(t)$ for each $t$ (noticing that this interval is open on the left), and define $(\F_t)_{t \in [0,2]}$ to be the complete filtration generated by $(W_t,\alpha^*_t)_{t \in [0,2]}$.
Let
\[
X_t := \int_0^t\alpha^*_sds + \sigma W_t = (t-1)\gamma1_{(1,2]}(t) + \sigma W_t, \ t \in [0,2].
\]
Finally, let $\Lambda = dt\delta_{\alpha^*_t}(da)$, and define $\mu := P((W,\Lambda,X) \in \cdot \ | \ \gamma)$. Clearly $\mu$ is $\gamma$-measurable. On the other hand, independence of $\gamma$ and $W$ implies
\[
\bar{\mu}^x_2 = \E[X_2 \ | \ \gamma] = \gamma.
\]
Thus $\gamma$ is also $\mu$-measurable, and we conclude that $\mu := P((W,\Lambda,X) \in \cdot \ | \ \mu)$. It is straightforward to check that
\[
\F^\mu_t = \begin{cases}
\{\emptyset, \widetilde{\Omega}\} &\text{if } t \le 1 \\
\sigma(\gamma) &\text{if } 1 < t \le 2
\end{cases}.
\]
Thus
\[
\E[\bar{\mu}^x_2 \ | \ \F^\mu_t] = \begin{cases}
\E[\gamma] = 0 &\text{if } t \le 1 \\
\E[\gamma \ | \ \gamma] = \gamma &\text{if } 1 < t \le 2
\end{cases}.
\]
Since $\bar{\mu}^x_2 = \gamma = \text{sign}(\gamma)$, we conclude that $\alpha^*_t = \text{sign}(\E[\bar{\mu}^x_2 \ | \ \F^\mu_t])$. It is then readly checked using the previous arguments that $(\widetilde{\Omega},(\F_t)_{t \in [0,2]},P,W,\mu,\Lambda,X)$ is a weak MFG solution.
\end{proof}

To be absolutely clear, the above two propositions imply the following: If $S := \{\nu \in \P(\X) : \bar{\nu}^x_2 \in \{-2,0,2\}\}$, then every strong MFG solution lies in $S$, but there exists a weak MFG solution with $P(\mu \in S) = 0$.

\begin{remark}
The example of Proposition \ref{pr:exampleweaksolution} can be modified to illustrate another strange phenomenon. The proof of Proposition \ref{pr:exampleweaksolution} has $\alpha^*_t = \gamma$ for $t \in (1,2]$ and $\alpha^*_t=0$ for $t \le 1$. Instead, we could set $\alpha^*_t = \eta_t$ for $t \le 1$, for any mean-zero $[-1,1]$-valued process $(\eta_t)_{t \in [0,1]}$ independent of $\gamma$ and $W$. The rest of the proof proceeds unchanged, yielding another weak MFG solution with the same $\C^d$-marginal $\mu^x$. The difference is in the control as well as the \emph{joint} distribution $\mu = P((W,\Lambda,X) \in \cdot \ | \ \gamma)$. (In fact, we could even choose $\alpha^*$ to be any mean-zero \emph{relaxed} control on the time interval $[0,1]$.) Intuitively, for $t \le 1$ we have $\E[\bar{\mu}^x_2 \ | \ \F^\mu_t] = 0$, and the choice of control on the time interval $[0,1]$ does not matter in light of \eqref{ex:jdef}; the agent then has some freedom to \emph{randomize} her choice of control among the family of non-unique optimal choices. This type of randomization can typically occur when optimal controls are non-unique, and although it is unnatural in some sense, our main results indicate that this behavior can indeed arise in the limit from the finite-player games.
\end{remark}

\subsection{Supports of weak solutions} \label{se:supportsofweaksolutions}
In this section, we attempt to partially explain what permits the existence of weak solutions which are not randomizations among strong solutions. As was mentioned in Remark \ref{re:optimalcontrolselection}, the culprit is the adaptedness required of controls. Indeed, in the example of Section \ref{se:example}, very different optimal controls arise depending on whether or not the measure $\mu$ is random. If $\mu$ is deterministic, then so is the optimal control, and we may write this optimal control as a functional of $\mu$ by
\[
\hat{\alpha}^D(t,\mu) = \text{sign}(\bar{\mu}^x_T), \ t \in [0,T].
\]
The problem is as follows: for each fixed deterministic $\mu$, the optimal control $(\hat{\alpha}^D(t,\mu))_{t \in [0,T]}$ is deterministic and thus trivially adapted, but when $\mu$ is allowed to be random then this control is no longer adapted and thus no longer admissible. 
If, for a different MFG problem, it happens that $\hat{\alpha}^D$ is in fact progressively measurable with respect to $(\F^\mu_t)_{t \in [0,T]}$, then this control is still admissible when $\mu$ is randomized; moreover, it should be \emph{optimal} when $\mu$ is randomized, since it was optimal for each realization of $\mu$.
The following results make this idea precise, but first some terminology will be useful. As usual we work under assumption \ref{assumption:A} at all times, and the initial state distribution $\lambda \in \P^{p'}(\R^d)$ is fixed.

\begin{definition} \label{def:universallyadmissible}
We say that a function $\hat{\alpha} : [0,T] \times \C^m \times \C^d \times \P^p(\X) \rightarrow A$ is a \emph{universally admissible control} if:
\begin{enumerate}
\item $\hat{\alpha}$ is progressively measurable with respect to the (universal completion of the) natural filtration $(\F^{W,X,\mu}_t)_{t \in [0,T]}$ on $\C^m \times \C^d \times \P^p(\X)$. Here $\F^{W,X,\mu}_t := \sigma(W_s,X_s,\mu(C) : s \le t, \ C \in \F^\X_t)$ for each $t$, where $(W,X,\mu)$ denotes the identity map on $\C^m \times \C^d \times \P^p(\X)$.
\item For each fixed $\nu \in \P^p(\X)$, the SDE
\begin{align}
dX_t = b(t,X_t,\nu^x_t,\hat{\alpha}(t,W,X,\nu))dt + \sigma(t,X_t,\nu^x_t)dW_t, \ X_0 \sim \lambda, \label{def:fixednusde}
\end{align}
is unique in joint law; that is, if we are given two pairs of processes $(W^i_t,X^i_t)_{t \in [0,T]}$ for $i=1,2$, possibly on different filtered probability spaces but with $(W^i_t)_{t \in [0,T]}$ a Wiener process in either case, then $(W^1,X^1)$ and $(W^2,X^2)$ have the same law.
\item Suppose we are given a filtered probability space $(\widetilde{\Omega},(\widetilde{\F}_t)_{t \in [0,T]},\widetilde{P})$ supporting an $(\widetilde{\F}_t)_{t \in [0,T]}$-Wiener process $\widetilde{W}$, an $\widetilde{\F}_0$-measurable $\R^d$-valued random variable $\widetilde{\xi}$ with law $\lambda$, and a $\P^p(\X)$-valued random variable $\tilde{\mu}$ independent of $(\xi,W)$ such that $\tilde{\mu}(C)$ is $\widetilde{\F}_t$-measurable for each $C \in \F^\X_t$ and $t \in [0,T]$. Then there exists a strong solution $\widetilde{X}$ of the SDE
\[
d\widetilde{X}_t = b(t,\widetilde{X}_t,\tilde{\mu}^x_t,\hat{\alpha}(t,W,\widetilde{X},\tilde{\mu}))dt + \sigma(t,\widetilde{X}_t,\tilde{\mu}^x_t)d\widetilde{W}_t, \ \widetilde{X}_0 = \widetilde{\xi},
\]
and it satisfies $\E\int_0^T|\hat{\alpha}(t,W,\widetilde{X},\tilde{\mu})|^pdt < \infty$.
\end{enumerate}
If $\hat{\alpha}$ is a universally admissible control, we say it is \emph{locally optimal} if for each fixed $\nu \in \P^p(\X)$ there exists a complete filtered probability space $(\Omega^{(\nu)},(\F^{(\nu)}_t)_{t \in [0,T]},P^\nu)$ supporting a Wiener process $W^\nu$ and a continuous adapted process $X^\nu$ such that $(W^\nu,X^\nu)$ satisfies the SDE \eqref{def:fixednusde} and:
\begin{enumerate}
\item[(4)] If $(\widetilde{\Omega},(\F_t)_{t \in [0,T]},P)$ supports a $m$-dimensional Wiener process $W$, a progressive $\P(A)$-valued process $\Lambda$, and a continuous adapted $\R^d$-valued process $X$ satisfying
\[
dX_t = \int_Ab(t,X_t,\nu^x_t,a)\Lambda_t(da)dt + \sigma(t,X_t,\nu^x_t)dW_t, \ P \circ X_0^{-1} = \lambda,
\]
then
\begin{align*}
\E^{P^{(\nu)}}\left[\Gamma\left(\nu^x,dt\delta_{\hat{\alpha}(t,W^\nu,X^\nu,\nu)}(da),X^\nu\right)\right] \ge \E^P\left[\Gamma(\nu^x,\Lambda,X)\right].
\end{align*}
\end{enumerate}
\end{definition}

We need an additional assumption \ref{assumption:C}, which simply requires the uniqueness of the optimal controls. Some simple conditions are given in \cite[Proposition 4.4]{carmonadelaruelacker-mfgcommonnoise} under which assumption \ref{assumption:C} holds: in particular, it suffices to assume that $b$ and $\sigma$ are affine in $(x,a)$, that $f$ is strictly concave in $(x,a)$, and that $g$ is concave in $x$.

\begin{assumption}{\textbf{C}} \label{assumption:C}
If $(\widetilde{\Omega}^i,(\F^i_t)_{t \in [0,T]},P^i,W^i,\mu^i,\Lambda^i,X^i)$ for $i=1,2$ both satisfy (1-5) of Definition \ref{def:mfgsolutionwithoutcommonnoise} as well as $P^1 \circ (\mu^1)^{-1} = P^2 \circ (\mu^2)^{-1}$, then $P^1 \circ (W^1,\mu^1,\Lambda^1,X^1)^{-1} = P^2 \circ (W^2,\mu^2,\Lambda^2,X^2)^{-1}$.
\end{assumption}

\begin{theorem} \label{th:randomizationdeterministicconverse}
Assume \ref{assumption:C} holds. Suppose that there exists a universally admissible and locally optimal control $\hat{\alpha} : [0,T] \times \C^m \times \C^d \times \P^p(\X) \rightarrow A$.
Then, for every weak MFG solution $(\widetilde{\Omega},(\F_t)_{t \in [0,T]},P,W,\mu,\Lambda,X)$ (without common noise), $P \circ \mu^{-1}$ is concentrated on the set of strong MFG solutions (without common noise). Conversely, if $\rho \in \P^p(\P^p(\X))$ is concentrated on the set of strong MFG solutions (without common noise), then there exists a weak MFG solution (without common noise) with $P \circ \mu^{-1} = \rho$.
\end{theorem}
\begin{proof} Let $(\widetilde{\Omega},(\F_t)_{t \in [0,T]},P,W,\mu,\Lambda,X)$ be a weak MFG solution (without common noise).

\noindent \textit{Step 1:}
We will first show that necessarily $\Lambda_t = \delta_{\hat{\alpha}(t,W,X,\mu)}$ holds $dt \otimes dP$-a.e. On $(\widetilde{\Omega},(\F_t)_{t \in [0,T]},P)$ we may use (3) of Definition \ref{def:universallyadmissible} to find a strong solution $X'$ of the SDE
\[
dX'_t = b(t,X'_t,\mu^x_t,\hat{\alpha}(t,W,X',\mu))dt + \sigma(t,X'_t,\mu^x_t)dW_t, \ X'_0 = X_0,
\]
wth $\E^P\int_0^T|\hat{\alpha}(t,W,X',\mu)|^pdt < \infty$.
In particular, $X'$ is adapted to the (completion of the) filtration $\F^{X_0,W,\mu}_t := \sigma(X_0,W_s,\mu(C) : s\le t, \ C \in \F^\X_t)$.
Let $\Lambda' := dt\delta_{\hat{\alpha}(t,W,X',\mu)}(da)$. Then it is clear that $(\widetilde{\Omega},(\F^{X_0,W,\mu}_t)_{t \in [0,T]},P,W,\mu,\Lambda',X')$ satisfies conditions (1-4) of Definition \ref{def:mfgsolutionwithoutcommonnoise}. Optimality of $P$ implies
\[
\E^P\left[\Gamma(\mu^x,\Lambda,X)\right] \ge \E^P\left[\Gamma(\mu^x,\Lambda',X')\right].
\]
On the other hand, for $P \circ \mu^{-1}$-a.e. $\nu \in \P^p(\X)$, the following hold under $P(\cdot \ | \ \mu=\nu)$:
\begin{itemize}
\item $W$ is a $(\F_t)_{t \in [0,T]}$-Wiener process.
\item $(W,\Lambda,X)$ satisfies
\[
dX_t = \int_Ab(t,X_t,\nu^x_t,a)\Lambda_t(da) + \sigma(t,X_t,\nu^x_t)dW_t.
\]
\item $(W,X')$ solves the SDE \eqref{def:fixednusde}.
\end{itemize}
From the local optimality of $\hat{\alpha}$ we conclude (keeping in mind the uniqueness condition (2) of Definition \ref{def:universallyadmissible}) that
\[
\E^P\left[\left.\Gamma(\mu^x,\Lambda,X)\right| \mu \right] \le \E^P\left[\left.\Gamma(\mu^x,\Lambda',X')\right| \mu \right].
\]
Thus 
\[
\E^P\left[\Gamma(\mu^x,\Lambda,X)\right] = \E^P\left[\Gamma(\mu^x,\Lambda',X')\right].
\]
By assumption \ref{assumption:C}, there is only one optimal control, and so $\Lambda = \Lambda' = dt\delta_{\hat{\alpha}(t,W,X',\mu)}(da)$, $P$-a.s. From uniqueness of the SDE solutions we conclude that $X=X'$ a.s. as well, completing the first step. (Note we do not use the assumptions of Definition \ref{def:universallyadmissible} for this last conclusion, but only the Lipschitz assumption (A.4).)

\noindent \textit{Step 2:} Next, we show that $P \circ \mu^{-1}$ is concentrated on the set of strong MFG solutions.
Using (2) and (3) of Definition \ref{def:universallyadmissible}, we know that for $P \circ \mu^{-1}$-a.e. $\nu \in \P^p(\X)$ there exists on some filtered probability space $(\Omega^{(\nu)},(\F^{(\nu)}_t)_{t \in [0,T]},P^\nu)$ a weak solution $X^\nu$ of the SDE
\[
dX^\nu_t = b(t,X^\nu_t,\nu^x_t,\hat{\alpha}(t,W^\nu,X^\nu,\nu))dt + \sigma(t,X^\nu_t,\nu^x_t)dW^\nu_t, \ P^\nu \circ (X^\nu_0)^{-1} = \lambda,
\] 
where $W^\nu$ is an $(\F^{(\nu)}_t)_{t \in [0,T]}$-Wiener process. From Step 1, on $(\widetilde{\Omega},(\F_t)_{t \in [0,T]},P)$ we have
\[
dX_t = b(t,X_t,\mu^x_t,\hat{\alpha}(t,W,X,\mu))dt + \sigma(t,X_t,\mu^x_t)dW_t, \ P \circ X_0^{-1} = \lambda.
\]
It follows from the $P$-independence of $\mu$, $X_0$, and $W$ along with the uniqueness in law of condition (2) of Definition \ref{def:universallyadmissible} that
\begin{align}
P((W,\Lambda,X) \in \cdot \ | \ \mu = \nu) = P^\nu \circ \left(W^\nu,dt\delta_{\hat{\alpha}(t,W^\nu,X^\nu,\nu)}(da),X^\nu\right)^{-1}, \label{pf:randomizationdeterministic1}
\end{align}
for $P \circ \mu^{-1}$-a.e. $\nu \in \P^p(\X)$. 
Since $\mu = P((W,\Lambda,X) \in \cdot \ | \ \mu)$, it follows that
\begin{align}
\nu = P^\nu \circ \left(W^\nu,dt\delta_{\hat{\alpha}(t,W^\nu,X^\nu,\nu)}(da),X^\nu\right)^{-1}, \text{ for } P \circ \mu^{-1}\text{-a.e. } \nu \in \P^p(\X). \label{pf:randomizationdeterministic2}
\end{align}
We conclude that $P \circ \mu^{-1}$-a.e. $\nu \in \P^p(\X)$ is a strong MFG solution, or more precisely that
\[
(\Omega^{(\nu)},(\F^{(\nu)}_t)_{t \in [0,T]},P^\nu,W^\nu,\nu,dt\delta_{\hat{\alpha}(t,W^\nu,X^\nu,\nu)}(da),X^\nu)
\]
is a strong MFG solution. Indeed, we just verified condition (6) of Definition \ref{def:mfgsolutionwithoutcommonnoise}, and conditions (1-4) are obvious. The optimality condition (5) of Definition \ref{def:mfgsolutionwithoutcommonnoise} is a simple consequence of the local optimality of $\hat{\alpha}$

\noindent \textit{Step 3:}
We turn now to the converse. Let $(\widetilde{\Omega},\F,P)$ be any probability space supporting a random variable $(\xi,W,\mu)$ with values in $\R^d \times \C^m \times \P^p(\X)$ with law $\lambda \times \W^m \times \rho$, where $\W^m$ is Wiener measure on $\C^m$.
Let $(\F_t)_{t \in [0,T]}$ denote the $P$-completion of $(\sigma(\xi,W_s,\mu(C) : s \le t, \ C \in \F^\X_t))_{t \in [0,T]}$.
Solve strongly on $(\widetilde{\Omega},(\F_t)_{t \in [0,T]},P)$ the SDE
\[
dX_t = b(t,X_t,\mu^x_t,\hat{\alpha}(t,W,X,\mu))dt + \sigma(t,X_t,\mu^x_t)dW_t, \ X_0 = \xi.
\]
Note that hypothesis (3) makes this possible. 
Define $\Lambda := dt\delta_{\hat{\alpha}(t,W,X,\mu)}(da)$. 
Clearly $P \circ \mu^{-1} = \rho$ by construction, and we claim that $(\widetilde{\Omega},(\F_t)_{t \in [0,T]},P,W,\mu,\Lambda,X)$ is a weak MFG solution. Using hypothesis (1), it is clear that conditions (1-4) of Definition \ref{def:mfgsolutionwithoutcommonnoise} hold, and thus we must only check the optimality condition (5) and the fixed point condition (6).

First, let $(\widetilde{\Omega}',(\F'_t)_{t \in [0,T]},P',W',\mu',\Lambda',X')$ be an alternative probability space satisfying (1-4) of Definition \ref{def:mfgsolutionwithoutcommonnoise} and $P' \circ (\mu')^{-1} = P \circ \mu^{-1} = \rho$. The uniqueness in law condition (2) of Definition \ref{def:universallyadmissible} implies that $P(((W,X) \in \cdot \ | \ \mu = \nu)$ is exactly the law of the solution of the SDE \eqref{def:fixednusde}, for $P \circ \mu^{-1}$-a.e. $\nu$. Applying local optimality of $\hat{\alpha}$ for each $\nu$, we conclude that
\[
\E^P\left[\left.\Gamma(\nu^x,\Lambda,X)\right| \mu = \nu\right] \ge \E^{P'}\left[\left.\Gamma(\nu^x,\Lambda',X')\right| \mu = \nu\right], \text{for } \rho-a.e. \ \nu.
\]
Integrate with respect to $\rho$ on both sides to get $\E^P[\Gamma(\mu^x,\Lambda,X)] \ge \E^{P'}[\Gamma((\mu')^x,\Lambda',X')]$, which verifies condition (5) of Definition \ref{def:mfgsolutionwithoutcommonnoise}. Finally, we check (6) by applying Step 1 to deterministic $\mu$ and again using uniqueness of the SDE \eqref{def:fixednusde} to find that both \eqref{pf:randomizationdeterministic1} and \eqref{pf:randomizationdeterministic2} hold for $\rho$-a.e. $\nu$.
\end{proof}

\subsection{Applications of Theorem \ref{th:randomizationdeterministicconverse}}
It is admittedly quite difficult to check that there exists a universally admissible, locally optimal control, and we will leave this problem open in all but the simplest cases. Note, however, that conditions (2) and (3) of Definition \ref{def:universallyadmissible} hold automatically when $\hat{\alpha}(t,w,x,\nu) = \hat{\alpha}'(t,w,x_0,\nu)$, for some $\hat{\alpha}' : [0,T] \times \C^m \times \R^d \times \P^p(\X) \rightarrow A$.

\subsubsection*{A simple class of examples}
Suppose $A \subset \R^k$ is convex, $g \equiv 0$, and $f=f(t,\mu,a)$ is twice differentiable in $a$ with uniformly negative Hessian in $a$. That is, $D_a^2f(t,\mu,a) \le -\delta$ for all $(t,\mu)$, for some $\delta > 0$. Suppose as usual that assumption \ref{assumption:A} holds. Define
\[
\hat{\alpha}(t,w,x,\nu) := \arg\min_{a \in A}f(t,\nu^x_t,a), \text{ for } (t,w,x,\mu) \in [0,T] \times \C^m \times \C^d \times \P^p(\X).
\]
It is straightforward to check that assumption \ref{assumption:C} holds and that $\hat{\alpha}$ is a universally admissible and locally optimal control. Of course, this example is simple in that the state process does not influence the optimization.

\subsubsection*{A possible general strategy}
The following approach may be more widely applicable. First, for a fixed $\nu \in \P^p(\X)$, we may define the value function $V[\nu](t,x)$ of the corresponding optimal control problem in the usual way, and it should solve a Hamilton-Jacobi-Bellman (HJB) PDE of the form
\[
\begin{cases}
-\partial_tV[\nu](t,x) - H(t,x,\nu^x_t,D_xV[\nu](t,x),D^2_xV[\nu](t,x)) &= 0, \text{ on } [0,T) \times \R^d, \\
\quad\quad\quad\quad\quad\quad\quad\quad\quad\quad\quad\quad\quad\quad\quad\quad\quad\quad\quad\!  V[\nu](T,x) &= g(x,\nu^x_T)
\end{cases},
\]
where the Hamiltonian $H : [0,T] \times \R^d \times \P^p(\R^d) \times \R^d \times \R^{d \times d} \rightarrow \R$ is defined by
\[
H(t,x,\mu,y,z) := \sup_{a \in A}\left[y^\top b(t,x,\mu,a) + f(t,x,\mu,a)\right] + \frac{1}{2}\mathrm{Tr}\left[z\sigma\sigma^\top(t,x,\mu)\right].
\]
Suppose that we can show (as is well known to be possible in very general situations) that for each $\nu$ the value function $V[\nu]$ is the unique (viscosity) solution of this HJB equation. Then, an optimal control can be obtained by finding $\hat{\alpha}(t,x_t,\nu)$ which achieves the supremum in
\[
H(t,x_t,\nu^x_t,D_xV[\nu](t,x_t),D_x^2V[\nu](t,x_t)), 
\]
for each $(t,x,\nu)$. The crux of this approach is to show that the value function $V[\nu](t,x)$ is \emph{adapted with respect to $\nu$} in some sense, which would imply that $\hat{\alpha}$ is universally admissible and locally optimal. A nice special case would be a Markovian dependence, $V[\nu](t,x) = \widetilde{V}(t,x,\nu^x_t)$. In short, we must study the dependence of a family of HJB equations on a path-valued parameter.

\section{Mean field games on a canonical space} \label{se:canonicalspace}
In this section, we begin to work toward the proofs of the main results announced in Sections \ref{se:mainresults} and \ref{se:converse}. This section briefly elaborates on the notion of mean field game solution on the canonical space, in order to state simpler conditions by which may check that a measure $P \in \P(\Omega)$ is a weak MFG solution, in the sense of Definition \ref{def:weakMFGsolutionlaw}. The definitions and notations of this section are again mostly borrowed from \cite{carmonadelaruelacker-mfgcommonnoise}, to which the reader is referred for more details.

First, we mention some notational conventions. We will routinely use the same letter $\phi$ to denote the natural extension of a function $\phi : E \rightarrow F$ to any product space $E \times E'$, given by $\phi(x,y) := \phi(x)$ for $(x,y) \in E \times E'$. Similarly, we will use the same symbol $(\F_t)_{t \in [0,T]}$ to denote the natural extension of a filtration $(\F_t)_{t \in [0,T]}$ on a space $E$ to any product space $E \times E'$, given by $(\F_t \otimes \{\emptyset, E'\})_{t \in [0,T]}$.

We will make use of the following canonical spaces, two of which have been defined already but are recalled for convenience:
\begin{align*}
\X &:= \C^m \times \V \times \C^d, \quad \Omega_0 := \R^d \times \C^{m_0} \times \C^m, \quad \Omega := \Omega_0 \times \P^p(\X) \times \V \times \C^d.
\end{align*}
From now on, let $\xi$, $B$, $W$, $\mu$, $\Lambda$, and $X$ denote the identity maps on $\R^d$, $\C^{m_0}$, $\C^m$, $\P^p(\X)$, $\V$, and $\C^d$, respectively. Note, for example, that our convention permits $W$ to denote both the identity map on $\C^m$ and the projection from $\Omega$ to $\C^m$. The canonical filtration $(\F^\Lambda_t)_{t \in [0,T]}$ is defined on $\V$ by
\[
\F^\Lambda_t := \sigma\left(\Lambda([0,s] \times C) : s \le t, \ C \in \B(A)\right).
\]
It is known (e.g. \cite[Lemma 3.8]{lacker-mfgcontrolledmartingaleproblems}) that there exists a $(\F^\Lambda_t)_{t \in [0,T]}$-predictable process $\overline{\Lambda} : [0,T] \times \V \rightarrow \P(A)$ such that $dt[\overline{\Lambda}(t,q)](da) = q$ for each $q$, or equivalently $\overline{\Lambda}(t,q) = q_t$ for a.e. $t$, for each $q$. Then, we may think of $(\overline{\Lambda}(t,\cdot))_{t \in [0,T]}$ as the canonical $\P(A)$-valued process defined on $\V$, and it is clear that $\F^\Lambda_t = \sigma(\overline{\Lambda}(s,\cdot) : s \le t)$. With this in mind, we may somewhat abusively write $\Lambda_t$ in place of $\overline{\Lambda}(t,\cdot)$, and with this notation $\F^\Lambda_t = \sigma(\Lambda_s : s \le t)$.

The canonical processes $B$, $W$, and $X$ generate obvious natural filtrations, denoted $(\F^B_t)_{t \in [0,T]}$, $(\F^W_t)_{t \in [0,T]}$, and $(\F^X_t)_{t \in [0,T]}$, respectively. We will frequently work with filtrations generated by several canonical processes, such as $\F^{\xi,B,W}_t := \sigma(\xi,B_s,W_s : s \le t)$ defined on $\Omega_0$, and $\F^{\xi,B,W,\Lambda}_t = \F^{\xi,B,W}_t \otimes \F^{\Lambda}_t$ defined on $\Omega_0 \times \V$. Our convention on canonical extensions of filtrations to product spaces permits the use of $(\F^{\xi,B,W}_t)_{t \in [0,T]}$ to refer also to the filtration on $\Omega_0 \times \V$ generated by $(\xi,B,W)$, and it should be clear from context on which space the filtration is defined. Hence, the filtration $(\F^\X_t)_{t \in [0,T]}$ defined just before  Definition \ref{def:weakMFGsolution} could alternatively be denoted $\F^\X_t = \F^{W,\Lambda,X}_t$, but we stick with the former notation for consistency. Define the canonical filtration $(\F^\mu_t)_{t \in [0,T]}$ on $\P^p(\X)$ by
\[
\F^\mu_t := \sigma\left(\mu(C) : C \in \F^\X_t\right).
\]

There is somewhat of a conflict in notation, between our use of $(\xi,B,W)$ here as the identity map on $\R^d \times \C^{m_0} \times \C^m$ and our previous use (beginning in Section \ref{se:finiteplayergames}) of the same letters for random variables with values in $(\R^d)^n \times \C^{m_0} \times (\C^m)^n$, defined on an $n$-player environment $\mathcal{E}_n = (\Omega_n,(\F^n_t)_{t \in [0,T]},\PP_n,\xi,B,W)$. However, we will almost exclusively discuss the random variables $(\xi,B,W)$ through the lenses of various probability measures, and thus it should be clear from context (i.e. from the nearest notated probability measure) which random variables $(\xi,B,W)$ we are working with at any given moment. For example, given $P \in \P(\Omega)$, the notation $P \circ (\xi,B,W)^{-1}$ refers to a measure on $\R^d \times \C^{m_0} \times \C^m$. On the other hand, $\PP_n$ is reserved for the measure on $\Omega_n$ in a typical $n$-player environment, and so $\PP_n \circ (\xi,B,W)^{-1}$ refers to a measure on $(\R^d)^n \times \C^{m_0} \times (\C^m)^n$.

Recall that the initial state distribution $\lambda \in \P^{p'}(\R^d)$ is fixed throughout. Let $\M_\lambda$ denote the set of $\rho \in \P^p(\Omega_0 \times \P^p(\X))$ satisfying
\begin{enumerate}
\item $\rho \circ \xi^{-1} = \lambda$,
\item $B$ and $W$ are independent Wiener processes on $(\Omega_0 \times \P^p(\X),(\F^{\xi,B,W,\mu}_t)_{t \in [0,T]},\rho)$. 
\end{enumerate}
(Note that the set $\M_\lambda$ was denoted $\P^p_c[(\Omega_0,\W_\lambda) \leadsto \P^p(\X)]$ in \cite{carmonadelaruelacker-mfgcommonnoise}; we prefer this shorter notation mainly because we will make no use of it after this section.)
For $\rho \in \M_\lambda$, the class $\A(\rho)$ of admissible controls is the set of probability measures $Q$ on $\Omega_0 \times \P^p(\X) \times \V$ satisfying:
\begin{enumerate}
\item $\F^\Lambda_t$ and $\F^{\xi,B,W,\mu}_T$ are conditionally independent under $Q$ given $\F^{\xi,B,W,\mu}_t$, for each $t \in [0,T]$,
\item $Q \circ (\xi,B,W,\mu)^{-1} = \rho$,
\item $\E^Q\int_0^T\int_A|a|^p\Lambda_t(da)dt < \infty$.
\end{enumerate}
We say $Q \in \A(\rho)$ is a \emph{strict control} if there exists an $A$-valued process $(\alpha_t)_{t \in [0,T]}$, progressively measurable with respect to the $Q$-completion of $(\F^{\xi,B,W,\mu,\Lambda}_t)_{t \in [0,T]}$, such that
\begin{align*}
Q\left(\Lambda = dt\delta_{\alpha_t}(da)\right) = Q(\Lambda_t = \delta_{\alpha_t} \ a.e. \ t) = 1. 
\end{align*}
We say $Q \in \A(\rho)$ is a \emph{strong control} if the above holds but with $(\alpha_t)_{t \in [0,T]}$ progressively measurable with respect to the $Q$-completion of $(\F^{\xi,B,W,\mu}_t)_{t \in [0,T]}$.

If $\rho \in \M_\lambda$ and $Q \in \A(\rho)$, note that $B$ and $W$ are Wiener processes on $(\Omega_0 \times \P^p(\X) \times \V,(\F^{\xi,B,W,\mu,\Lambda}_t)_{t \in [0,T]},Q)$.
For each $\rho \in \M_\lambda$ and $Q \in \A(\rho)$, on the completion of the filtered probability space $(\Omega_0 \times \P^p(\X) \times \V, (\F^{\xi,B,W,\mu,\Lambda}_t)_{t \in [0,T]},Q)$, there exists a unique strong solution $Y$ of the SDE
\begin{align}
Y_t = \xi + \int_0^t\int_Ab(s,Y_s,\mu^x_s,a)\Lambda_s(da)ds + \int_0^t\sigma(s,Y_s,\mu^x_s)dW_s + \int_0^t\sigma_0(s,Y_s,\mu^x_s)dB_s. \label{def:SDE}
\end{align}
Viewing $Y$ as a random element of $\C^d$, let $\RC(Q) := Q \circ (\xi,B,W,\mu,\Lambda,Y)^{-1} \in \P(\Omega)$ denote the joint law of the solution and the inputs. Define
\[
\RC\A(\rho) := \RC(\A(\rho)) = \left\{\RC(Q) : Q \in \A(\rho)\right\},
\]
which we think of as the set of admissible joint laws for the optimal control problem associated to $\rho$.
Alternatively, $\RC(Q)$ may be defined as the unique element $P$ of $\P(\Omega)$ such that $P \circ (\xi,B,W,\mu,\Lambda)^{-1} = Q$ and such that the canonical processes  $(\xi,B,W,\mu,\Lambda,X)$ verify the state SDE on $\Omega$:
\begin{align}
X_t = \xi + \int_0^t\int_Ab(s,X_s,\mu^x_s,a)\Lambda_s(da)ds + \int_0^t\sigma(s,X_s,\mu^x_s)dW_s + \int_0^t\sigma_0(s,X_s,\mu^x_s)dB_s. \label{def:canonicalSDE}
\end{align}
It follows from standard estimates (e.g. \cite[Lemma 2.4]{carmonadelaruelacker-mfgcommonnoise}) that $\RC(Q) \in \P^p(\Omega)$.

Recalling the definition of the objective functional $\Gamma$ from \eqref{def:gamma}, we define the reward associated to an element $P \in \P^p(\Omega)$ by
\[
J(P) := \E^P\left[\Gamma(\mu^x,\Lambda,X)\right].
\]
Define the set of optimal controls corresponding to $\rho$ by
\begin{align*}
\A^{*}(\rho) &:= \arg\max_{Q \in \A(\rho)}J(\RC(Q)),
\end{align*}
and note that
\begin{align*}
\RC\A^{*}(\rho) &:= \RC(\A^{*}(\rho)) = \arg\max_{P \in \RC\A(\rho)}J(P).
\end{align*}
Let us now adapt the definition of MFG solution to the canonical space $\Omega$:

\begin{definition}[MFG pre-solution] \label{def:mfgsolution}
We say $P \in \P(\Omega)$ is a \emph{MFG pre-solution} if it satisfies the following:
\begin{enumerate}
\item $\xi$, $W$, and $(B,\mu)$ are independent under $P$.
\item $P \in \RC\A(\rho)$ where $\rho := P \circ (\xi,B,W,\mu)^{-1}$ is in $\M_\lambda$.
\item $\mu = P((W,\Lambda,X) \in \cdot \ | \ B,\mu)$ a.s. That is, $\mu$ is a version of the conditional law of $(W,\Lambda,X)$ given $(B,\mu)$.
\end{enumerate}
\end{definition}

The following two Lemmas give us a characterization of MFG solution which is convenient for taking limits. The first is more or less obvious, stated as a Lemma merely for emphasis, while the second has more content and is discussed thoroughly in \cite{carmonadelaruelacker-mfgcommonnoise}.

\begin{lemma}[Lemma 3.9 of \cite{carmonadelaruelacker-mfgcommonnoise}]  \label{le:canonicalmfgsolution}
Let $P \in \P^p(\Omega)$, and define $\rho := P \circ (\xi,B,W,\mu)^{-1}$. If $P$ is an MFG pre-solution and $P \in \RC\A^{*}(\rho)$, then $P$ is a weak MFG solution in the sense of Definition \ref{def:weakMFGsolutionlaw}.
\end{lemma}

\begin{lemma}[Lemma 3.7 of \cite{carmonadelaruelacker-mfgcommonnoise}]  \label{le:presolution}
Let $P \in \P^p(\Omega)$, and define $\rho := P \circ (\xi,B,W,\mu)^{-1}$. Suppose the following hold under $P$:
\begin{enumerate}
\item $B$ and $W$ are independent $(\F^{\xi,B,W,\mu,\Lambda,X}_t)_{t \in [0,T]}$-Wiener processes, and $P \circ \xi^{-1} = \lambda$.
\item $\xi$, $W$, and $(B,\mu)$ are independent.
\item $\mu = P((W,\Lambda,X) \in \cdot \ | \ B,\mu), \ a.s.$
\item The canonical processes $(\xi,B,W,\mu,\Lambda,X)$ verify the state equation \eqref{def:canonicalSDE} on $\Omega$.
\end{enumerate}
Then $P$ is a MFG pre-solution.
\end{lemma}

We close the section with three useful results from \cite{carmonadelaruelacker-mfgcommonnoise}, topological in nature. They will not be used until the final step of the proof of Theorem \ref{th:mainconvergence}, in Section \ref{se:optimalityinthelimit}.

\begin{lemma}[Lemma 3.12 of \cite{carmonadelaruelacker-mfgcommonnoise}] \label{le:rccontinuous}
The map $\RC : \bigcup_{\rho \in \M_\lambda}\A(\rho) \rightarrow \P^p(\Omega)$ is continuous.
\end{lemma}

\begin{lemma} \label{le:jcontinuous}
The map $J : \P^p(\Omega) \rightarrow \R$ is upper semicontinuous, and for each $\rho \in \M_\lambda$ the sets $\A^{*}(\rho)$ and $\RC\A^{*}(\rho)$ are nonempty and compact. Moreover, the restriction of $J$ to a set $K \subset \P^p(\Omega)$ is continuous whenever $K$ satisfies the uniform integrability condition
\begin{align}
\lim_{r \rightarrow\infty}\sup_{P \in K}\E^P\left[\int_0^T\int_{\{|a| > r\}}|a|^{p'}\Lambda_t(da)dt\right] = 0. \label{def:uniformintegrability}
\end{align}
\end{lemma}
\begin{proof}
This is all covered by Lemma 3.13 of \cite{carmonadelaruelacker-mfgcommonnoise}, except for the final claim.
Now let $P_n \rightarrow P_\infty$ in $\P^p(\Omega)$ with $P_n \in K$ for each $n$. The continuity and growth assumptions on $g$ imply that $\E^{P_n}[g(X_T,\mu^x_T)] \rightarrow \E^P[g(X_T,\mu^x_T)]$, and the $f$ term causes the only problems.
The convergence $P_n \rightarrow P_\infty$ implies (e.g. by \cite[Theorem 7.12]{villanibook})
\begin{align}
\lim_{r \rightarrow\infty}\sup_{n}\E^{P_n}\left[\|X\|_T^p1_{\{\|X\|_T^p > r\}} + \int_{\C^d}\|z\|^p_T\mu^x(dz)1_{\left\{\int_{\C^d}\|z\|^p_T\mu^x(dz) > r\right\}}\right] = 0. \label{pf:uniformintegrability1}
\end{align}
For $1 \le n \le \infty$, define probability measures $Q_n$ on $\widetilde{\Omega} := [0,T] \times \R^d \times \P^p(\R^d) \times A$ by
\[
Q_n(C) := \frac{1}{T}\E^{P_n}\left[\int_0^T\int_A1_{\{(t,X_t,\mu^x_t,a) \in C\}}\Lambda_t(da)dt\right], \ C \in \B(\widetilde{\Omega}).
\]
Certainly $Q_n \rightarrow Q_\infty$ weakly in $\P(\widetilde{\Omega})$. Since the $[0,T]$-marginal is the same for each $Q_n$, it is known (e.g. \cite{jacodmemin-stable} or \cite[Lemma A.3]{lacker-mfgcontrolledmartingaleproblems}) that this implies $\int\phi\,dQ_n\rightarrow \int\phi\,dQ_\infty$ for each bounded measurable $\phi : \widetilde{\Omega} \rightarrow \R$ with $\phi(t,\cdot)$ continuous for each $t$. Thus $Q_n \circ f^{-1} \rightarrow Q_\infty \circ f^{-1}$ weakly in $\P(\R)$, by continuity of $f(t,\cdot)$ for each $t$. But it follows from \eqref{def:uniformintegrability}, \eqref{pf:uniformintegrability1}, and the growth assumption of (A.5) that 
\[
\lim_{r \rightarrow \infty}\sup_n\int_{\{|f| > r\}}f\,dQ_n = 0,
\]
and thus $\int f\,dQ_n \rightarrow \int f\,dQ_\infty$.
\end{proof}

The following definition highlights a useful subclass of admissible controls, which Lemma \ref{le:adapteddensity} shows is dense in the class of admissible controls in a sense.

\begin{definition} \label{def:adapted}
A function $\phi : \Omega_0 \times \P^p(\X) \rightarrow \V$ is said to be \emph{adapted} if $\phi^{-1}(C) \in \F^{\xi,B,W,\mu}_t$ for each $C \in \F^\Lambda_t$ and $t \in [0,T]$. We say $\phi$ is \emph{compact} if there exists a compact set $K \subset [0,T] \times A$ such that $\phi(\omega,\nu)(K^c) = 0$ for each $(\omega,\nu) \in \Omega_0 \times \P^p(\X)$. For $\rho \in \M_\lambda$, let $\A_a(\rho)$ denote the set of measures of the form
\[
\rho \circ (\xi,B,W,\mu,\phi(\xi,B,W,\mu))^{-1}
\]
where $\phi$ is adapted and compact and $\P^p(\X) \ni \nu \mapsto \phi(\omega,\nu)$ is continuous for each $\omega \in \Omega_0$.
\end{definition}

\begin{lemma} \label{le:adapteddensity}
For each $\rho \in \M_\lambda$, $\A_a(\rho)$ is a dense subset of $\A(\rho)$. Moreover, for each $P \in \RC\A(\rho)$ with $\E^P\int_0^T\int_A|a|^{p'}\Lambda_t(da)dt < \infty$,  
there exist $P_n \in \RC\A_a(\rho)$ such that $K := \{P_n : n \ge 1\}$ satisfies \eqref{def:uniformintegrability} and $P_n \rightarrow P$ in $\P^p(\Omega)$; in particular, $J(P_n) \rightarrow J(P)$.
\end{lemma}
\begin{proof}
Lemma 3.11 of \cite{carmonadelaruelacker-mfgcommonnoise} covers the first claim in the case that $A$ is bounded, while the general case is treated in the second step of the proof of Lemma 3.17 in \cite{carmonadelaruelacker-mfgcommonnoise}. Except for the claim that $K$ satisfies the uniform integrability condition \eqref{def:uniformintegrability}, the second statement is precisely Lemma 3.17 of \cite{carmonadelaruelacker-mfgcommonnoise}, the proof of which elucidates this uniform integrability.
\end{proof}

\section{Proof of Theorem \ref{th:mainconvergence}} \label{se:mainconvergenceproof}
With the mean field game concisely summarized on the canonical space, we now turn to the proof of Theorem \ref{th:mainconvergence}. Throughout the section, we work with the notation and assumptions of Theorem \ref{th:mainconvergence}.
Following Lemma \ref{le:canonicalmfgsolution}, the strategy is to prove the claimed relative compactness, then that any limit is a MFG pre-solution using Lemma \ref{le:presolution}, and then finally that any limit corresponds to an optimal control. First, we establish some useful estimates for the $n$-player systems.

\subsection{Estimates} 
The first estimate below, Lemma \ref{le:finitestateestimate}, is fairly standard, but it is important that it is independent of the number of agents $n$. The second estimate, Lemma \ref{le:valuebound}, will be used to establish some uniform integrability of the equilibrium controls, and it is precisely where we need the coercivity of the running cost $f$.
Note in the following proofs that the initial states $X^i_0[\Lambda] = X^i_0 = \xi^i$ and the initial empirical measure $\widehat{\mu}^x_0[\Lambda] = \widehat{\mu}^x_0 = \frac{1}{n}\sum_{i=1}^n\delta_{\xi^i}$ do not depend on the choice of control. Recall the definition of the truncated supremum norm \eqref{def:truncatedsupnorm}.

\begin{lemma} \label{le:finitestateestimate}
There exists a constant $c_5 \ge 1$, depending only on $p$, $p'$, $T$, and the constant $c_1$ of assumption (A.4) such that, for each $\gamma \in [p,p']$, $\beta = (\beta^1,\ldots,\beta^n) \in \A_n^n(\mathcal{E}_n)$, and $1 \le k \le n$,
\begin{align*}
\E^{\PP_n}[\|X^k[\beta]\|_T^\gamma] &\le c_5\E^{\PP_n}\left[1 + |\xi^1|^\gamma + \int_0^T\int_A|a|^\gamma\beta^k_t(da)dt + \frac{1}{n}\sum_{i=1}^n\int_0^T\int_A|a|^\gamma\beta^i_t(da)dt \right],
\end{align*}
and
\begin{align*}
\E^{\PP_n}\int_{\C^d}\|z\|_T^\gamma\widehat{\mu}^x[\beta](dz) = \frac{1}{n}\sum_{i=1}^n\E^{\PP_n}[\|X^i[\beta]\|_T^\gamma] &\le c_5\E^{\PP_n}\left[1 + |\xi^1|^\gamma + \frac{1}{n}\sum_{i=1}^n\int_0^T\int_A|a|^\gamma\beta^i_t(da)dt\right].
\end{align*}
\end{lemma}
\begin{proof}
We omit $[\beta]$ from the notation throughout the proof, as well as the superscript $\PP_n$ which should appear above the expectations. Abbreviate $\Sigma := \sigma\sigma^\top + \sigma_0\sigma_0^\top$.
Apply the Burkholder-Davis-Gundy inequality and assumption (A.4) to find a universal constant $C > 0$ (which will change from line to line) such that, for all $\gamma \in [p,p']$,
\begin{align*}
\E[\|X^k\|^\gamma_t] \le &\,C\E\left[ |\xi^k|^\gamma + \left(\int_0^t\int_A|b(s,X^k_s,\widehat{\mu}^x_s,a)|\beta^k_s(da)ds\right)^\gamma + \left(\int_0^t\left|\Sigma(s,X^k_s,\widehat{\mu}^x_s)\right|ds\right)^{\gamma/2}\right]   \\
	\le &\,C\,\E\left\{1 + |\xi^k|^\gamma + \int_0^t\left[\|X^k\|^\gamma_s + \left(\int_{\C^d}\|z\|^p_s\widehat{\mu}^x(dz)\right)^{\gamma/p} + \int_A|a|^\gamma\beta^k_s(da)\right] ds\right\} \\
	& + C\,\E\left\{\left[\int_0^t\left(\|X^k\|^{p_\sigma}_s + \left(\int_{\C^d}\|z\|^p_s\widehat{\mu}^x(dz)\right)^{p_\sigma/p}\right)ds\right]^{\gamma/2}\right\} \\
	\le &\,C\,\E\left\{1 + |\xi^k|^\gamma + \int_0^t\left[\|X^k\|^\gamma_s + \int_{\C^d}\|z\|^\gamma_s\widehat{\mu}^x(dz) + \int_A|a|^\gamma\beta^k_s(da)\right] ds \right\}.
\end{align*}
The last line follows from the bound $(\int\|z\|_s^p\nu(dz))^{\gamma/p} \le \int\|z\|_s^\gamma\nu(dz)$ for $\nu \in \P(\C^d)$, which holds because $\gamma \ge p$. To deal with the $\gamma/2$ outside of the time integral, we used the following argument. If $\gamma \ge 2$, we simply use Jensen's inequality to pass $\gamma/2$ inside of the time integral, and then use the inequality $|x|^{p_\sigma\gamma/2} \le 1 + |x|^\gamma$, which holds because $p_\sigma \le 2$. The other case is $1 \vee p_\sigma \le p \le \gamma < 2$, and we use then the inequalities $|x|^{\gamma/2} \le 1 + |x|$ and $|x|^{p_\sigma} \le 1 + |x|^\gamma$. By Gronwall's inequality,
\begin{align}
\E[\|X^k\|^\gamma_t] \le C\E\left\{1 + |\xi^k|^\gamma + \int_0^t\left[\int_{\C^d}\|z\|^\gamma_s\widehat{\mu}^x(dz) + \int_A|a|^\gamma\beta^k_s(da)\right]ds\right\} \label{pf:finitestateestimate1}
\end{align}
Note that $\E^{\PP_n}[|\xi^k|^\gamma] = \E^{\PP_n}[|\xi^1|^\gamma]$ for each $k$, and
average over $k=1,\ldots,n$ to get
\begin{align*}
\E\int_{\C^d}\|z\|^\gamma_t\widehat{\mu}^x(dz) &= \frac{1}{n}\sum_{i=1}^n\E[\|X^i\|^\gamma_t] \\
	&\le C\E\left\{1 + |\xi^1|^\gamma + \int_0^t\left[\int_{\C^d}\|z\|^\gamma_s\widehat{\mu}^x(dz) + \frac{1}{n}\sum_{i=1}^n\int_A|a|^\gamma\beta^i_s(da)\right]ds\right\}.
\end{align*}
Apply Gronwall's inequality once again to prove the second claimed inequality. The first claim follows from the second and from \eqref{pf:finitestateestimate1}.
\end{proof}

\begin{lemma} \label{le:valuebound}
There exist constants $c_6,c_7 > 0$, depending only $p$, $p'$, $T$, and the constants $c_1,c_2,c_3$ of assumption \ref{assumption:A}, such that for each $\beta = (\beta^1,\ldots,\beta^n) \in \A_n^n(\mathcal{E}_n)$, the following hold:
\begin{enumerate}
\item For each $1 \le k \le n$,
\begin{align*}
\E^{\PP_n}\int_0^T\int_A(|a|^{p'} - c_6|a|^p)\beta^k_t(da)dt \le c_7\E^{\PP_n}\left[ 1 + |\xi^1|^p + \frac{1}{n}\sum_{i \neq k}^n\int_0^T\int_A|a|^p\beta^i_t(da)\right] - c_7J_k(\beta).
\end{align*}
\item If for some  $n \ge k \ge 1$, $\epsilon > 0$, and $\widetilde{\beta}^k \in \A_n(\mathcal{E}_n)$ we have
\[
J_k(\widetilde{\beta}^k) \ge \sup_{\widetilde{\beta} \in \A_n(\mathcal{E}_n)}J_k((\beta^{-k},\widetilde{\beta})) - \epsilon,
\]
then
\begin{align*}
\E^{\PP_n}\int_0^T\int_A(|a|^{p'} - c_6|a|^p)\widetilde{\beta}^k_t(da)dt \le c_7\E^{\PP_n}\left[1 + \epsilon + |\xi^1|^p + \frac{1}{n}\sum_{i \neq k}^n\int_0^T\int_A|a|^p\beta^i_t(da)\right].
\end{align*}
\item If $\beta$ is an $\epsilon$-Nash equilibrium for some $\epsilon=(\epsilon_1,\ldots,\epsilon_n) \in [0,\infty)^n$, then
\begin{align*}
\frac{1}{n}\sum_{i=1}^n\E^{\PP_n}\int_0^T\int_A(|a|^{p'} - c_6|a|^p)\beta^i_t(da)dt \le c_7\left(1 + \E^{\PP_n}|\xi^1|^p + \frac{1}{n}\sum_{i=1}^n\epsilon_i\right).
\end{align*}
\end{enumerate}
\end{lemma}
\begin{proof}
Recall that $\E^{\PP_n}[|\xi^1|^p] < \infty$ and that every $\widetilde{\beta} \in \A_n(\mathcal{E}_n)$ is required to satisfy
\[
\E^{\PP_n}\int_0^T\int_A|a|^p\widetilde{\beta}_t(da)dt < \infty.
\]
Moreover, if $\E^{\PP_n}\int_0^T\int_A|a|^{p'}\widetilde{\beta}_t(da)dt = \infty$ then the upper bound of assumption (A.5) implies that $J_k((\beta^{-k},\widetilde{\beta})) = -\infty$, for each $\beta \in \A_n^n(\mathcal{E}_n)$ and $1 \le k \le n$.

\textit{Proof of (1):} First, use the upper bounds of $f$ and $g$ from assumption (A.5) to get
\begin{align*}
J_k(\beta) &\le c_2(T+1)\E^{\PP_n}\left[1 + \left\|X^k[\beta]\right\|_T^p + \int_{\C^d}\left\|z\right\|^p_T\widehat{\mu}^x[\beta](dz)\right] - c_3\E^{\PP_n}\int_0^T\int_A|a|^{p'}\beta^k_t(da)dt  \\
	&\le 3c_5c_2(T+1)\E^{\PP_n}\left[1 + |\xi^1|^p + \int_0^T\int_A|a|^p\beta^k_t(da)dt + \frac{1}{n}\sum_{i=1}^n\int_0^T\int_A|a|^p\beta^i_t(da)dt\right]  \\
	&\quad\quad  - c_3\E^{\PP_n}\int_0^T\int_A|a|^{p'}\beta^k_t(da)dt,
\end{align*}
where the last inequality follows from Lemma \ref{le:finitestateestimate} (and $c_5 \ge 1$). This proves the first claim, with $c_6 := 6c_5c_2(T+1)/c_3$ and $c_7 := c_6 \vee (1/c_3)$.

\textit{Proof of (2):} Fix $a_0 \in A$ arbitrarily. Abuse notation somewhat by writing $a_0$ in place of the constant strict control $(\delta_{a_0})_{t \in [0,T]} \in \A_n(\mathcal{E}_n)$. Lemma \ref{le:finitestateestimate} implies
\[
\E^{\PP_n}\left[\left\|X^k[(\beta^{-k},a_0)]\right\|_T^p\right] \le c_5\E^{\PP_n}\left[1 + |\xi^1|^p + T\left(1 + \frac{1}{n}\right)|a_0|^p + \frac{1}{n}\sum_{i \neq k}^n\int_0^T\int_A|a|^p\beta^i_t(da) dt\right]
\]
and
\[
\E^{\PP_n}\int_{\C^d}\left\|z\right\|_T^p\widehat{\mu}^x[(\beta^{-k},a_0)](dz) \le c_5\E^{\PP_n}\left[1 + |\xi^1|^p + \frac{T}{n}|a_0|^p + \frac{1}{n}\sum_{i \neq k}^n\int_0^T\int_A|a|^p\beta^i_t(da) dt\right].
\]
Use the hypothesis along with the lower bounds on $f$ and $g$ from assumption (A.5) to get
\begin{align*}
J_k((\beta^{-k},\widetilde{\beta}^k))	 &\ge J_k((\beta^{-k},a_0)) - \epsilon \\
	&\ge -c_2(T+1)\E^{\PP_n}\left[1 + \left\|X^k[(\beta^{-k},a_0)]\right\|_T^p + \int_{\C^d}\left\|z\right\|_T^p\widehat{\mu}^x[(\beta^{-k},a_0)](dz) + |a_0|^{p'}\right]  - \epsilon  \\
	&\ge -C\E^{\PP_n}\left[1 + |\xi^1|^p + \frac{1}{n}\sum_{i \neq k}^n\int_0^T\int_A|a|^p\beta^i_t(da) dt\right]  - \epsilon,
\end{align*}
where $C > 0$ depends only on $c_2$, $c_5$, $T$, and $|a_0|^{p'}$. Applying this with the first result with $\beta$ replaced by $(\beta^{-k},\widetilde{\beta}^k)$ proves (2), replacing $c_7$ by $c_7(1+C)$.

\textit{Proof of (3):} If $\beta$ is an $\epsilon$-Nash equilibrium, then applying (2) with $\widetilde{\beta}^k=\beta^k$ gives
\begin{align*}
\E^{\PP_n}\int_0^T\int_A(|a|^{p'} - c_6|a|^p)\beta^k_t(da)dt \le c_7\E^{\PP_n}\left[1 + \epsilon_k + |\xi^1|^p + \frac{1}{n}\sum_{i = 1}^n\int_0^T\int_A|a|^p\beta^i_t(da)\right].
\end{align*}
The proof is completed by averaging over $k=1,\ldots,n$, rearranging terms, and replacing $c_6$ by $c_6 + c_7$.
\end{proof}

\subsection{Relative compactness and MFG pre-solution}
This section proves that $(P_n)_{n=1}^\infty$, defined in \eqref{def:pn}, is relatively compact and that each limit point is a MFG pre-solution. First, we state a tailor-made tightness result for It\^o processes. It is essentially an application of Aldous' criterion, but the proof is deferred to Section \ref{ap:itocompact}.

\begin{proposition} \label{pr:itocompact}
Fix $c > 0$ and a positive integer $k$. For each $\kappa \ge 0$, let $\Q_\kappa \subset \P(\V \times \C^d)$ denote the set of laws $P \circ (\Lambda,X)^{-1}$ of $\V \times \C^d$-valued random variables $(\Lambda,X)$ defined on some filtered probability space $(\Theta,(\G_t)_{t \in [0,T]},P)$ satisfying
\[
dX_t = \int_AB(t,a)\Lambda_t(da)dt + \Sigma(t)dW_t,
\]
where the following hold:
\begin{enumerate}
\item $W$ is a $(\G_t)_{t \in [0,T]}$-Wiener process of dimension $k$.
\item $\Sigma : [0,T] \times \Theta \rightarrow \R^{d \times k}$ is progressively measurable, and $B : [0,T] \times \Theta \times A \rightarrow \R^d$ is jointly measurable with respect to the progressive $\sigma$-field on $[0,T] \times \Theta$ and the Borel $\sigma$-field on $A$.
\item $X_0$ is $\G_0$-measurable.
\item There exists a nonnegative $\G_T$-measurable random variable $Z$ such that
\begin{enumerate}
\item For each $(t,\omega,a) \in [0,T] \times \Theta \times A$,
\begin{align*}
|B(t,a)| &\le c\left(1 + |X_t| + Z + |a|\right), \quad  |\Sigma\Sigma^\top(t)| \le c\left(1 + |X_t|^{p_\sigma} + Z^{p_\sigma}\right)
\end{align*}
\item Lastly,
\[
\E^P\left[|X_0|^{p'} + Z^{p'} + \int_0^T\int_A|a|^{p'}\Lambda_t(da)dt\right] \le \kappa.
\]
\end{enumerate}
\end{enumerate}
(That is, we vary over $\Sigma$, $B$, $Z$, $k$, and the probability space.) Then, for any triangular array $\{\kappa_{n,i} : 1 \le i \le n\} \subset [0,\infty)$ with $\sup_n\frac{1}{n}\sum_{i=1}^n\kappa_{n,i} < \infty$, the set
\[
\Q := \left\{\frac{1}{n}\sum_{i=1}^nP_i : n \ge 1, \ P^i \in \Q_{\kappa_{n,i}} \text{ for } i=1,\ldots,n\right\}
\]
is relatively compact in $\P^p(\V \times \C^d)$.
\end{proposition}

\begin{lemma} \label{le:convergence-tight}
$(P_n)_{n=1}^\infty$ is relatively compact in $\P^p(\Omega)$, and
\begin{align}
\sup_n\E^{P_n}\left[\|X\|^{p'}_T + \int_{\C^d}\|z\|^{p'}_T\mu(dz) + \int_0^T\int_A|a|^{p'}\Lambda_t(da)dt\right] < \infty. \label{def:convergence-tight0}
\end{align}
\end{lemma}
\begin{proof}
We first establish \eqref{def:convergence-tight0}. Since $\Lambda^n$ is a $\epsilon^n$-Nash equilibrium, part (3) of Lemma \ref{le:valuebound} implies
\[
\frac{1}{n}\sum_{k=1}^n\E^{\PP_n}\int_0^T\int_A(|a|^{p'} - c_6|a|^p)\Lambda^{n,k}_t(da)dt \le c_7\left(1 + \E^{\PP_n}[|\xi^1|^p] + \frac{1}{n}\sum_{k=1}^n\epsilon^n_k \right).
\]
The right-hand side above is bounded in $n$, because of hypothesis \eqref{def:epsilonconverge} and because $\PP_n \circ (\xi^1)^{-1} = \lambda \in \P^p(\R^d)$ for each $n$. Since $p' > p$, it follows that
\begin{align}
\sup_n\frac{1}{n}\sum_{k=1}^n\E^{\PP_n}\int_0^T\int_A|a|^{p'}\Lambda^{n,k}_t(da)dt < \infty. \label{pf:convergence-tight3}
\end{align}
Lemma \ref{le:finitestateestimate} implies
\begin{align*}
\E^{\PP_n}\int_{\C^d}\|z\|^{p'}_T\widehat{\mu}^x[\Lambda^n](dz) \le c_5\E^{\PP_n}\left[1 + |\xi^1|^{p'} + \frac{1}{n}\sum_{k=1}^n\int_0^T\int_A|a|^{p'}\Lambda^{n,k}_t(da)dt\right] =: \kappa_n.
\end{align*}
Thus
\begin{align*}
\E^{P_n}&\left[\|X\|^{p'}_T + \int_{\C^d}\|z\|^{p'}_T\mu(dz) + \int_0^T\int_A|a|^{p'}\Lambda_t(da)dt\right] \\
	&= \frac{1}{n}\sum_{k=1}^n\E^{\PP_n}\left[\|X^k[\Lambda^n]\|^{p'}_T + \int_{\C^d}\|z\|^{p'}_T\widehat{\mu}^x[\Lambda^n](dz) + \int_0^T\int_A|a|^{p'}\Lambda^{n,k}_t(da)dt\right] \\
	&\le c_5\E^{\PP_n}\left[2 + 2|\xi^1|^{p'} + \frac{3}{n}\sum_{k=1}^n\int_0^T\int_A|a|^{p'}\Lambda^{n,k}_t(da)dt\right] \\
	&\le 3\kappa_n.
\end{align*}
Recall in the last line that $c_5 \ge 1$.
From \eqref{pf:convergence-tight3} we conclude that $\sup_n\kappa_n < \infty$, and \eqref{def:convergence-tight0} follows.

To prove that $(P_n)_{n=1}^\infty$, it suffices to show that each family of marginals is relatively compact (e.g. by \cite[Lemma A.2]{lacker-mfgcontrolledmartingaleproblems}).
Since $(P_n \circ (\xi,B,W)^{-1})_{n=1}^\infty$ is a singleton, it is trivially compact. We may apply Proposition \ref{pr:itocompact} to show that
\begin{align*}
P_n \circ (\Lambda,X)^{-1} = \frac{1}{n}\sum_{i=1}^n \PP_n \circ (\Lambda^{n,i},X^{n,i}[\Lambda^n])^{-1}
\end{align*}
forms a relatively compact sequence. Indeed, in the notation of Proposition \ref{pr:itocompact}, we use $Z = (\int_{\C^d}\|z\|_T^p\widehat{\mu}^x[\Lambda^n](dz))^{1/p}$ and $c = c_1$ of assumption (A.4) to check that $\PP_n \circ (\Lambda^{n,i},X^{n,i}[\Lambda^n])^{-1}$ is in $\Q_{\kappa_{n,i}}$ for each $1 \le i \le n$, where
\[
\kappa_{n,i} = \kappa_n + \E^{\PP_n}\left[|\xi^i|^{p'} + \int_0^T\int_A|a|^{p'}\Lambda^{n,i}_t(da)dt\right].
\]
Since $c_5 \ge 1$, we have $\frac{1}{n}\sum_{i=1}^n\kappa_{n,i} \le 2\kappa_n$, and so $\sup_n\frac{1}{n}\sum_{i=1}^n\kappa_{n,i} < \infty$. Thus, Proposition \ref{pr:itocompact} establishes the relative compactness of $(P_n \circ (\Lambda,X)^{-1})_{n=1}^\infty$. Next, note that
$P_n \circ (W,\Lambda,X)^{-1}$ is the mean measure of $P_n \circ \mu^{-1}$ for each $n$, since for each bounded measurable $\phi : \X \rightarrow \R$ we have
\begin{align*}
\E^{P_n}\left[\phi(W,\Lambda,X)\right] &= \frac{1}{n}\sum_{i=1}^n\E^{\PP_n}\left[\phi(W^i,\Lambda^{n,i},X^i[\Lambda^n])\right] = \E^{\PP_n}\int_\X\phi\,d\widehat{\mu}[\Lambda^n] = \E^{P_n}\int_\X\phi\,d\mu.
\end{align*}
Since also 
\[
\sup_n\E^{P_n}\left[\|W\|^{p'}_T + \int_0^T\int_A|a|^{p'}\Lambda_t(da)dt + \|X\|^{p'}_T\right] < \infty,
\]
the relative compactness of $(P_n \circ \mu^{-1})_{n=1}^\infty$ in $\P^p(\P^p(\X))$ follows from the relative compactness of $(P_n \circ (W,\Lambda,X)^{-1})_{n=1}^\infty$ in $\P^p(\X)$. Indeed, when $p=0$ and $\P^0$ is given the topology of weak convergence, this is a well known result of Sznitman, stated in (2.5) of the proof of \cite[Proposition 2.2]{sznitman}. See \cite[Corollary B.2]{lacker-mfgcontrolledmartingaleproblems} for the generalization to $\P^p$. This completes the proof.
\end{proof}

\begin{lemma} \label{le:convergence-pre-solution}
Any limit point $P$ of $(P_n)_{n=1}^\infty$ in $\P^p(\Omega)$ is a MFG pre-solution.
\end{lemma}
\begin{proof}
We abuse notation somewhat by assume that $P_n \rightarrow P$, with the understanding that this is along a subsequence. We check that $P$ satisfies the four conditions of Lemma \ref{le:presolution}.

\begin{enumerate}
\item Of course,
\begin{align*}
P_n \circ (\xi,B,W)^{-1} &= \frac{1}{n}\sum_{i=1}^n\PP_n \circ (\xi^i,B,W^i)^{-1} = \lambda \times \W^{m_0} \times \W^m,
\end{align*}
where $\W^k$ denotes Wiener measure on $\C^k$. Thus $P \circ (\xi,B,W)^{-1} = \lambda \times \W^{m_0} \times \W^m$ as well. On $\Omega_n$, we know $\sigma(W^i_s - W^i_t,B_s-B_t : i=1,\ldots,n, \ s \in [t,T])$ is $\PP_n$-independent of $\F^n_t$ for each $t \in [0,T]$. It follows that, on $\Omega$, $\sigma(W_s - W_t,B_s-B_t : s \in [t,T])$ is $P_n$-independent of $\F^{\xi,B,W,\mu,\Lambda,X}_t$. Hence $B$ and $W$ are Wiener processes on $(\Omega,(\F^{\xi,B,W,\mu,\Lambda,X}_t)_{t \in [0,T]},P)$.
\item Fix bounded continuous functions $\phi : \R^d \times \C^m  \rightarrow \R$ and $\psi : \C^{m_0} \times \P^p(\X) \rightarrow \R$. Since $(\xi^1,W^1),\ldots,(\xi^n,W^n)$ are i.i.d. under $\PP_n$ with common law $P \circ (\xi,W)^{-1}$ for each $n$, the law of large numbers implies
\begin{align*}
\lim_{n\rightarrow\infty}\E^{\PP_n}\left[\left|\frac{1}{n}\sum_{i=1}^n\phi(\xi^i,W^i) - \E^P[\phi(\xi,W)]\right|\psi(B,\widehat{\mu}[\Lambda^n])\right] = 0.
\end{align*}
This implies
\begin{align*}
\E^P\left[\phi(\xi,W)\psi(B,\mu)\right] &= \lim_{n\rightarrow\infty}\frac{1}{n}\sum_{i=1}^n\E^{\PP_n}\left[\phi(\xi^i,W^i)\psi(B,\widehat{\mu}[\Lambda^n]) \right] \\
	&= \E^P[\phi(\xi,W)]\lim_{n\rightarrow\infty}\frac{1}{n}\sum_{i=1}^n\E^{\PP_n}\left[\psi(B,\widehat{\mu}[\Lambda^n]) \right] \\
	&= \E^P[\phi(\xi,W)]\E^P\left[\psi(B,\mu)\right].
\end{align*}
This shows $(B,\mu)$ is independent of $(\xi,W)$ under $P$. Since $\xi^i$ and $W^i$ are independent under $\PP_n$, it follows that $\xi$ and $W$ are independent under $P_n$, for each $n$. Thus $\xi$ and $W$ are independent under $P$, and we conclude that $\xi$, $W$, and $(B,\mu)$ are independent under $P$.
\item Let $\phi : \X \rightarrow \R$ and $\psi : \C^{m_0} \times \P^p(\X) \rightarrow \R$ be bounded and continuous. Then
\begin{align*}
\E^P\left[\psi(B,\mu)\phi(W,\Lambda,X)\right] &= \lim_{n\rightarrow\infty}\E^{\PP_n}\left[\psi(B,\widehat{\mu}[\Lambda^n])\frac{1}{n}\sum_{i=1}^n\phi(W^i,\Lambda^{n,i},X^i[\Lambda^n]) \right] \\
	&= \lim_{n\rightarrow\infty}\E^{\PP_n}\left[\psi(B,\widehat{\mu}[\Lambda^n])\int_\X\phi\,d\widehat{\mu}[\Lambda^n] \right] \\
	&= \E^P\left[\psi(B,\mu)\int_\X\phi\,d\mu \right].
\end{align*}
\item Since $(\xi^i,B,W^i,\widehat{\mu}[\Lambda^n],\Lambda^{n,i},X^i[\Lambda^n])$ verify the state SDE under $\PP_n$, the canonical processes $(\xi,B,W,\mu,\Lambda,X)$ verify the state equation \eqref{def:SDE} under each $P_n$, for each $n$. It follows from the results of Kurtz and Protter \cite{kurtzprotter-weakconvergence} that the state equation holds under the limit measure $P$ as well.
\end{enumerate}
\end{proof}

\subsection{Modified finite-player games}
The last step of the proof, executed in the next Section \ref{se:optimalityinthelimit}, is to show that any limit $P$ of $P_n$ is optimal.
This step is more involved, and we devote this subsection to studying a useful technical device which we call the \emph{$k$-modified $n$-player game}, in which agent $k$ is removed from the empirical measures. Intuitively, if the $n$-player game is modified so that the empirical measure (present in the state process dynamics and objective functions) no longer includes agent $k$, then the optimization problem of agent $k$ de-couples from that of the other agents; agent $k$ may then treat the empirical measure of the other $n-1$ agents as fixed and thus faces exactly the type of control problem encountered in the MFG. Let us make this idea precise.

For $\beta = (\beta^1,\ldots,\beta^n) \in \A_n^n(\mathcal{E}_n)$, define $Y^{-k}[\beta] = (Y^{-k,1}[\beta],\ldots,Y^{-k,n}[\beta])$ to be the unique strong solution on $(\Omega_n,(\F^n_t)_{t \in [0,T]},\PP_n)$ of the SDE
\begin{align*}
Y^{-k,i}_t[\beta] &= \xi^i + \int_0^t\int_Ab(s,Y^{-k,i}_s[\beta],\widehat{\mu}^{-k,x}_s[\beta],a)\beta^i_s(da)dt + \int_0^t\sigma(s,Y^{-k,i}_s[\beta],\widehat{\mu}^{-k,x}_s[\beta])dW^i_s \\
	&\quad\quad + \int_0^t\sigma_0(s,Y^{-k,i}_s[\beta],\widehat{\mu}^{-k,x}_s[\beta])dB_s, \\
\widehat{\mu}^{-k,x}[\beta] &:= \frac{1}{n-1}\sum_{i \neq k}^n\delta_{Y^{-k,i}[\beta]}.
\end{align*}
Define also
\[
\widehat{\mu}^{-k}[\beta] = \frac{1}{n-1}\sum_{i \neq k}^n\delta_{(W^i,\beta^i,Y^{-k,i}[\beta])}.
\]
Intuitively, $Y^{-k,i}$ is agent $i$'s state process in an analog of the $n$-player game, in which agent $k$ has been removed from the empirical measure. Naturally, for fixed $k$, the $k$-modified state processes $Y^{-k}[\beta]$ should not be far from the true state processes $X[\beta]$ if $n$ is large, and we will quantify this precisely.
We will need to be somewhat explicit about the choice of metric on $\V$, so we define $d_\V$ by
\[
d^p_\V(q,q') := T\ell_{[0,T] \times A}(q/T,q'/T) = \inf_\gamma\int_{[0,T]^2 \times A^2}(|t-t'|^p + |a-a'|^p)\gamma(dt,dt',da,da'),
\]
where the infimum is over measures $\gamma$ on $[0,T]^2 \times A^2$ with marginals $q$ and $q'$. By choosing $\gamma = dt\delta_{t}(dt')q_t(da)q'_t(da')$, we note that
\begin{align}
d^p_\V(q,q') \le 2^{p-1}\int_0^T\int_A|a|^pq_t(da)dt + 2^{p-1}\int_0^T\int_A|a|^pq'_t(da)dt. \label{vmetricinequality}
\end{align}
Define the $p'$-Wasserstein distance $\ell_{\X,p'}$ on $\P^{p'}(\X)$ with respect to the metric
\begin{align}
d_\X((w,q,x),(w',q',x')) &:= \|w-w'\|_T + d_\V(q,q') + \|x-x'\|_T. \label{def:dX}
\end{align}

\begin{lemma} \label{le:modifiedsystem}
There exists a constant $c_8 > 0$ such that, for each $n \ge k \ge 1$ and $\beta = (\beta^1,\ldots,\beta^n) \in \A_n^n(\mathcal{E}_n)$, we have
\begin{align*}
\E^{\PP_n}&\left[\ell_{\X,p'}^{p'}(\widehat{\mu}^{-k}[\beta],\widehat{\mu}[\beta]) + \left\|X^k[\beta] - Y^{-k,k}[\beta]\right\|^{p'}_T \right] \le c_8(1+M[\beta])/n, \text{ where} \\
M[\beta] &:= \E^{\PP_n}\left[|\xi^1|^{p'} + \frac{1}{n}\sum_{i=1}^n\int_0^T\int_A|a|^{p'}\beta^i_t(da)dt\right].
\end{align*}
\end{lemma}
\begin{proof}
Throughout the proof, $n$ is fixed, expected values are all with respect to $\PP_n$, and the notation $[\beta]$ is omitted. Define the truncated $p'$-Wasserstein distance $\ell_t$ on $\P^{p'}(\C^d)$ by
\begin{align}
\ell^p_t(\mu,\nu) := \inf\left\{\int_{\C^d \times \C^d}\|x-y\|_t^{p'} \, \gamma(dx,dy) : \gamma \in \P(\C^d \times \C^d) \text{ has marginals } \mu, \nu\right\} \label{def:truncatedwasserstein}
\end{align}
Apply the Doob's maximal inequality and Jensen's inequality (using the assumption $p' \ge 2$) to find a constant $C > 0$ (which will change from line to line but depends only on $d$, $p$, $p'$, $T$, $c_1$, and $c_5$) such that
\begin{align*}
\E\left[\|X^i - Y^{-k,i}\|^{p'}_t\right] \le &\,C\E\int_0^t\int_A|b(s,X^i_s,\widehat{\mu}^x_s,a) - b(s,Y^{-k,i}_s,\widehat{\mu}^{-k,x}_s,a)|^{p'}\beta^i_s(da)ds \\
	&+ C\E\int_0^t\left|\sigma(s,X^i_s,\widehat{\mu}^x_s) - \sigma(s,Y^{-k,i}_s,\widehat{\mu}^{-k,x}_s)\right|^{p'}ds  \\
	&+ C\E\int_0^t\left|\sigma_0(s,X^i_s,\widehat{\mu}^x_s) - \sigma_0(s,Y^{-k,i}_s,\widehat{\mu}^{-k,x}_s)\right|^{p'}ds   \\
	\le &\,C\E\int_0^t\left(\|X^i - Y^{-k,i}\|^{p'}_s + \ell^{p'}_s(\widehat{\mu}^x,\widehat{\mu}^{-k,x})\right)ds.
\end{align*}
The last line followed from the Lipschitz assumption (A.4), along with the observation that
\[
\ell_{\R^d,p}(\nu^1_s,\nu^2_s) \le \ell_{\R^d,p'}(\nu^1_s,\nu^2_s) \le \ell_s(\nu^1,\nu^2),
\]
for each $\nu^1,\nu^2 \in \P^{p}(\C^d)$. By Gronwall's inequality (updating the constant $C$),
\begin{align}
\E\left[\|X^i - Y^{-k,i}\|^{p'}_t\right] \le C\E\int_0^t\ell^{p'}_s(\widehat{\mu}^x,\widehat{\mu}^{-k,x})ds.  \label{pf:mkv2}
\end{align}
Now we define a standard coupling of the empirical measures $\widehat{\mu}^x$ and $\widehat{\mu}^{-k,x}$: first, draw a number $j$ from $\{1,\ldots,n\}$ uniformly at random, and consider $X^j$ to be a sample from $\widehat{\mu}^x$. If $j \neq k$, choose $Y^{-k,j}$ to be a sample from $\widehat{\mu}^{-k,x}$, but if $j = k$, draw another number $j'$ from $\{1,\ldots,n\} \backslash \{k\}$ uniformly at random, and choose $Y^{-k,j'}$ to be a sample from $\widehat{\mu}^{-k,x}$. This yields
\begin{align}
\ell^{p'}_t(\widehat{\mu}^x,\widehat{\mu}^{-k,x}) &\le \frac{1}{n}\sum_{i \neq k}^n\|X^i - Y^{-k,i}\|_t^{p'} + \frac{1}{n(n-1)}\sum_{i\neq k}^n\|X^k - Y^{-k,i}\|_t^{p'} \label{pf:mkv3}
\end{align}
We know from Lemma \ref{le:finitestateestimate} that
\[
\frac{1}{n-1}\sum_{i \neq k}^n\E[\|X^i\|_T^{p'}] \le c_5(1+M),
\]
It should be clear that an analog of Lemma \ref{le:finitestateestimate} holds for $Y^{-k,i}$ as well, with the same constant. In particular,
\[
\frac{1}{n-1}\sum_{i \neq k}^n\E[\|Y^{-k,i}\|_T^{p'}] \le c_5(1+M).
\]
Combine the above four inequalities, averaging \eqref{pf:mkv2} over $i \neq k$, to get
\begin{align*}
\E\left[\ell^{p'}_t(\widehat{\mu}^x,\widehat{\mu}^{-k,x})\right] &\le C\E\int_0^t\ell^{p'}_s(\widehat{\mu}^x,\widehat{\mu}^{-k,x})ds + 2^{p'}c_5(1+M)/n.
\end{align*}
Gronwall's inequality yields a new constant such that
\begin{align*}
\E\left[\ell^{p'}_T(\widehat{\mu}^x,\widehat{\mu}^{-k,x})\right] \le  C(1+M)/n.
\end{align*}
Return to \eqref{pf:mkv2} to find
\begin{align}
\E\left[\|X^i - Y^{-k,i}\|^{p'}_T\right] \le  C(1+M)/n, \text{ for } i=1,\ldots,n. \label{pf:mkv4}
\end{align}
The same coupling argument leading to \eqref{pf:mkv3} also yields
\begin{align}
\ell^{p'}_{\X,p'}(\widehat{\mu},\widehat{\mu}^{-k}) &\le \frac{1}{n}\sum_{i \neq k}^n\|X^i - Y^{-k,i}\|^{p'}_T \nonumber \\
	&\quad + \frac{1}{n(n-1)}\sum_{i\neq k}^n d^{p'}_\X((W^i,\beta^{i},Y^{-k,i}),(W^k,\beta^{k},X^k)) \label{pf:mkv5}
\end{align}
Using \eqref{vmetricinequality}, we find yet another constant such that
\begin{align*}
\E\left[d^{p'}_\X((W^i,\beta^{i},Y^{-k,i}),(W^k,\beta^{k},X^k))\right] &\le 3^{p'-1}\E\left[\|W^i - W^k\|_T^{p'} + d_{\V}^{p'}(\beta^{i},\beta^{k}) + \|Y^{-k,i}  -X^k \|_T^{p'}\right] \\
	&\le C\E\left[\int_0^T\int_A|a|^{p'}\beta^i_t(da)dt + \int_0^T\int_A|a|^{p'}\beta^k_t(da)dt \right. \\
		&\quad\quad\quad\quad\quad\left. \vphantom{\int_0^T}+ \|W^1\|_T^{p'} + |Y^{-k,i}\|_T^{p'} + \|X^k \|_T^{p'}\right] \\
	&\le C\left(2nM + 2nc_5(1+M) + \E[\|W^1\|_T^{p'}]\right).
\end{align*}
Thus
\begin{align*}
\frac{1}{n-1}\sum_{i \neq k}^n\E\left[d^{p'}_\X((W^i,\beta^{i},Y^{-k,i}),(W^k,\beta^{k},X^k))\right] &\le C(1+M).
\end{align*}
Applying this bound and \eqref{pf:mkv4} to \eqref{pf:mkv5} completes the proof.
\end{proof}

\subsection{Optimality in the limit} \label{se:optimalityinthelimit}
Before we complete the proof, recall the definitions of $\RC$, $\A$, and $\A^*$ from Section \ref{se:canonicalspace}. 
The final step is to show that $P \in \RC\A^*(P \circ (\xi,B,W,\mu)^{-1})$, for any limit $P$ of $(P_n)_{n=1}^\infty$. The idea of the proof is to use the density of adapted controls (see Lemma \ref{le:adapteddensity}) to construct nearly optimal controls for the MFG with nice continuity properties. From these controls we build admissible controls for the $n$-player game, and it must finally be argued that the inequality obtained from the $\epsilon^n$-Nash assumption on $\Lambda^n$ may be passed to the limit.

\begin{proof}[Proof of Theorem \ref{th:mainconvergence}]
Let $P$ be a limit point of $(P_n)_{n=1}^\infty$, which we know exists by Lemma \ref{le:convergence-pre-solution}, and again abuse notation by assuming that $P_n \rightarrow P$. Let $\rho := P \circ (\xi,B,W,\mu)^{-1}$. We know from Lemma \ref{le:convergence-tight} that $P$ is a MFG pre-solution, and in light of Lemma \ref{le:convergence-pre-solution} we need only to check that $P$ is optimal. Fix some $Q^* \in \A^*(\rho)$, and set $P^* := \RC(Q^*)$.  (Lemma \ref{le:jcontinuous} assures us that $\A^{*}(\rho)$ is nonempty.) By Lemma \ref{le:adapteddensity}, there exist compact adapted functions $\phi_i : \Omega_0 \times \P^p(\X) \rightarrow \V$ (see Definition \ref{def:adapted}) such that 
\begin{enumerate}
\item $\phi_i(\omega,\cdot)$ is continuous for each $\omega \in \Omega_0$, and
\item $Q^* = \lim_{i \rightarrow \infty}Q_i$, and $J(\RC(Q^*)) = \lim_{i \rightarrow \infty}J(\RC(Q_i))$, where 
\[
Q_i := \rho \circ (\xi,B,W,\mu,\phi_i(\xi,B,W,\mu))^{-1}.
\]
\end{enumerate}
Fix $\delta > 0$, and find $i_0$ large enough that
\begin{align}
J(\RC(Q_{i_0})) \ge J(\RC(Q^*)) - \delta = \sup_{P' \in \RC\A*(\rho)}J(P') - \delta. \label{pf:optimal-1}
\end{align}
Set $\widetilde{Q} := Q_{i_0}$ and $\tilde{\phi} := \phi_{i_0}$, for ease of notation; we will use no other $\phi_i$ or $Q_i$ from now on. 
For $1 \le k \le n$, let
\[
\rho_{n,k} := \PP_n \circ (\xi^k,B,W^k,\widehat{\mu}^{-k}[\Lambda^n])^{-1},
\]
and
\begin{align*}
Q_{n,k} &:= \rho_{n,k} \circ (\xi,B,W,\mu,\tilde{\phi}(\xi,B,W,\mu))^{-1} \\
	&= \PP_n \circ \left(\xi^k,B,W^k,\widehat{\mu}^{-k}[\Lambda^n],\widetilde{\phi}(\xi^k,B,W^k,\widehat{\mu}^{-k}[\Lambda^n])\right)^{-1}.
\end{align*}
It follows from Lemma \ref{le:modifiedsystem} that
\begin{align*}
\lim_{n\rightarrow\infty}\frac{1}{n}\sum_{k=1}^n\rho_{n,k} &= \lim_{n\rightarrow\infty}\frac{1}{n}\sum_{k=1}^n\PP_n \circ (\xi^k,B,W^k,\widehat{\mu}[\Lambda^n])^{-1} = \rho.
\end{align*}
Since
\[
\frac{1}{n}\sum_{k=1}^n\rho_{n,k} \circ (\xi,B,W)^{-1} = P  \circ (\xi,B,W)^{-1}
\]
does not depend on $n$, the continuity of $\tilde{\phi}(\omega,\cdot)$ for each $\omega \in \Omega_0$ implies (using e.g. \cite[Lemma A.3]{lacker-mfgcontrolledmartingaleproblems} to deal with the possible discontinuity of $\tilde{\phi}$ in $\omega$)
\begin{align*}
\widetilde{Q} = \lim_{n\rightarrow\infty}\frac{1}{n}\sum_{k=1}^nQ_{n,k}.
\end{align*}
It is fairly straightforward to check that $\RC$ is a linear map, and it is even more straightforward to check that $J$ is linear.
Moreover, since $\tilde{\phi}$ is a compact function, the continuity of $\RC$ and $J$ of Lemmas \ref{le:rccontinuous} and \ref{le:jcontinuous} imply
\begin{align}
\lim_{n\rightarrow\infty}\frac{1}{n}\sum_{k=1}^{n}J(\RC(Q_{n,k})) &= \lim_{n\rightarrow\infty}J\left(\RC\left(\frac{1}{n}\sum_{k=1}^{n}Q_{n,k}\right)\right) \nonumber = J(\RC(\widetilde{Q})) \nonumber \\
	&\ge \sup_{P' \in \RC\A(\rho)}J(P') - \delta, \label{pf:optimal1}
\end{align}
where the last step used \eqref{pf:optimal-1}.

Now, for $k \le n$, define $\beta^{n,k} \in \A_n(\mathcal{E}_n)$ by
\[
\beta^{n,k} := \tilde{\phi}\left(\xi^k,B,W^k,\widehat{\mu}^{-k}[\Lambda^n]\right).
\]
For $\beta \in \A_n(\mathcal{E}_n)$, abbreviate $(\Lambda^{n,-k},\beta) := ((\Lambda^n)^{-k},\beta)$.
Since agent $k$ is removed from the empirical measure, we have $\widehat{\mu}^{-k}[\Lambda^n] = \widehat{\mu}^{-k}[(\Lambda^{n,-k},\beta)]$ for any $\beta \in \A_n(\mathcal{E}_n)$.
The key point is that for each $k \le n$,
\begin{align}
\PP_n \circ \left(\xi^k,B,W^k,\widehat{\mu}^{-k}[(\Lambda^{n,-k},\beta^{n,k})],\beta^{n,k},Y^{-k,k}[(\Lambda^{n,-k},\beta^{n,k})]\right)^{-1} = \RC(Q_{n,k}). \label{pf:optimal1.1}
\end{align}
To prove \eqref{pf:optimal1.1}, let $P'$ denote the measure on the left-hand side. Since $\widehat{\mu}^{-k}[\Lambda^n] = \widehat{\mu}^{-k}[(\Lambda^{n,-k},\beta^{n,k})]$, we have
\begin{align*}
P' \circ (\xi,B,W,\mu,\Lambda)^{-1} &= Q_{n,k}.
\end{align*}
Since the processes 
\[
\left(\xi^k,B,W^k,\widehat{\mu}^{-k}[(\Lambda^{n,-k},\beta^{n,k})],\beta^{n,k},Y^{-k,k}[(\Lambda^{n,-k},\beta^{n,k})]\right)
\]
verify the state SDE \eqref{def:SDE} on $(\Omega_n,(\F^n_t)_{t \in [0,T]},\PP_n)$, the canonical processes $(\xi,B,W,\mu,\Lambda,X)$ verify the state SDE \eqref{def:SDE} under $P'$. Hence, $P' = \RC(Q_{n,k})$.
With \eqref{pf:optimal1.1} in hand, by definition of $J$ the inequality \eqref{pf:optimal1} then translates to
\begin{align}
\lim_{n\rightarrow\infty}\frac{1}{n}\sum_{k=1}^{n}\E^{\PP_n}\left[\Gamma\left(\widehat{\mu}^{-k,x}[(\Lambda^{n,-k},\beta^{n,k})],\beta^{n,k},Y^{-k,k}[(\Lambda^{n,-k},\beta^{n,k})]\right) \right] &\ge \sup_{P' \in \RC\A(\rho)}J(P') - \delta. \label{pf:optimal2}
\end{align}

Before completing the proof, we check more technical point:
\begin{align}
0 = \lim_{n\rightarrow\infty}\frac{1}{n}\sum_{k=1}^{n}\E^{\PP_n}&\left[\Gamma\left(\widehat{\mu}^{-k,x}[(\Lambda^{n,-k},\beta^{n,k})],\beta^{n,k},Y^{-k,k}[(\Lambda^{n,-k},\beta^{n,k})]\right) \right. \nonumber \\
	&\quad\left. - \Gamma\left(\widehat{\mu}^x[(\Lambda^{n,-k},\beta^{n,k})],\beta^{n,k},X^k[(\Lambda^{n,-k},\beta^{n,k})]\right) \right] \label{pf:optimal3}
\end{align}
Indeed, it follows from Lemma \ref{le:finitestateestimate} (and an obvious analog for the modified state processes $Y$) that 
\begin{align*}
Z_{n,k} &:= \E^{\PP_n}\left[\vphantom{\int_{\C^d}}\|X^k[(\Lambda^{n,-k},\beta^{n,k})]\|^{p'}_T + \|Y^{-k,k}[(\Lambda^{n,-k},\beta^{n,k})]\|^{p'}_T \right. \\
	&\quad\left. + \int_{\C^d}\|z\|^{p'}_T\widehat{\mu}^x[(\Lambda^{n,-k},\beta^{n,k})](dz) + \int_{\C^d}\|z\|^{p'}_T\widehat{\mu}^{-k,x}[(\Lambda^{n,-k},\beta^{n,k})](dz)\right] \\
	&\le 4c_4\E^{\PP_n}\left[|\xi^1|^{p'} + \frac{1}{n}\sum_{i=1}^n\int_0^T\int_A|a|^{p'}\Lambda^{n,i}_t(da)dt + \int_0^T\int_A|a|^{p'}\beta^{n,k}_t(da)dt\right].
\end{align*}
Lemma \ref{le:convergence-tight} says that
\[
\sup_n\E^{\PP_n}\left[\frac{1}{n}\sum_{i=1}^n\int_0^T\int_A|a|^{p'}\Lambda^{n,i}_t(da)dt\right] < \infty.
\]
Compactness of $\tilde{\phi}$ implies that there exists a compact set $K \subset A$ such that $\widetilde{\beta}^{n,k}_t(K^c) = 0$ for a.e. $t \in [0,T]$ and all $n \ge k \ge 1$. Thus
\begin{align*}
\sup_n\frac{1}{n}\sum_{k=1}^nZ_{n,k} < \infty,
\end{align*}
and we have the uniform integrability needed to deduce \eqref{pf:optimal3}, from Lemma \ref{le:modifiedsystem} and from the continuity and growth assumptions (A.5) on $f$ and $g$.

A simple manipulation of the definitions yields $J(P_n) = \frac{1}{n}\sum_{k=1}^nJ_k(\Lambda^n)$. Then, since $P_n \rightarrow P$, the upper semicontinuity of $J$ of Lemma \ref{le:jcontinuous} implies
\begin{align*}
J(P) \ge \limsup_{n\rightarrow\infty}\frac{1}{n}\sum_{k=1}^nJ_k(\Lambda^n).
\end{align*}
Finally, use the fact that $\Lambda^n$ is a relaxed $\epsilon^n$-Nash equilibrium to get
\begin{align*}
J(P) &\ge \liminf_{n\rightarrow\infty}\frac{1}{n}\sum_{k=1}^n\left[J_k((\Lambda^{n,-k},\beta^{n,k})) - \epsilon^n_k\right] \\
	&= \liminf_{n\rightarrow\infty}\frac{1}{n}\sum_{k=1}^n \E^{\PP_n}\left[\Gamma\left(\widehat{\mu}^x[(\Lambda^{n,-k},\beta^{n,k})],\beta^{n,k},X^k[(\Lambda^{n,-k},\beta^{n,k})]\right) \right] \\
	&= \liminf_{n\rightarrow\infty} \frac{1}{n}\sum_{k=1}^{n}\E^{\PP_n}\left[\Gamma\left(\widehat{\mu}^{-k,x}[(\Lambda^{n,-k},\beta^{n,k})],\beta^{n,k},Y^{-k,k}[(\Lambda^{n,-k},\beta^{n,k})]\right) \right] \\
	&\ge \sup_{P' \in \RC\A(\rho)}J(P') - \delta
\end{align*}
The second line follows from the definition of $J_k$, and the $\epsilon^n_k$ drops out because of the hypothesis \eqref{def:epsilonconverge}. The third line comes from \eqref{pf:optimal3}, and the last is from \eqref{pf:optimal2}. Since $P \in \RC\A(\rho)$, and since $\delta > 0$ was arbitrary, this shows that $P \in \RC\A^{*}(\rho)$.
\end{proof}

\section{Proof of Theorem \ref{th:converseconvergence}} \label{se:converseconvergenceproof}
This section is devoted to the proof of Theorem \ref{th:converseconvergence}, which we split into two pieces.

\begin{theorem} \label{th:partialconverseconvergence}
Suppose assumptions \ref{assumption:A} and \ref{assumption:B} hold. Let $P \in \P(\Omega)$ be a weak MFG solution. Then there exist, for each $n$,
\begin{enumerate}
\item $\epsilon_n \ge 0$,
\item an $n$-player environment $\mathcal{E}_n = (\Omega_n,(\F^n_t)_{t \in [0,T]},\PP_n,\xi,B,W)$, and
\item a relaxed $(\epsilon_n,\ldots,\epsilon_n)$-Nash equilibrium $\Lambda^n = (\Lambda^{n,1},\ldots,\Lambda^{n,n})$ on $\mathcal{E}_n$,
\end{enumerate}
such that $\lim_{n\rightarrow\infty}\epsilon_n = 0$ and $P_n \rightarrow P$ in $\P^p(\Omega)$, where
\begin{align*}
P_n := \frac{1}{n}\sum_{i=1}^n\PP_n \circ \left(\xi^i,B,W^i,\widehat{\mu}[\Lambda^n],\Lambda^{n,i},X^i[\Lambda^n]\right)^{-1}.
\end{align*}
\end{theorem}

Theorem \ref{th:partialconverseconvergence} is nearly the same as Theorem \ref{th:converseconvergence}, except that the equilibria $\Lambda^n$ are now \emph{relaxed} instead of \emph{strong}, and the environments $\mathcal{E}_n$ are now part of the conclusion of the theorem instead of the input.
We will prove Theorem \ref{th:partialconverseconvergence} by constructing a convenient sequence of environments $\mathcal{E}_n$, which all live on the same larger probability space supporting an i.i.d. sequence of state processes corresponding to the given MFG solution. This kind of argument is known as \emph{trajectorial propagation of chaos} in the literature on McKean-Vlasov limits, and the Lipschitz assumption in the measure argument is useful here. The precise choice of environments also facilitates the proof of the following Proposition. Recall the definition of a \emph{strong $\epsilon$-Nash equilibrium} from Remark \ref{re:equivalentstrongequilibrium} and the discussion preceding it.

\begin{proposition} \label{pr:fixednapproximation}
Let $\mathcal{E}_n$ be the environments defined in the proof of Theorem \ref{th:partialconverseconvergence} (in Section \ref{se:environments}).
Let $\Lambda^0 = (\Lambda^{0,1},\ldots,\Lambda^{0,n}) \in \A_n^n(\mathcal{E}_n)$. Then there exist strong strategies $\Lambda^k = (\Lambda^{k,1},\ldots,\Lambda^{k,n}) \in \A_n^n(\mathcal{E}_n)$ such that:
\begin{enumerate}
\item In $\P^p\left(\C^{m_0} \times (\C^m)^n \times \V^n \times (\C^d)^n\right)$,
\[	
\lim_{k\rightarrow\infty}\PP_n \circ \left(B,W,\Lambda^k,X[\Lambda^k]\right)^{-1} =\PP_n \circ \left(B,W,\Lambda^0,X[\Lambda^0]\right)^{-1},
\]
\item $\lim_{k\rightarrow\infty}J_i(\Lambda^k) = J_i(\Lambda^0)$, for $i=1,\ldots,n$,
\item 
\[
\limsup_{k\rightarrow\infty}\sup_{\beta \in \A_n(\mathcal{E}_n)}J_i((\Lambda^{k,-i},\beta)) \le \sup_{\beta \in \A_n(\mathcal{E}_n)}J_i((\Lambda^{0,-i},\beta)), \text{ for } i=1,\ldots,n.
\]
\end{enumerate}
In particular, if $\Lambda^0$ is a relaxed $\epsilon^0=(\epsilon^0_1,\ldots,\epsilon^0_n)$-Nash equilibrium, then $\Lambda^k$ is a strong $(\epsilon^0 + \epsilon^k)$-Nash equilibrium, where
\[
\epsilon^k_i := \left[\sup_{\beta \in \A_n(\mathcal{E}_n)}J_i((\Lambda^{k,-i},\beta)) - J_i(\Lambda^k) - \epsilon^0_i\right]^+ \rightarrow 0 \text{ as } k \rightarrow \infty.
\]
\end{proposition}

\begin{proof}[Proof of Theorem \ref{th:converseconvergence}]
Recall that strong strategies are insensitive to the choice of $n$-player environment (see Remark \ref{re:insensitivetoenvironment}), and so it suffices to prove the theorem on any given sequence of environments, such as those provided by Theorem \ref{th:partialconverseconvergence}. By Theorem \ref{th:partialconverseconvergence} we may find $\epsilon_n \rightarrow 0$ and a relaxed $(\epsilon_n,\ldots,\epsilon_n)$-Nash equilibrium $\Lambda^n$ for the $n$-player game, with the desired convergence properties. Then, by Proposition \ref{pr:fixednapproximation}, we find for each $n$ each $k$ a strong $\epsilon^{n,k}=(\epsilon_n + \epsilon^{n,k}_1,\ldots,\epsilon_n + \epsilon^{n,k}_n)$-Nash equilibrium $\Lambda^{n,k} \in \A_n^n(\mathcal{E}_n)$ with the convergence properties defined in Proposition \ref{pr:fixednapproximation}. For each $n$, choose $k_n$ large enough to make $\epsilon^{n,k_n}_i \le 2^{-n}$ for each $i=1,\ldots,n$ and so that the sequences in (1-3) of Proposition \ref{pr:fixednapproximation} are each within $2^{-n}$ of their respective limits.
\end{proof}

\subsection{Construction of environments} \label{se:environments}
Fix a weak MFG solution $P$. Define $P_{B,\mu} := P \circ (B,\mu)^{-1}$. We will work on the space
\[
\overline{\Omega} := [0,1] \times \C^{m_0} \times \P^p(\X) \times \X^\infty.
\]
Let $(U,B,\mu,(W^i,\Lambda^i,Y^i)_{i=1}^\infty)$ denote the identity map (i.e. coordinate processes) on $\overline{\Omega}$. For $n \in \N \cup \{\infty\}$, consider the complete filtration $(\overline{\F}^n_t)_{t \in [0,T]}$ generated by $U$, $B$, $\mu$, and $(W^i,\Lambda^i,Y^i)_{i=1}^n$, that is the completion of
\[
\sigma\left\{\left(U,B_s,\mu(C_1),(W^i_s,\Lambda^i([0,s] \times C_2),Y^i_s)_{i=1}^n\right) : s \le t, \ C_1 \in \F^\X_t, \ C_2 \in \B(A)\right\}.
\] 
Define the probability measure $\PP$ on $(\overline{\Omega},\overline{\F}^\infty_T)$ by
\[
\PP := duP_{B,\mu}(d\beta,d\nu)\prod_{i=1}^\infty\nu(dw^i,dq^i,dy^i).
\]
By construction,
\[
\PP \circ (Y^i_0,B,W^i,\mu,\Lambda^i,Y^i)^{-1} = P, \text{ for each } i,
\]
and $(W^i,\Lambda^i,Y^i)_{i=1}^\infty$ are conditionally i.i.d. with common law $\mu$ given $(B,\mu)$. Moreover, $U$ and $(B,\mu,(W^i,\Lambda^i,Y^i)_{i=1}^\infty)$ are independent under $\PP$. We will work with the $n$-player environments
\[
\mathcal{E}_n := \left(\overline{\Omega},(\overline{\F}^n_t)_{t \in [0,T]},\PP,(Y^1_0,\ldots,Y^n_0),B,(W^1,\ldots,W^n)\right),
\]
and we will show that the canonical process $(\Lambda^1,\ldots,\Lambda^n)$ is a relaxed $(\epsilon_n,\ldots,\epsilon_n)$-Nash equilibrium for some $\epsilon_n \rightarrow \infty$. Including the seemingly superfluous random variable $U$ makes the class of admissible controls as rich as possible, in a sense which will be more clear later; until the proof of Proposition \ref{pr:fixednapproximation}, $U$ will be behind the scenes.

Define $X[\beta]$ and $\widehat{\mu}[\beta]$ for $\beta \in \A_n^n(\mathcal{E}_n)$ as usual, as in Section \ref{se:finiteplayergames}.
For each $(\overline{\F}^\infty_t)_{t \in [0,T]}$-progressive $\P(A)$-valued process $\beta$ on $\overline{\Omega}$ and each $i \ge 1$, define $Y^i[\beta]$ to be the unique solution of the SDE
\[
dY^i_t[\beta] = \int_Ab(t,Y^i_t[\beta],\mu^x_t,a)\beta_t(da) + \sigma(t,Y^i_t[\beta],\mu^x_t)dW^i_t + \sigma_0(t,Y^i_t[\beta],\mu^x_t)dB_t, \ Y^i_0[\beta] = Y^i_0. 
\]
Note that if $\beta = (\beta^1,\ldots,\beta^n) \in \A_n^n(\mathcal{E}_n)$ then $X^i[\beta]$ differs from $Y^i[\beta^i]$ only in the measure flow which appears in the dynamics; $X^i[\beta]$ depends on the empirical measure flow of $(X^1[\beta],\ldots,X^n[\beta])$, whereas $Y^i[\beta^i]$ depends on the random measure $\mu$ coming from the MFG solution.
Define the canonical $n$-player strategy profile by
\[
\overline{\Lambda}^n = (\overline{\Lambda}^{n,1},\ldots,\overline{\Lambda}^{n,n}) := (\Lambda^1,\ldots,\Lambda^n) \in \A_n^n(\mathcal{E}_n).
\]
This abbreviation serves in part to indicate which $n$ we are working with at any given moment, so that we can suppress the index $n$ from the rest of the notation. Note that $Y^i[\overline{\Lambda}^{n,i}] = Y^i[\Lambda^i] = Y^i$.

\subsection{Trajectorial propagation of chaos}
Intuition from the theory of propagation of chaos suggests that the state processes $(Y^1,\ldots,Y^n)$ and $(X^1,\ldots,X^n)$ should be close in some sense, and the purpose of this section is to make this quantitative.
For $\beta \in \A_n(\mathcal{E}_n)$, abbreviate
\[
(\overline{\Lambda}^{n,-i},\beta) := ((\overline{\Lambda}^n)^{-i},\beta) \in \A_n^n(\mathcal{E}_n).
\]
Recall the definition of the metric $d_\X$ on $\X$ from \eqref{def:dX}, and again define the $p'$-Wasserstein metric $\ell_{\X,p'}$ on $\P^p(\X)$ relative to the metric $d_\X$.

\begin{lemma} \label{le:yLLN}
Fix $i$ and a $(\overline{\F}^\infty_t)_{t \in [0,T]}$-progressive $P(A)$-valued process $\beta$, and define
\[
\widehat{\nu}^{n,i}[\beta] := \frac{1}{n}\left(\sum_{k \neq i}^n\delta_{(W^k,\Lambda^k,Y^k)} + \delta_{(W^i,\beta,Y^i[\beta])}\right).
\]
There exists a sequence $\delta_n > 0$ converging to zero such that
\[
\E^\PP\left[\ell^{p'}_{\X,p'}(\widehat{\nu}^{n,i}[\beta],\mu)\right] \le \delta_n\left(1 + \E^\PP\int_0^T\int_A|a|^{p'}\beta_t(da)dt\right).
\]
\end{lemma}
\begin{proof}
Expectations are all with respect to $\PP$ throughout the proof.
For $1 \le i \le n$ define 
\[
\widehat{\nu}^n := \frac{1}{n}\sum_{k=1}^n\delta_{(W^k,\Lambda^k,Y^k)}.
\]
Using the obvious coupling, we find
\[
\ell^{p'}_{\X,p'}(\widehat{\nu}^{n,i}[\beta],\widehat{\nu}^n) \le \frac{1}{n}d^{p'}_\X\left((W^i,\Lambda^i,Y^i),(W^i,\beta,Y^i[\beta])\right).
\]
Using \eqref{vmetricinequality}, we find a constant $C > 0$, depending only on $p$, $p'$, and $T$, such that
\begin{align*}
\E\left[d^{p'}_\X\left((W^i,\Lambda^i,Y^i),(W^i,\beta,Y^i[\beta])\right)\right] \le &C\E\left[\int_0^T\int_A|a|^{p'}\beta_t(da)dt + \int_0^T\int_A|a|^{p'}\Lambda^i_t(da)dt\right. \\
	&\quad\quad\quad\left.\vphantom{\int_0^T} + \|Y^i\|^{p'}_T + \|Y^i[\beta]\|^{p'}_T\right]
\end{align*}
Analogously to Lemma \ref{le:finitestateestimate}, it holds that
\begin{align}
\E[\|Y^i[\beta]\|^{p'}_T] &\le c_5\E\left[1 + |Y^i_0|^{p'} + \int_{\C^d}\|z\|^{p'}_T\mu^x(dz) + \int_0^T\int_A|a|^{p'}\beta_t(da)dt\right]. \label{pf:yLLN1}
\end{align}
Note that $\E\int_{\C^d}\|z\|^{p'}_T\mu^x(dz) < \infty$ and that $\E[|Y^i_0|^{p'}] = \E[|Y^1_0|^{p'}] < \infty$. Apply \eqref{pf:yLLN1} also with $\beta = \Lambda^i$, we find a new constant, still called $C$ and still independent of $n$, such that
\[
\E\left[d^{p'}_\X\left((W^i,\Lambda^i,Y^i),(W^i,\beta,Y^i[\beta])\right)\right] \le C\left(1 + \E\int_0^T\int_A|a|^{p'}\beta_t(da)dt\right).
\]
Finally, recall that $(W^k,\Lambda^k,Y^k)_{k=1}^\infty$ are conditionally i.i.d. given $(B,\mu)$ with common conditional law $\mu$. Since also they are $p'$-integrable, it follows from the law of large numbers that
\[
\lim_{n\rightarrow\infty}\E\left[\ell^{p'}_{\X,p'}(\widehat{\nu}^n,\mu)\right] = 0.
\]
Complete the proof by using the triangle inequality to get
\[
\E\left[\ell^{p'}_{\X,p'}(\widehat{\nu}^{n,i}[\beta],\mu)\right] \le \frac{C2^{p'-1}}{n}\left(1 + \E\int_0^T\int_A|a|^{p'}\beta_t(da)dt\right) + 2^{p'-1}\E\left[\ell^{p'}_{\X,p'}(\widehat{\nu}^n,\mu)\right].
\]
\end{proof}

\begin{lemma} \label{le:pathwisepropagation}
There is a sequence $\delta_n > 0$ converging to zero such that for each $1 \le i \le n$ and each $\beta \in \A_n(\mathcal{E}_n)$,
\[
\E^\PP\left[\ell_{\X,p'}^{p'}(\widehat{\mu}[(\overline{\Lambda}^{n,-i},\beta)],\mu) + \left\|X^i[(\overline{\Lambda}^{n,-i},\beta)] - Y^i[\beta]\right\|_T^{p'}\right] \le \delta_n\left(1 + \E^\PP\int_0^T\int_A|a|^{p'}\beta_t(da)dt\right).
\]
\end{lemma}
\begin{proof}
The proof is similar to that of Lemma \ref{le:modifiedsystem}, and we work again with the truncated $p'$-Wasserstein distances $\ell_t$ on $\C^d$ defined in \eqref{def:truncatedwasserstein}. Throughout this proof, $n$ and $i$ are fixed, and expectations are all with respect to $\PP$. Abbreviate $\overline{X}^k = X^k[(\overline{\Lambda}^{n,-i},\beta)]$ and $\widehat{\mu} = \widehat{\mu}[(\overline{\Lambda}^{n,-i},\beta)]$ throughout. Define $\overline{Y}^i := Y^i[\beta]$ and $\overline{Y}^k := Y^k$ for $k \neq i$. As in the proof of Lemma \ref{le:modifiedsystem}, we use the Burkholder-Davis-Gundy inequality 
followed by Gronwall's inequality to find a constant $C_1 > 0$, depending only on $c_1$, $p'$, and $T$, such that
\begin{align}
\E\left[\|\overline{X}^k - \overline{Y}^k\|^{p'}_t\right] \le C_1\E\int_0^t\ell^{p'}_s(\widehat{\mu}^x,\mu^x)ds, \text{ for } 1 \le k \le n. \label{pf:pathwisepropagation1}
\end{align}
Define $\widehat{\nu}^{n,i} = \widehat{\nu}^{n,i}[\beta]$ as in Lemma \ref{le:yLLN}, and write $\widehat{\nu}^{n,i,x} := (\widehat{\nu}^{n,i})^x$ for the empirical distribution of $(\overline{Y}^1,\ldots,\overline{Y}^n)$. 
Use \eqref{pf:pathwisepropagation1} and the triangle inequality to get
\begin{align*}
\frac{1}{n}\sum_{k=1}^n\E\left[\|\overline{X}^k - \overline{Y}^k\|^{p'}_t\right]  &\le 2^{p'-1}C_1\E\int_0^t\left(\ell^{p'}_s(\widehat{\mu}^x,\widehat{\nu}^{n,i,x}) + \ell^{p'}_s(\widehat{\nu}^{n,i,x},\mu^x) \right)ds \\
	&\le 2^{p'-1}C_1\E\int_0^t\left(\frac{1}{n}\sum_{k=1}^n\|\overline{X}^k - \overline{Y}^k\|^{p'}_s + \ell^{p'}_s(\widehat{\nu}^{n,i,x},\mu^x) \right)ds
\end{align*}
By Gronwall's inequality and Lemma \ref{le:yLLN}, with $C_2 := 2^{p'-1}C_1e^{2^{p'-1}C_1T}$ we have
\begin{align}
\frac{1}{n}\sum_{k=1}^n\E\left[\|\overline{X}^k - \overline{Y}^k\|^{p'}_t\right] &\le  C_2\E\int_0^t\ell^{p'}_s(\widehat{\nu}^{n,i,x},\mu^x)ds \le C_2 T\E\left[\ell^{p'}_{\X,p'}(\widehat{\nu}^{n,i},\mu)\right]  \nonumber \\
	&\le C_2 T\delta_n\left(1 + \E\int_0^T\int_A|a|^{p'}\beta_t(da)dt\right).  \label{pf:pathwisepropagation2}
\end{align}
The obvious coupling yields the inequality
\begin{align*}
\ell^{p'}_{\X,p'}(\widehat{\mu},\widehat{\nu}^{n,i}) \le \frac{1}{n}\sum_{k=1}^n\|\overline{X}^k - \overline{Y}^k\|^{p'}_T,
\end{align*}
and then the triangle inequality implies
\begin{align*}
\E\left[\ell_{\X,p'}^{p'}(\widehat{\mu},\mu)\right] \le 2^{p'-1}\frac{1}{n}\sum_{k=1}^n\E\left[\|\overline{X}^k - \overline{Y}^k\|^{p'}_T\right] + 2^{p'-1}\E\left[\ell^{p'}_{\X,p'}(\widehat{\nu}^{n,i},\mu)\right].
\end{align*}
Conclude from Lemma \ref{le:yLLN} and \eqref{pf:pathwisepropagation2}.
\end{proof}

\subsection{Proof of Theorem \ref{th:partialconverseconvergence}}
With Lemma \ref{le:pathwisepropagation} in hand, we begin the proof of Theorem \ref{th:partialconverseconvergence}. The convergence $P_n \rightarrow P$ follows immediately from Lemma \ref{le:pathwisepropagation}, and it remains only to check that $\overline{\Lambda}^n$ is a relaxed $(\epsilon_n,\ldots,\epsilon_n)$-Nash equilibrium for some $\epsilon_n \rightarrow 0$. Define
\begin{align*}
\epsilon_n &:= \max_{i=1}^n\left[\sup_{\beta \in \A_n(\mathcal{E}_n)}J_i((\overline{\Lambda}^{n,-i},\beta)) - J_i(\overline{\Lambda}^n)\right] \\
  &= \sup_{\beta \in \A_n(\mathcal{E}_n)}J_1((\overline{\Lambda}^{n,-1},\beta)) - J_1(\overline{\Lambda}^n),
\end{align*}
where the second equality follows from exchangeability, or more precisely from the fact that (using the notation of Remark \ref{re:exchangeability}) the measure
\[
\PP \circ \left(\xi_\pi,B,W_\pi,\widehat{\mu}[\overline{\Lambda}^n_\pi],\overline{\Lambda}^n_\pi,X[\overline{\Lambda}^n_\pi]_\pi\right)^{-1}
\]
does not depend on the choice of permutation $\pi$.
Recall that $P \in \P(\Omega)$ was the given MFG solution, and define $\rho := P \circ (\xi,B,W,\mu)^{-1}$ so that $P \in \RC\A^{*}(\rho)$.
For each $n$, find $\beta^n \in \A_n(\mathcal{E}_n)$ such that 
\begin{align}
J_1((\overline{\Lambda}^{n,-1},\beta^n)) \ge \sup_{\beta \in \A_n(\mathcal{E}_n)}J_1((\overline{\Lambda}^{n,-1},\beta)) - 1/n. \label{pf:converse0}
\end{align}
To complete the proof, it suffices to prove the following:
\begin{align}
&\lim_{n\rightarrow\infty}J_1(\overline{\Lambda}^n) = \E^\PP\left[\Gamma(\mu^x,\Lambda^1,Y^1)\right], \label{pf:converse1}\\
&\lim_{n\rightarrow\infty}\left|\E^\PP\left[\Gamma(\widehat{\mu}^x[(\overline{\Lambda}^{n,-1},\beta^n)],\beta^n,\right.\left.X^1[(\overline{\Lambda}^{n,-1},\beta^n)]) - \Gamma(\mu^x,\beta^n,Y^1[\beta^n])\right]\right| = 0. \label{pf:converse2}
\end{align}
Indeed, note that $\PP \circ \left(\xi^1,B,W^1,\mu,\Lambda^1,Y^1\right)^{-1} = P$ holds by construction. Since
\[
P'_n := \PP \circ \left(\xi^1,B,W^1,\mu,\beta^n,Y^1[\beta^n]\right)^{-1}
\]
is in $\RC\A(\rho)$ for each $n$, and since $P$ is in $\RC\A^{*}(\rho)$, we have
\[
\E^\PP\left[\Gamma(\mu^x,\beta^n,Y^1)\right] = J(P) \ge J(P'_n) = \E^\PP\left[\Gamma(\mu^x,\beta^n,Y^1[\beta^n])\right], \text{ for all } n.
\]
Thus, from \eqref{pf:converse1} and \eqref{pf:converse2} it follows that
\begin{align*}
\lim_{n\rightarrow\infty}J_1(\overline{\Lambda}^n) &\ge \limsup_{n\rightarrow\infty}\E^\PP\left[\Gamma(\mu^x,\beta^n,Y^1[\beta^n])\right] \\
	&= \limsup_{n\rightarrow\infty}J_1((\overline{\Lambda}^{n,-1},\beta^n)) \\
	&= \limsup_{n\rightarrow\infty}\sup_{\beta \in \A_n(\mathcal{E}_n)}J_1((\overline{\Lambda}^{n,-1},\beta)),
\end{align*}
where of course in the last step we have used \eqref{pf:converse0}. Since $\epsilon_n \ge 0$, this shows $\epsilon_n \rightarrow 0$.

\subsubsection*{Proof of \eqref{pf:converse1}:}
First, apply Lemma \ref{le:pathwisepropagation} with $\beta = \Lambda^1$ (so that $(\overline{\Lambda}^{n,-1},\beta) = \overline{\Lambda}^n$) to get
\begin{align*}
\lim_{n\rightarrow\infty}\PP \circ \left(Y^1_0,B,W^1,\widehat{\mu}[\overline{\Lambda}^n],\Lambda^1,X^1[\overline{\Lambda}^n]\right)^{-1} &= \PP \circ \left(Y^1_0,B,W^1,\mu,\Lambda^1,Y^1\right)^{-1},
\end{align*}
where the limit is taken in $\P^p(\Omega)$. Moreover, since $\E^\PP\int_0^T\int_A|a|^{p'}\Lambda^1_t(da)dt < \infty$, we use the continuity of $J$ of Lemma \ref{le:jcontinuous} (since the additional uniform integrability condition holds trivially) to conclude that
\begin{align*}
\lim_{n\rightarrow\infty}J_1(\overline{\Lambda}^n)	&= \lim_{n\rightarrow\infty}\E^\PP\left[\Gamma(\widehat{\mu}^x[\overline{\Lambda}^n],\Lambda^1,X^1[\overline{\Lambda}^n])\right] = \E^\PP\left[\Gamma(\mu^x,\Lambda^1,Y^1)\right].
\end{align*}

\subsubsection*{Proof of \eqref{pf:converse2}:} 
This step is fairly involved and thus divided into several steps. The first two steps identify a relative compactness for the laws of the empirical measure and state process pairs, crucial for the third and fourth steps below. Step (3) focuses on the $g$ term, and Step (4) uses the additional assumption \ref{assumption:B} to deal with the $f$ term.

\subsubsection*{Proof of \eqref{pf:converse2}, Step (1):}
We show first that 
\begin{align}
\sup_n\E\int_0^T\int_A|a|^{p'}\beta^n_t(da)dt < \infty. \label{pf:converse3}
\end{align}
By \eqref{pf:converse0} and Lemma \ref{le:valuebound}(2), we have
\begin{align*}
\E\int_0^T\int_A(|a|^{p'} - c_6|a|^p)\beta^n_t(da)dt &\le c_7\E\left[1 + \frac{1}{n} + |\xi^1|^p + \frac{1}{n}\sum_{i=2}^n\int_0^T\int_A|a|^p\Lambda^i_t(da)dt\right] \\
	&= c_7\E\left[1 + \frac{1}{n} + |\xi^1|^p + \frac{n-1}{n}\int_0^T\int_A|a|^p\Lambda^1_t(da)dt\right],
\end{align*}
where the second line follows from symmetry. Since $\E[|\xi^1|^p] < \infty$ and $\E\int_0^T\int_A|a|^p\Lambda^1_t(da)dt < \infty$, we have proven \eqref{pf:converse3}.

\subsubsection*{Proof of \eqref{pf:converse2}, Step (2):}
Define $\A_R$ for $R > 0$ to be the set of $(\overline{\F}^\infty_t)_{t \in [0,T]}$-progressive $\P(A)$-valued processes $\beta$ such that 
\[
\E\int_0^T\int_A|a|^{p'}\beta_t(da)dt \le R.
\]
According to \eqref{pf:converse3}, there exists $R > 0$ such that $\beta^n \in \A_R$ for all $n$. Define also
\begin{align*}
S_R := \left\{\PP \circ \left(\widehat{\mu}^x[(\overline{\Lambda}^{n,-1},\beta)],X^1[(\overline{\Lambda}^{n,-1},\beta)]\right)^{-1} : n \ge 1, \beta \in \A_R \right\}.
\end{align*}
We show next that $S_R$ is relatively compact in $\P^p(\P^p(\C^d) \times \C^d)$. Note first that it follows from Lemma \ref{le:finitestateestimate} that
\begin{align}
\sup\left\{\E^\PP\int_{\C^d}\|z\|_T^{p'}\widehat{\mu}^x[(\overline{\Lambda}^{n,-1},\beta)](dz) : n \ge 1, \ \beta \in \A_R\right\} < \infty. \label{pf:converse5}
\end{align}
By symmetry, we have
\[
\left\{ \PP \circ (X^1[(\overline{\Lambda}^{n,-1},\beta)])^{-1} : n \ge 1, \ \beta \in \A_R \right\} = \left\{ \frac{1}{n}\sum_{k=1}^n\PP \circ (X^k[(\overline{\Lambda}^{n,-k},\beta)])^{-1} : n \ge 1, \ \beta \in \A_R \right\},
\]
and by Proposition \ref{pr:itocompact} this set is relatively compact in $\P^p(\C^d)$. For $\beta \in \A_R$, the mean measure of $\PP \circ (\widehat{\mu}^x[(\overline{\Lambda}^{n,-1},\beta)])^{-1}$ is exactly
\begin{align*}
\frac{1}{n}\sum_{k=1}^n\PP \circ (X^k[(\overline{\Lambda}^{n,-1},\beta)])^{-1},
\end{align*}
and it follows again from Proposition \ref{pr:itocompact} that the family
\[
\left\{\frac{1}{n}\sum_{k=1}^n\PP \circ (X^k[(\overline{\Lambda}^{n,-1},\beta)])^{-1} : n \ge 1, \beta \in \A_R\right\}
\]
is relatively compact in $\P^p(\C^d)$. From this and \eqref{pf:converse5} we conclude that $\PP \circ (\widehat{\mu}^x[(\overline{\Lambda}^{n,-1},\beta)])^{-1}$ are relatively compact in $\P^p(\P^p(\C^d))$. Hence, $S_R$ is relatively compact. (See Corollary B.2 and Lemma A.2 of \cite{lacker-mfgcontrolledmartingaleproblems} regarding these last two conclusions.)

\subsubsection*{Proof of \eqref{pf:converse2}, Step (3):}
Since $\beta^n \in \A_R$ for each $n$, to prove \eqref{pf:converse2} it suffices to show that
\begin{align}
\sup_{\beta \in \A_R}I^\beta_n \rightarrow 0, \label{pf:converse3.1}
\end{align}
where
\begin{align*}
I^\beta_n &:= \E\left[\Gamma(\widehat{\mu}^x[(\overline{\Lambda}^{n,-1},\beta)],\beta,X^1[(\overline{\Lambda}^{n,-1},\beta)]) - \Gamma(\mu^x,\beta,Y^1[\beta])\right] \\
	&= \E\left[\int_0^T\int_A\left(f(t,X^1_t[(\overline{\Lambda}^{n,-1},\beta)],\widehat{\mu}^x_t[(\overline{\Lambda}^{n,-1},\beta)],a) - f(t,Y^1_t[\beta],\mu^x_t,a)\right)\beta_t(da)dt\right] \\
	&\quad + \E\left[g(X^1_T[(\overline{\Lambda}^{n,-1},\beta)],\widehat{\mu}^x_T[(\overline{\Lambda}^{n,-1},\beta)]) - g(Y^1_T[\beta],\mu^x_T)\right].
\end{align*}
We start with the $g$ term.
Define
\begin{align*}
Q_n^\beta &:= \PP \circ (\widehat{\mu}^x[(\overline{\Lambda}^{n,-1},\beta)],X^1[(\overline{\Lambda}^{n,-1},\beta)])^{-1}, \\
Q^\beta &:= \PP \circ (\mu^x,Y^1[\beta])^{-1}.
\end{align*}
Using the metric on $\P^p(\C^d) \times \C^d$ given by
\[
((\mu,x),(\mu',x')) \mapsto \left[\ell^p_{\C^d,p}(\mu,\mu') + \|x-x'\|_T^p\right]^{1/p},
\]
we define the $p$-Wasserstein metric $\ell_{\P^p(\C^d) \times \C^d,p}$ on $\P^p(\P^p(\C^d) \times \C^d)$.
By Lemma \ref{le:pathwisepropagation}, we have
\begin{align*}
\ell_{\P^p(\C^d) \times \C^d,p}^{p'}(Q_n^\beta,Q^\beta) &\le \E\left[\ell^p_{\C^d,p}\left(\widehat{\mu}^x[(\overline{\Lambda}^{n,-1},\beta)],\mu^x\right) + \|X^1[(\overline{\Lambda}^{n,-1},\beta)] - Y^1[\beta]\|_T^p\right]^{p'/p} \\
	&\le 2^{p'/p-1}\E\left[\ell^{p'}_{\X,p'}\left(\widehat{\mu}[(\overline{\Lambda}^{n,-1},\beta)],\mu\right) + \|X^1[(\overline{\Lambda}^{n,-1},\beta)] - Y^1[\beta]\|_T^{p'}\right] \\
	&\le 2^{p'/p-1}\delta_n(1 + R),
\end{align*}
and thus $Q_n^\beta \rightarrow Q^\beta$ in $\P^p(\P^p(\C^d) \times \C^d)$, uniformly in $\beta \in \A_R$. The function
\[
\P^p(\P^p(\C^d) \times \C^d) \ni Q \mapsto \int Q(d\nu,dx)g(x_T,\nu_T)
\]
is continuous, and so its restriction to the closure of $S_R$ is \emph{uniformly} continuous. Thus, since $\{Q_n^\beta : n \ge 1, \ \beta \in \A_R\} \subset S_R$,
\[
\lim_{n\rightarrow\infty}\sup_{\beta \in \A_R}\left|\E\left[g(X^1_T[(\overline{\Lambda}^{n,-1},\beta)],\widehat{\mu}^x_T[(\overline{\Lambda}^{n,-1},\beta)]) - g(Y^1_T[\beta],\mu^x_T)\right]\right| = 0.
\]

\subsubsection*{Proof of \eqref{pf:converse2}, Step (4):}
To deal with the $f$ term in $I^\beta_n$ it will be useful to define $G : \P^p(\C^d) \times \C^d \rightarrow \R$ by
\[
G\left((\mu^1,x^1),(\mu^2,x^2)\right) := \int_0^T\sup_{a \in A}\left|f(t,x^1_t,\mu^1_t,a) - f(t,x^2_t,\mu^2_t,a)\right|dt
\]
With the $g$ term taken care of in Step (3) above, the proof of \eqref{pf:converse3.1} and thus the theorem will be complete if we show that
\begin{align}
0 &= \lim_{n\rightarrow\infty}\sup_{\beta \in \A_R}\E\left[Z^n_\beta\right], \text{ where} \label{pf:converse4.1} \\
Z^n_\beta &:= G\left((\widehat{\mu}^x[(\overline{\Lambda}^{n,-1},\beta)],X^1[(\overline{\Lambda}^{n,-1},\beta)]),(\mu^x,Y^1[\beta])\right). \nonumber
\end{align}
Fix $\eta > 0$, and note that by relative compactness of $S_R$ we may find (e.g. by \cite[Theorem 7.12]{villanibook}) a compact set $K \subset \P^p(\C^d) \times \C^d$ such that, if the event $K_\beta$ is defined by
\[
K_\beta := \left\{\left(\widehat{\mu}^x[(\overline{\Lambda}^{n,-1},\beta)],X^1[(\overline{\Lambda}^{n,-1},\beta)]\right) \in K\right\},
\]
then
\[
\E\left[\left(1 + \int_{\C^d}\|z\|^p_T\widehat{\mu}^x[(\overline{\Lambda}^{n,-1},\beta)](dz) + \|X^1[(\overline{\Lambda}^{n,-1},\beta)\|_T^p\right)1_{K^c_\beta}\right] \le \eta,
\]
for all $n \ge 1$ and $\beta \in \A_R$. Sending $n \rightarrow \infty$, it follows from Lemma \ref{le:pathwisepropagation} that also 
\[
\E\left[\left(1 + \int_{\C^d}\|z\|^p_T\mu^x(dz) + \|Y^1[\beta]\|_T^p\right)1_{K^c_\beta}\right] \le \eta.
\]
Hence, the growth condition of Assumption \ref{assumption:B} implies
\begin{align}
\E\left[1_{K^c_\beta}Z^n_\beta\right] \le c_4\eta, \label{pf:converse6}
\end{align}
for all $n \ge 1$ and $\beta \in \A_R$.
Assumption \ref{assumption:B} implies that $G$ is continuous, and thus uniformly continuous on $K \times K$. We will check next that $\E[1_{K_\beta}Z^n_\beta]$
converges to zero, uniformly in $\beta \in \A_R$. Indeed, by uniform continuity there exists $\eta_0 > 0$ such that if $(\mu^1,x^1),(\mu^2,x^2) \in K$ and $G((\mu^1,x^1),(\mu^2,x^2)) > \eta$ then $\|x^1-x^2\|_T + \ell_{\C^d,p}(\mu^1,\mu^2) > \eta_0$. Thus, since $G$ is bounded on $K \times K$, say by $C > 0$, we use Markov's inequality and Lemma \ref{le:pathwisepropagation} to conclude that
\begin{align*}
\E\left[1_{K_\beta}Z^n_\beta\right] &\le \eta + C\PP\left\{\left\|X^1[(\overline{\Lambda}^{n,-1},\beta)] - Y^1[\beta]\right\|_T + \ell_{\C^d,p}\left(\widehat{\mu}^x[(\overline{\Lambda}^{n,-1},\beta)],\mu^x\right) > \eta_0\right\} \\
	&\le \eta + 2^{p'-1}C\eta_0^{-p'}\E\left[\left\|X^1[(\overline{\Lambda}^{n,-1},\beta)] - Y^1[\beta]\right\|_T^{p'} + \ell_{\C^d,p}^{p'}\left(\widehat{\mu}^x[(\overline{\Lambda}^{n,-1},\beta)],\mu^x\right)\right] \\
	&\le \eta + 2^{p'-1}C\eta_0^{-p'}\delta_n\left(1 + \E\int_0^T\int_A|a|^{p'}\beta_t(da)dt\right) \\
	&\le \eta + 2^{p'-1}C\eta_0^{-p'}\delta_n(1+R),
\end{align*}
whenever $\beta \in \A_R$, where $\delta_n \rightarrow 0$ is from Lemma \ref{le:pathwisepropagation}. Combining this with \eqref{pf:converse6}, we get
\[
\limsup_{n\rightarrow\infty}\sup_{\beta \in \A_R}\E\left[Z^n_\beta\right] \le (1+c_4)\eta.
\]
This holds for each $\eta > 0$, completing the proof of \eqref{pf:converse4.1} and thus of the theorem. \hfill \qedsymbol

\subsection{Proof of Proposition \ref{pr:fixednapproximation}} \label{se:fixednapproximation}

Throughout the section, the number of agents $n$ is fixed, and we work on the $n$-player environment $\mathcal{E}_n$ specified in Section \ref{se:environments}. The proof of Proposition \ref{pr:fixednapproximation} is split into two main steps. In this first step, we approximate the relaxed strategy $\Lambda^0$ by bounded strong strategies, and we check the convergences (1) and (2) claimed in Proposition \ref{pr:fixednapproximation}. The second step verifies the somewhat more subtle inequality (3) of Proposition \ref{pr:fixednapproximation}.

\begin{remark} \label{re:adapteddensity}
Propositions \ref{pr:fixednapproximation}, \ref{pr:equilibriuminclusions}, and \ref{pr:verystrongeq} are really just instances of the density of strong (and strict) controls in the class of weak controls, in a sense made precise by Lemma \ref{le:adapteddensity}. Indeed, a consequence of Lemma \ref{le:adapteddensity} may be stated more transparently as follows. Suppose $(\widetilde{\Omega},(\F_t)_{t \in [0,T]},P)$ is a filtered probability space supporting a $(\F_t)_{t \in [0,T]}$-Wiener process $\widetilde{W}$ (of any dimension), an $\F_0$-measurable random variable $\widetilde{\xi}$ living in some Euclidean space, and a progressively measurable $\P(A)$-valued process $(\widetilde{\Lambda}_t)_{t \in [0,T]}$, satisfying $\E^P\int_0^T\int_A|a|^{p'}\Lambda_t(da)dt < \infty$. Then, if $\G_t := \sigma(\widetilde{\xi},\widetilde{W}_s : s \le t)$, then there exists a sequence $(\alpha^k)_{k=1}^\infty$ of $(\G_t)_{t \in [0,T]}$-progressively measurable $A$-valued processes such that
\[
\lim_{k\rightarrow\infty} P \circ \left(\widetilde{\xi},\widetilde{W},dt\delta_{\alpha^k_t}(da)\right)^{-1} = P \circ \left(\widetilde{\xi},\widetilde{W},dt\widetilde{\Lambda}_t(da)\right)^{-1},
\]
and
\[
\lim_{r\rightarrow\infty}\sup_k\E^P\left[\int_0^T|\alpha^k_t|^{p'}1_{\{|\alpha^k_t| > r\}}dt\right] < \infty.
\]
\end{remark}

Before we prove Proposition \ref{pr:fixednapproximation}, we need the following lemma, which is a simple variant of a standard result:

\begin{lemma} \label{le:finitestateconvergence}
Suppose $\widetilde{\Lambda}^k = (\widetilde{\Lambda}^{k,1},\ldots,\widetilde{\Lambda}^{k,n}) \in \A_n^n(\mathcal{E}_n)$ is such that
\[
\lim_{k\rightarrow\infty}\PP \circ (\xi,B,W,\widetilde{\Lambda}^k)^{-1} = \PP \circ (\xi,B,W,\Lambda^0)^{-1},
\]
with the limit taken in $\P^p((\R^d)^n \times \C^{m_0} \times (\C^m)^n \times \V^n)$. Then 
\[
\lim_{k\rightarrow\infty}\PP \circ \left(B,W,\widetilde{\Lambda}^k,X[\widetilde{\Lambda}^k]\right)^{-1} = \PP \circ \left(B,W,\Lambda^0,X[\Lambda^0]\right)^{-1},
\]
in $\P^p(\C^{m_0} \times (\C^m)^n \times \V^n \times (\C^d)^n)$.
\end{lemma}
\begin{proof}
This is analogous to the proof of Lemma \ref{le:rccontinuous}, given in \cite{carmonadelaruelacker-mfgcommonnoise}, which is itself an instance of a standard method proving weak convergence of SDE solutions, so we only sketch the proof. It can be shown as in Proposition \ref{pr:itocompact} that $\{\PP \circ (X[\widetilde{\Lambda}^k])^{-1} : k \ge 1\}$ is relatively compact in $\P^p((\C^d)^n)$, and thus $\{\PP \circ \left(B,W,\widetilde{\Lambda}^k,X[\widetilde{\Lambda}^k]\right)^{-1} : k \ge 1\}$ is relatively compact in $\P^p(\C^{m_0} \times (\C^m)^n \times \V^n \times (\C^d)^n)$. Using the results of Kurtz and Protter \cite{kurtzprotter-weakconvergence}, it is straightforward to check that under any limit point the canonical processes satsify a certain SDE, and the claimed convergence follows from uniqueness of the SDE solution.
\end{proof}

We are now ready to prove Proposition \ref{pr:fixednapproximation}.

\subsubsection*{Step 1:} 
Define $\overline{\V}$ analogously to $\V$, but with $A$ replaced by $A^n$. That is, $\overline{\V}$ is the set of measures $q$ on $[0,T] \times A^n$ with first marginal equal to Lebesgue measure and with
\[
\int_{[0,T] \times A^n}\sum_{i=1}^n|a_i|^pq(dt,da_1,\ldots,da_n) < \infty.
\]
Endow $\overline{\V}$ with the $p$-Wasserstein metric. Define
\[
\overline{\Lambda}^0_t(da_1,\ldots,da_n) := \prod_{i=1}^n\Lambda^{0,i}_t(da_i),
\]
and identify this $\P(A^n)$-valued process with the random element $\overline{\Lambda}^0 := dt\overline{\Lambda}^0_t(da)$ of $\overline{\V}$. By Lemma \ref{le:adapteddensity} (see also Remark \ref{re:adapteddensity}), with $A$ replaced by $A^n$, there exists a sequence of bounded $A^n$-valued processes $\alpha^k = (\alpha^{k,1},\ldots,\alpha^{k,n})$ such that, if we define
\[
\overline{\Lambda}^k := dt\delta_{\alpha^{k}_t}(da_1,\ldots,da_n) = dt\prod_{i=1}^n\delta_{\alpha^{k,i}_t}(da_i),
\]
then we have
\begin{align}
\lim_{r\rightarrow\infty}\sup_k\E^\PP\left[\int_0^T|\alpha^k_t|^{p'}1_{\{|\alpha^k_t| > r\}}dt \right] = 0 \label{pf:fpss1}
\end{align}
and
\[
\lim_{k\rightarrow \infty}\PP \circ \left(\xi,B,W,\overline{\Lambda}^{k}\right)^{-1} = \PP \circ \left(\xi,B,W,\overline{\Lambda}^0\right)^{-1},
\]
in $\P^{p}((\R^d)^n \times \C^{m_0} \times (\C^m)^n \times \overline{\V})$. Defining $\pi_i : [0,T] \times A^n \rightarrow [0,T] \times A$ by $\pi_i(t,a_1,\ldots,a_n) := (t,a_i)$, we note that the map $\overline{\V} \ni q \mapsto q \circ \pi_i^{-1} \in \V$ is continuous. Define $\Lambda^{k,i}_t := \delta_{\alpha^{k,i}_t}$ and $\Lambda^{k} = (\Lambda^{k,1},\ldots,\Lambda^{k,n})$, and conclude that 
\[
\lim_{k\rightarrow \infty}\PP \circ \left(\xi,B,W,\Lambda^{k}\right)^{-1} = \PP \circ \left(\xi,B,W,\Lambda^0\right)^{-1},
\]
in $\P^{p}((\R^d)^n \times \C^{m_0} \times (\C^m)^n \times \V^n)$, for each $k$. By Lemma \ref{le:finitestateconvergence},
\begin{align*}
\lim_{k\rightarrow\infty}\PP \circ \left(B,W,\Lambda^{k},X[\Lambda^{k}]\right)^{-1} = \PP \circ \left(B,W,\Lambda^0,X[\Lambda^0]\right)^{-1},
\end{align*}
in $\P^p(\C^{m_0} \times (\C^m)^n \times \V^n \times (\C^d)^n)$. It follows from the uniform integrability \eqref{pf:fpss1} and the continuity of $J$ of Lemma \ref{le:jcontinuous} that
\begin{align*}
\lim_{k\rightarrow\infty}J_i(\Lambda^{k}) = J_i(\Lambda^0), \ i=1,\ldots,n.
\end{align*}
This verifies (1) and (2) of Proposition \ref{pr:fixednapproximation}.

\subsubsection*{Step 2:}
It remains to justify the inequality (3) of Proposition \ref{pr:fixednapproximation}. We prove this only for $i=1$, since the cases $i=2,\ldots,n$ are identical. For each $k$ find $\beta^k \in \A_n(\mathcal{E}_n)$ such that
\begin{align}
J_i((\Lambda^{k,-1},\beta^k)) \ge \sup_{\beta \in 
\A_n(\mathcal{E}_n)}J_i((\Lambda^{k,-1},\beta)) - \frac{1}{k}. \label{pf:fpss5}
\end{align}
First, use Lemma \ref{le:valuebound}(2) to get
\begin{align*}
\E\int_0^T\int_A(|a|^{p'} - c_6|a|^p)\beta^k_t(da)dt &\le c_7\E\left[1 + \frac{1}{k} + |\xi^1|^p + \frac{1}{n}\sum_{i=2}^n\int_0^T\int_A|a|^p\Lambda^{k,i}_t(da)dt\right].
\end{align*}
Since $\E[|\xi^1|^p] < \infty$, and since
\[
\lim_{k\rightarrow\infty}\E\int_0^T\int_A|a|^p\Lambda^{k,i}_t(da)dt = \E\int_0^T\int_A|a|^p\Lambda^{0,i}_t(da)dt < \infty,
\]
holds by construction, for $i=2,\ldots,n$, it follows that
\[
R := \sup_k \E^{\PP}\int_0^T\int_A|a|^{p'}\beta^k_t(da)dt < \infty.
\]
It follows as in Proposition \ref{pr:itocompact} (or more precisely \cite[Proposition B.4]{lacker-mfgcontrolledmartingaleproblems}) that the set
\[
\left\{\PP \circ \left((\Lambda^{k,-1},\beta^k),X[(\Lambda^{k,-1},\beta^k)]\right)^{-1} : k \ge 1\right\}
\]
is relatively compact in $\P^p(\V^n \times (\C^d)^n)$. Hence, the set 
\begin{align}
\left\{P_k := \PP \circ \left(B,W,(\Lambda^{k,-1},\beta^k),X[(\Lambda^{k,-1},\beta^k)]\right)^{-1} : k \ge 1\right\} \label{pf:fpss6}
\end{align}
is relatively compact in $\P^p(\C^{m_0} \times (\C^m)^n \times \V^n \times (\C^d)^n)$ (e.g. by \cite[Lemma A.2]{lacker-mfgcontrolledmartingaleproblems}).
By the following Lemma \ref{le:limitrepresentation}, every limit point $P$ of $(P_k)_{k=1}^\infty$ is of the form
\begin{align}
P = \PP \circ \left(B,W,(\Lambda^{0,-1},\beta),X[(\Lambda^{0,-1},\beta)]\right)^{-1}, \text{ for some } \beta \in \A_n(\mathcal{E}_n). \label{pf:fpss7}
\end{align}
This implies
\[
\limsup_{k\rightarrow\infty}J_i((\Lambda^{k,-1},\beta^k)) \le \sup_{\beta \in \A_n(\mathcal{E}_n)}J_i((\Lambda^{0,-1},\beta)).
\]
Because of \eqref{pf:fpss5}, this completes the proof of Proposition \ref{pr:fixednapproximation}.

\begin{lemma} \label{le:limitrepresentation}
Every limit point $P$ of $(P_k)_{k=1}^\infty$ (defined in \eqref{pf:fpss6}) is of the form \eqref{pf:fpss7}.
\end{lemma}
\begin{proof}
Let us abbreviate
\[
\Omega^{(n)} := \C^{m_0} \times (\C^m)^n \times \V^n \times (\C^d)^n.
\]
Let $(B,W=(W^1,\ldots,W^n),\Lambda=(\Lambda^1,\ldots,\Lambda^n),X=(X^1,\ldots,X^n))$ denote the identity map on $\Omega^{(n)}$, and let $(\F^{(n)}_t)_{t \in [0,T]}$ denote the natural filtration,
\[
\F^{(n)}_t = \sigma\left((B_s,W_s,\Lambda([0,s] \times C),X_s) : s \le t, \ C \in \B(A)\right).
\]
Fix a limit point $P$ of $P_k$. 
It is easily verified that $P$ satisfies
\begin{align}
P \circ \left(X_0,B,W,(\Lambda^2,\ldots,\Lambda^n)\right)^{-1} &= \PP \circ \left(X_0,B,W,(\Lambda^{0,2},\ldots,\Lambda^{0,n})\right)^{-1}. \label{pf:limitrepresentation1}
\end{align}
Moreover, for each $k$, we know that $B$ and $W$ are independent $(\F^{(n)}_t)_{t \in [0,T]}$-Wiener processes under $P_k$, and thus this is true under $P$ as well. Note that $(B,W,(\Lambda^{k,-1},\beta^k),X[(\Lambda^{k,-1},\beta^k)])$ satisfy the state SDE under $\PP$, or equivalently under $P_k$ the canonical processes verify the SDE
\begin{align}
\begin{cases}
dX^i_t &= \int_Ab(t,X^i_t,\widehat{\mu}^x_t,a)\Lambda^i_t(da)dt + \sigma(t,X^i_t,\widehat{\mu}^x_t)dW^i_t + \sigma_0(t,X^i_t,\widehat{\mu}^x_t)dB_t, \ i=1,\ldots,n \\
\widehat{\mu}^x_t &= \frac{1}{n}\sum_{k=1}^n\delta_{X^k_t}.
\end{cases} \label{pf:SDE}
\end{align}
The results of Kurtz and Protter \cite{kurtzprotter-weakconvergence} imply that this passes to the limit: The canonical processes on $\Omega^{(n)}$ verify the same SDE under $P$.

It remains only to show that there exists $\beta \in \A_n(\mathcal{E}_n)$ such that 
\begin{align}
\PP \circ (X_0,B,W,(\Lambda^{0,-1},\beta))^{-1} = P \circ (X_0,B,W,\Lambda)^{-1}. \label{pf:limitrepresentation2}
\end{align}
Indeed, from uniqueness in law of the solution of the SDE \eqref{pf:SDE} it will then follow that
\[
\PP \circ (B,W,(\Lambda^{0,-1},\beta),X[(\Lambda^{0,-1},\beta)])^{-1} = P.
\]
The independent uniform random variable $U$ built into $\mathcal{E}_n$ now finally comes into play. 
Using a well known result from measure theory (e.g. \cite[Theorem 5.10]{kallenberg-foundations}) we may find a measurable function
\[
\overline{\beta} = (\overline{\beta}^1,\ldots,\overline{\beta}^n) : [0,1] \times (\R^d)^n \times \C^{m_0} \times (\C^m)^n \rightarrow \V^n
\]
such that
\begin{align}
\PP \circ \left(X_0,B,W,\overline{\beta}(U,X_0,B,W)\right)^{-1} = P \circ (X_0,B,W,\Lambda)^{-1}. \label{pf:limitrepresentation3}
\end{align}
Since $B$ and $W$ are independent $(\F^{(n)}_t)_{t \in [0,T]}$-Wiener processes under $P$, it follows that
\[
(\overline{\beta}(U,X_0,B,W)_s)_{s \in [0,t]} \quad \text{and} \quad \sigma(B_s-B_t,W_s-W_t : s \in [t,T])
\]
are independent under $\PP$, for each $t \in [0,T]$. Thus, $(\overline{\beta}(U,X_0,B,W)_t)_{t \in [0,T]}$ is progressively measurable with respect to the $\PP$-completion of the filtration $(\sigma(U,X_0,B_s,W_s : s \le t))_{t \in [0,T]}$. In particular, $(\overline{\beta}(U,X_0,B,W))_{t \in [0,T]} \in \A_n^n(\mathcal{E}_n)$ and $\beta := (\overline{\beta}^1(U,X_0,B,W)_t)_{t \in [0,T]}$ is in $\A_n(\mathcal{E}_n)$. Now note that \eqref{pf:limitrepresentation1} and \eqref{pf:limitrepresentation3} together imply
\begin{align*}
\PP \circ \left(X_0,B,W,\left(\overline{\beta}^2(U,X_0,B,W),\ldots,\overline{\beta}^n(U,X_0,B,W)\right)\right)^{-1} = P \circ \left(X_0,B,W,(\Lambda^2,\ldots,\Lambda^n)\right)^{-1}.
\end{align*}
On the other hand, \eqref{pf:limitrepresentation3} implies that the conditional law under $P$ of $\Lambda^1$ given $(X_0,B,W,\Lambda^2,\ldots,\Lambda^n)$ is the same as the conditional law under $\PP$ of $\overline{\beta}^1(U,X_0,B,W)$ given 
\[
\left(X_0,B,W,\overline{\beta}^2(U,X_0,B,W),\ldots,\overline{\beta}^n(U,X_0,B,W)\right).
\]
This completes the proof of \eqref{pf:limitrepresentation2}.
\end{proof}

\section{Proof of Theorem \ref{th:withoutcommonnoise}} \label{se:withoutcommonnoiseproof}
This section explains the proof of Theorem \ref{th:withoutcommonnoise}, which specializes the main results to the setting without common noise essentially by means of the following simple observation. Note that although we assume $\sigma_0\equiv 0$ throughout the section, \emph{weak MFG solution} has the same meaning as in Definition \ref{def:weakMFGsolution}, distinct from Definition \ref{def:mfgsolutionwithoutcommonnoise} of \emph{weak MFG solution without common noise}.

\begin{lemma} \label{le:with-withoutcommonnoise}
If $(\widetilde{\Omega},(\F_t)_{t \in [0,T]},P,B,W,\mu,\Lambda,X)$ is a weak MFG solution, then $(\widetilde{\Omega},(\F_t)_{t \in [0,T]},P,W,\mu,$ $\Lambda,X)$ is a weak MFG solution without common noise. Conversely, if $(\widetilde{\Omega},(\F_t)_{t \in [0,T]},P,W,\mu,\Lambda,X)$ is a weak MFG solution without common noise, then we may construct (by enlarging the probability space, if necessary) an $m_0$-dimensional Wiener process $B$ independent of $(W,\mu,\Lambda,X)$ such that $(\widetilde{\Omega},(\F_t)_{t \in [0,T]},P,B,W,\mu,\Lambda,X)$ is a weak MFG solution.
\end{lemma}
\begin{proof}
The only difficulty comes from the conditional independence required in condition (3) of both Definitions \ref{def:weakMFGsolution} and \ref{def:mfgsolutionwithoutcommonnoise}, and it is convenient here to reformulate the definitions slightly.
Lemma \ref{le:presolution} tells us that Definition \ref{def:weakMFGsolution} of a weak MFG solution is equivalent to an alternative definition, in which the conditional independence is omitted from condition (3) and is added to condition (5). To be precise, define the following conditions:
\begin{enumerate}
\item[(3.a)] $(\Lambda_t)_{t \in [0,T]}$ is $(\F_t)_{t \in [0,T]}$-progressively measurable with values in $\P(A)$ and
\[
\E^P\int_0^T\int_A|a|^p\Lambda_t(da)dt < \infty.
\]
\item[(5.a)] Suppose $(\widetilde{\Omega}',(\F'_t)_{t \in [0,T]},P')$ is another filtered probability space supporting $(B',W',\mu',\Lambda',X')$ satisfying (3.a), (1,2,4) of Definition \ref{def:weakMFGsolution}, and $P \circ (B,\mu)^{-1} = P' \circ (B',\mu')^{-1}$, with $\sigma(\Lambda'_s : s \le t)$ conditionally independent of $\F^{X'_0,B',W',\mu'}_T$ given $\F^{X'_0,B',W',\mu'}_t$, for each $t \in [0,T]$. Then
\begin{align*}
\E^P[\Gamma(\mu^x,\Lambda,X)] \ge \E^{P'}[\Gamma(\mu'^x,\Lambda',X')].
\end{align*}
\end{enumerate}
Then, by Lemma \ref{le:presolution}, $(\widetilde{\Omega},(\F_t)_{t \in [0,T]},P,B,W,\mu,\Lambda,X)$ is a weak MFG solution if and only if it satisfies Definition \ref{def:weakMFGsolution} with conditions (3) and (5) replaced by (3.a) and (5.a). In fact, the same is true if (5.a) is replaced by
\begin{enumerate}
\item[(5'.a)] If $(\Lambda'_t)_{t \in [0,T]}$ is $(\F^{X_0,B,W,\mu}_t)_{t \in [0,T]}$-progressively measurable with values in $\P(A)$ and
\[
\E^P\int_0^T\int_A|a|^p\Lambda'_t(da)dt < \infty,
\]
and if $X'$ is the unique strong solution of 
\begin{align}
dX'_t = \int_Ab(t,X'_t,\mu^x_t,a)\Lambda'_t(da)dt + \sigma(t,X'_t,\mu^x_t)dW_t, \ X'_0 = X_0, \label{def:X'SDE}
\end{align}
then $\E^P[\Gamma(\mu^x,\Lambda,X)] \ge \E^P[\Gamma(\mu^x,\Lambda',X')]$.
\end{enumerate}
Indeed, this follows from the density of strong controls provided by Lemma \ref{le:adapteddensity} (see also Remark \ref{re:adapteddensity}). Analogously, for the setting without common noise, consider the following condition:
\begin{enumerate}
\item[(5'.b)] If $(\Lambda'_t)_{t \in [0,T]}$ is $(\F^{X_0,W,\mu}_t)_{t \in [0,T]}$-progressively measurable with values in $\P(A)$ and
\[
\E^P\int_0^T\int_A|a|^p\Lambda'_t(da)dt < \infty,
\]
and if $X'$ is the unique strong solution of \eqref{def:X'SDE},
then $\E^P[\Gamma(\mu^x,\Lambda,X)] \ge \E^P[\Gamma(\mu^x,\Lambda',X')]$.
\end{enumerate}
It is proven exactly as in Lemma \ref{le:presolution} that $(\widetilde{\Omega},(\F_t)_{t \in [0,T]},P,W,\mu,\Lambda,X)$ is a weak MFG solution without common noise if and only if it satisfies Definition \ref{def:mfgsolutionwithoutcommonnoise} with conditions (3) and (5) replaced by (3.a) and (5'.b). We are now ready to prove the lemma:

Suppose $(\widetilde{\Omega},(\F_t)_{t \in [0,T]},P,B,W,\mu,\Lambda,X)$ is a weak MFG solution. It is straightforward to check that $(\widetilde{\Omega},(\F_t)_{t \in [0,T]},P,W,\mu,\Lambda,X)$ satisfies condition (3.a) as well as (1,2,4) of Definition \ref{def:mfgsolutionwithoutcommonnoise}. Condition (5) of Definition \ref{def:weakMFGsolution} cleary implies condition (5'.b). Finally $\mu = P((W,\Lambda,X) \in \cdot \ | \ B,\mu)$ implies $\mu = P((W,\Lambda,X) \in \cdot \ | \ \mu)$, which verifies the final condition (6) of Definition \ref{def:mfgsolutionwithoutcommonnoise}. Hence $(\widetilde{\Omega},(\F_t)_{t \in [0,T]},P,W,\mu,\Lambda,X)$ is a weak MFG solution without common noise.

Conversely, let $(\widetilde{\Omega},(\F_t)_{t \in [0,T]},P,W,\mu,\Lambda,X)$ be a weak MFG solution without common noise, and assume without loss of generality that $(\widetilde{\Omega},(\F_t)_{t \in [0,T]},P)$ supports an $(\F_t)_{t \in [0,T]}$-Wiener process $B$ of dimension $m_0$ which is independent of $(W,\mu,\Lambda,X)$. Again, condition (3.a) as well as (1), (2), and (4) of Definition \ref{def:weakMFGsolution} clearly hold. The consistency condition $\mu = P((W,\Lambda,X) \in \cdot \ | \ \mu)$ and the independence of $B$ and $(W,\mu,\Lambda,X)$ imply $\mu = P((W,\Lambda,X) \in \cdot \ | \ B,\mu)$. Finally, to check (5'.a), note first that the independence of $B$ and $(X_0,W,\mu)$ implies easily that $\F^{X_0,B,W,\mu}_t$ and $\F^{X_0,W,\mu}_T$ are conditionally independent given $\F^{X_0,W,\mu}_t$. Thus, if $(\Lambda'_t)_{t \in [0,T]}$ is $(\F^{X_0,B,W,\mu}_t)_{t \in [0,T]}$-progressively measurable, then $\sigma(\Lambda'_s : s \le t)$ is conditionally independent of $\F^{X_0,W,\mu}_T$ given $\F^{X_0,W,\mu}_t$, and condition (5) of Definition \ref{def:mfgsolutionwithoutcommonnoise} implies that $\E^P[\Gamma(\mu^x,\Lambda,X)] \ge \E^P[\Gamma(\mu^x,\Lambda',X')]$, where $X'$ is defined as in \eqref{def:X'SDE}. This verifies (5'.a), and so $(\widetilde{\Omega},(\F_t)_{t \in [0,T]},P,B,W,\mu,\Lambda,X)$ is a weak MFG solution.
\end{proof}

\begin{proof}[Proof of Theorem \ref{th:withoutcommonnoise}]
At this point, the proof is mostly straightforward. The first claim, regarding the adaptation of Theorem \ref{th:mainconvergence}, follows immediately from Theorem \ref{th:mainconvergence} and the observation of Lemma \ref{le:with-withoutcommonnoise}. The second claim, about adapting Theorem \ref{th:converseconvergence}, is not so immediate but requires nothing new. First, notice that Theorem \ref{th:partialconverseconvergence} remains true if we replace ``weak MFG solution'' by ``weak MFG solution without common noise,'' and if we define $P_n$ instead by \eqref{def:pnwithoutcommonnoise}; this is a consequence of Theorem \ref{th:partialconverseconvergence} and Lemma \ref{le:with-withoutcommonnoise}. Then, we must only check that Proposition \ref{pr:fixednapproximation} remains true if we replace ``strong'' by ``very strong,'' and if we replace the conclusion (1) by
\begin{enumerate}
\item[(1')] In $\P^p((\C^m)^n \times \V^n \times (\C^d)^n)$
\[
\lim_{k\rightarrow\infty}\PP_n \circ \left(W,\Lambda^k,X[\Lambda^k]\right)^{-1} = \PP_n \circ \left(W,\Lambda^0,X[\Lambda^0]\right)^{-1}.
\]
\end{enumerate}
It is straightforward to check that the proof of Proposition \ref{pr:fixednapproximation} given in Section \ref{se:fixednapproximation} translates mutatis mutandis to this new setting.
\end{proof}

\appendix

\section{Proof of Propositions \ref{pr:equilibriuminclusions} and \ref{pr:verystrongeq}} \label{se:appendixproofs1}

\subsection{Proof of Proposition \ref{pr:equilibriuminclusions}} \label{se:proof-equilibriuminclusions}

\subsubsection*{Step 1:}
We first show that every strong $\epsilon$-Nash equilibrium is also a relaxed $\epsilon$-Nash equilibrium.
Suppose $\Lambda = (\Lambda^1,\ldots,\Lambda^n) \in \A_n^n(\mathcal{E}_n)$ is a strong $\epsilon$-Nash equilibrium on $\mathcal{E}_n$. Lemma \ref{le:valuebound}(3) implies
\[
\E^{\PP_n}\int_0^T\int_A|a|^{p'}\Lambda^i_t(da)dt < \infty, \ i=1,\ldots,n.
\]
Let $\delta > 0$, and find $\beta^* \in \A_n(\mathcal{E}_n)$ such that
\begin{align}
J_i((\Lambda^{-i},\beta^*)) \ge \sup_{\beta \in \A_n(\mathcal{E}_n)}J_i((\Lambda^{-i},\beta)) - \delta. \label{pf:equilibriuminclusions1}
\end{align}
Lemma \ref{le:valuebound}(2) implies $\E^{\PP_n}\int_0^T\int_A|a|^{p'}\beta^*_t(da)dt < \infty$.
Thus, by Lemma \ref{le:adapteddensity} (see also Remark \ref{re:adapteddensity}), we may find a sequence of $(\F^{s,n}_t)_{t \in [0,T]}$-progressively measurable $A$-valued processes $(\alpha^k_t)_{t \in [0,T]}$ such that
\begin{align}
\lim_{r\rightarrow\infty}\sup_k\E^{\PP_n}\int_0^T|\alpha^k_t|^{p'}1_{\{|\alpha^k_t| > r\}}dt = 0, \label{pf:equilibriuminclusions2}
\end{align}
and
\[
\PP_n \circ \left(\xi,B,W,\beta^*\right)^{-1} = \lim_{k \rightarrow\infty}\PP_n \circ \left(\xi,B,W,dt\delta_{\alpha^k_t}(da)\right)^{-1},
\]
in $\P^p((\R^d)^n \times \C^{m_0} \times (\C^m)^n \times \V)$. Abbreviate $\beta^k = dt\delta_{\alpha^k_t}(da)$. Since $\Lambda$ is a strong strategy, we may write $\Lambda = \widehat{\Lambda}(\xi,B,W)$ for some measurable function $\widehat{\Lambda}$, and it follows (e.g. from \cite[Lemma A.3]{lacker-mfgcontrolledmartingaleproblems}, which deals with the potential discontinuity of $\widehat{\Lambda}$) that
\begin{align}
\PP_n \circ \left(\xi,B,W,(\Lambda^{-i},\beta^*)\right)^{-1} = \lim_{k \rightarrow\infty}\PP_n \circ \left(\xi,B,W,(\Lambda^{-i},\beta^k)\right)^{-1}, \label{pf:equilibriuminclusions3}
\end{align}
Lemma \ref{le:finitestateconvergence} gives
\[
\PP_n \circ \left(\xi,B,W,(\Lambda^{-i},\beta^*),X[(\Lambda^{-i},\beta^*)]\right)^{-1} = \lim_{k \rightarrow\infty}\PP_n \circ \left(\xi,B,W,(\Lambda^{-i},\beta^k),X[(\Lambda^{-i},\beta^k)]\right)^{-1}.
\]
Hence, the uniform integrability \eqref{pf:equilibriuminclusions2} and continuity of $J$ of Lemma \ref{le:jcontinuous} imply
\begin{align}
\lim_{k\rightarrow\infty}J_i((\Lambda^{-i},\beta^k)) = J_i((\Lambda^{-i},\beta^*)). \label{pf:equilibriuminclusions3.5}
\end{align}
Finally, since $\Lambda$ is a strong $\epsilon$-Nash equilibrium, it holds for each $k$ that
\begin{align*}
J_i(\Lambda)  + \epsilon_i &\ge \sup_{\beta \in \A_n(\mathcal{E}_n) \text{ strong}}J_i\left((\Lambda^{-i},\beta)\right) \ge J_i\left((\Lambda^{-i},\beta^k)\right).
\end{align*}
Thus, sending $k \rightarrow \infty$ and applying \eqref{pf:equilibriuminclusions1} yields
\[
J_i(\Lambda)  + \epsilon_i \ge J_i\left((\Lambda^{-i},\beta^*)\right) \ge \sup_{\beta \in \A_n(\mathcal{E}_n)}J_i((\Lambda^{-i},\beta)) - \delta.
\]
Sending $\delta \downarrow 0$ shows that $\Lambda$ is in fact a relaxed $\epsilon$-Nash equilibrium. \hfill \qedsymbol

\subsubsection*{Step 2:}
The proof that every strict $\epsilon$-Nash is a relaxed $\epsilon$-Nash equilibrium follows the same structure; the only difference is that we construct the sequence $\alpha^k$ from $\beta^*$ a bit differently. First, let $\iota_k : A \rightarrow A$ be a measurable function satisfying $\iota_k(a) = a$ for $|a| \le k$ and $|\iota_k(a)| \le k$ for all $a \in A$. Let $\widetilde{\beta}^k_t := \beta^*_t \circ \iota_k^{-1}$, so that $\widetilde{\beta}^k \rightarrow \beta^*$ a.s., and clearly
\begin{align}
\int_{\{|a| > r\}}|a|^{p'}\widetilde{\beta}^k_t(da) \le \int_{\{|a| > r\}}|a|^{p'}\beta^*_t(da), \ r > 0. \label{pf:equilibriuminclusions4}
\end{align}
For each $k$, apply the well-known Chattering Lemma \cite[Theorem 2.2(b)]{elkaroui-partialobservations} to find a sequence of $(\F^{n}_t)_{t \in [0,T]}$-progressively measurable $A$-valued processes $\alpha^{k,j}_t$ such that
\begin{align}
\widetilde{\beta}^k = \lim_{j \rightarrow\infty}dt\delta_{\alpha^{k,j}_t}(da), \ a.s. \label{pf:equilibriuminclusions5}
\end{align}
We then find a subsequence $j_k$ such that $\beta^k := dt\delta_{\alpha^{k,j_k}_t}(da)$ converges a.s. to $\beta^*$, and \eqref{pf:equilibriuminclusions3} holds. It follows also from \eqref{pf:equilibriuminclusions4} and \eqref{pf:equilibriuminclusions5} that
\[
\lim_{r\rightarrow\infty}\sup_{k,j}\E^{\PP_n}\int_0^T|\alpha^{k,j}_t|^{p'}1_{\{|\alpha^{k,j}_t| > r\}}dt = 0,
\]
so that \eqref{pf:equilibriuminclusions3.5} holds as well. The rest of the proof is as in Step 1.

\subsection{Proof of Proposition \ref{pr:verystrongeq}} \label{se:proof-verystrongeq}

First, note that when $\sigma_0 \equiv 0$, Lemma \ref{le:finitestateconvergence} holds true when the common noise $B$ is omitted everywhere it appears. With this in mind, the proof of Proposition \ref{pr:verystrongeq} follows exactly Step 1 of the proof of Proposition \ref{pr:equilibriuminclusions}, except of course with the word ``strong'' replaced by ``very strong,'' and with the common noise $B$ removed everywhere it appears.

\section{Proof of Proposition \ref{pr:itocompact}} \label{ap:itocompact}
This proof is similar to the proofs of \cite[Proposition B.2]{carmonadelaruelacker-mfgcommonnoise} and \cite[Proposition B.4]{lacker-mfgcontrolledmartingaleproblems}.
For each $1 \le i \le n$ and $P \in \Q_{\kappa_{n,i}}$, apply the Burkholder-Davis-Gundy inequality and the growth assumption to find a constant $C > 0$ (which will change from line to line but depends only on $c$, $T$, and $p'$) such that
\begin{align*}
\E^P[\|X\|^{p'}_t] \le &\,C\E^P\left[ |X_0|^{p'} + \left(\int_0^t\int_A|B(s,a)|\Lambda_s(da)ds\right)^{p'} + \left(\int_0^t\left|\Sigma\Sigma^\top(s)\right|ds\right)^{p'/2}\right]  \\
	\le &\,C\,\E\left\{1 + |X_0|^{p'} + Z^{p'} + \int_0^t\left(\|X\|^{p'}_s + \int_A|a|^{p'}\Lambda_s(da)\right) ds \right\},
\end{align*}
where we used also $p' \ge 2$ and Jensen's inequality. By Gronwall's inequality,
\begin{align*}
\E^P[\|X\|^{p'}_T] \le C\E^P\left[1 + |X_0|^{p'} + Z^{p'} + \int_0^T\int_A|a|^{p'}\Lambda_t(da)dt\right] \le C(1 + \kappa_{n,i}).
\end{align*}
Thus
\begin{align}
\sup_{P \in \Q}\E^P[\|X\|^{p'}_T] &= \sup_n\sup\left\{\frac{1}{n}\sum_{i=1}^n\E^{P_i}[\|X\|^{p'}_T] : P_i \in \Q_{\kappa_{n,i}} \text{ for } i=1,\ldots,n\right\} \nonumber \\
	&\le C\sup_n\frac{1}{n}\sum_{i=1}^n(1+\kappa_{n,i}) < \infty. \label{pf:itocompact1}
\end{align}
By assumption, we have also
\begin{align}
\sup_{P \in \Q}\E^P\int_0^T\int_A|a|^{p'}\Lambda_t(da)dt \le \sup_n\frac{1}{n}\sum_{i=1}^n\kappa_{n,i} < \infty. \label{pf:itocompact3}
\end{align}
In light of \eqref{pf:itocompact1} and \eqref{pf:itocompact3}, it suffices to show that $\{P \circ X^{-1} : P \in \Q\} \subset \P(\C^d)$ is tight; see \cite[Proposition B.3]{lacker-mfgcontrolledmartingaleproblems}. To check this, we will verify Aldous' criterion \cite[Lemma 16.12]{kallenberg-foundations} for tightness, or
\begin{align}
\lim_{\delta\downarrow 0}\sup_{P \in \Q}\sup_\tau \E^P[|X_{(\tau + \delta) \wedge T} - X_\tau|^p] = 0, \label{pf:itocompact2}
\end{align}
where the supremum is over stopping times $\tau$ valued in $[0,T]$.
The Burkholder-Davis-Gundy inequality implies that there exists a constant $C' > 0$ (which again depends only on $c$, $T$, and $p$ and will change from line to line) such that, for any $i$ and any $P \in \Q_{\kappa_{n,i}}$,
\begin{align*}
\E^P[|X_{(\tau + \delta) \wedge T} - X_\tau|^p] &\le C'\E^P\left[\left|\int_\tau^{(\tau + \delta) \wedge T}\int_AB(t,a)\Lambda_t(da)dt\right|^p + \left|\int_\tau^{(\tau + \delta) \wedge T}dt|\Sigma(t)|^2\right|^{p/2} \right] \\
	&\le C'\E^P\left[\left|c\int_\tau^{(\tau + \delta) \wedge T}\left(1 + |X_t|	+ Z + \int_A|a|\Lambda_t(da)\right)dt\right|^p\right] \\
	&\quad + C'\E^P\left[\left|c\int_\tau^{(\tau + \delta) \wedge T}\left(1 + |X_t|^{p_\sigma} + Z^{p_\sigma}\right)dt\right|^{p/2}\right] \\
	&\le C'\E^P\left[(\delta^p + \delta^{p/2})\left(1 + \|X\|_T^p + Z^p\right) + \left|\int_\tau^{(\tau + \delta) \wedge T}\int_A|a|\Lambda_t(da)dt\right|^p\right].
\end{align*}
Since $p' > p$, we have $\E^P[Z^p] \le \E^P[Z^{p'}]^{p/p'} \le \kappa_{n,i}^{p/p'}$ for $P \in \Q_{\kappa_{n,i}}$, and thus by assumption
\begin{align*}
\sup_{P \in \Q}\E^P[Z^p] &= \sup_n\sup\left\{\frac{1}{n}\sum_{i=1}^n\E^{P_i}[Z^p] : P_i \in \Q_{\kappa_{n,i}} \text{ for } i=1,\ldots,n\right\} \\
	&\le \sup_n\frac{1}{n}\sum_{i=1}^n\kappa_{n,i}^{p/p'} < \infty.
\end{align*}
This and \eqref{pf:itocompact1} imply
\[
\lim_{\delta \downarrow 0}\sup_{P \in \Q}(\delta^p + \delta^{p/2})\E^P\left[1 + \|X\|_T^p + Z^p\right] = 0.
\]
To control the term with $\Lambda$, note that $p' > p$ and \eqref{pf:itocompact3} imply
\[
\lim_{\delta \downarrow 0}\sup_{P \in \Q}\sup_\tau\E^P\int_\tau^{(\tau + \delta) \wedge T}\int_A|a|^p\Lambda_t(da)dt = 0.
\]
Putting this all together proves \eqref{pf:itocompact2}. \hfill \qedsymbol

\bibliographystyle{amsplain}
\bibliography{MFGconvergence-bib}

\end{document}